\documentclass{amsart}
\usepackage{amsmath}
\usepackage{graphicx}
\usepackage{latexsym}
\usepackage{color}
\usepackage{amscd}
\usepackage[all]{xy}
\usepackage{enumerate}
\usepackage{hyperref}
\usepackage{subfigure}
\theoremstyle{plain}
\newtheorem{theorem}{Theorem}[section]
\newtheorem{corollary}[theorem]{Corollary}
\newtheorem{lemma}[theorem]{Lemma}
\newtheorem{proposition}[theorem]{Proposition}

\theoremstyle{definition}
\newtheorem{definition}[theorem]{Definition}
\newtheorem{example}[theorem]{Example}

\newtheorem{remark}[theorem]{Remark}
\newtheorem{question}[theorem]{Question}



\newcommand{\CPb}{\overline{\mathbb{CP}}{}^{2}}
\newcommand{\CP}{{\mathbb{CP}}{}^{2}}

\newcommand{\Z}{\mathbb{Z}}

\newcommand{\K}{{\rm K3}}

\def\Ker{\operatorname{Ker}}
\def\Int{\operatorname{Int}}
\def\Diff{\operatorname{Diff}}

\def\Re{\operatorname{Re}}
\def\id{\operatorname{id}}
\def\Mod{\operatorname{Mod}}
\def\Crit{\operatorname{Crit}}
\def\Push{\operatorname{Push}}

\def\Sign{\operatorname{Sign}}

\def \x {\times}
\def \eu{{\text{e}}}

\begin{document}

\title[Multisections and symplectic $4$-manifolds]
{Multisections of Lefschetz fibrations and topology of symplectic $4$-manifolds}

\author[R. \.{I}. Baykur]{R. \.{I}nan\c{c} Baykur}
\address{Department of Mathematics and Statistics, University of Massachusetts, Amherst, MA 01003-9305, USA}
\email{baykur@math.umass.edu}

\author[K. Hayano]{Kenta Hayano}
\address{Department of Mathematics, Graduate School of Science, Hokkaido University, Sapporo, Hokkaido 060-0810, Japan}
\email{k-hayano@math.sci.hokudai.ac.jp}

\begin{abstract}
We initiate a study of positive multisections of Lefschetz fibrations via positive factorizations in framed mapping class groups of surfaces. Using our methods, one can effectively capture various interesting symplectic surfaces in symplectic $4$-manifolds as multisections, such as Seiberg-Witten basic classes and exceptional classes, or branched loci of compact Stein surfaces as branched coverings of the $4$-ball. Various problems regarding the topology of symplectic $4$-manifolds, such as the smooth classification of symplectic Calabi-Yau $4$-manifolds, can be translated to combinatorial problems in this manner. After producing special monodromy factorizations of Lefschetz pencils on symplectic Calabi-Yau $\K$ and Enriques surfaces, and introducing monodromy substitutions tailored for generating multisections, we obtain several novel applications, allowing us to construct: new counter-examples to Stipsicz's conjecture on fiber sum indecomposable Lefschetz fibrations, non-isomorphic Lefschetz pencils of the same genera on the same new symplectic $4$-manifolds, the very first examples of exotic Lefschetz \textit{pencils}, and new exotic embeddings of surfaces.
\end{abstract}

\maketitle

\tableofcontents

\section{Introduction} 

Since the groundbreaking work of Donaldson it is known that every symplectic $4$-manifold admits a symplectic Lefschetz pencil \cite{Donaldson}, and conversely, every Lefschetz fibration with non-empty critical locus admits a symplectic structure \cite{GS}. On the other hand, Lefschetz pencils\,/\,fibrations are determined by their monodromy factorizations, which are prescribed by products of positive Dehn twists isotopic to identity\,/\,boundary multitwist on the fiber \cite{Kas_1980, Matsumoto3}. These results yield a combinatorial description of symplectic $4$-manifolds in terms of ordered tuples of isotopy classes of simple closed curves on an orientable surface. Here we will extend this fundamental approach, by introducing and studying positive factorizations in a \textit{framed mapping class group}, so as to describe symplectic $4$-manifolds together with various important symplectic surfaces in them in terms of ordered simple closed curves and arcs between marked points on an orientable surface.

Let $X$ be a closed oriented $4$-manifold equipped with a Lefschetz fibration \linebreak $f\colon X \to S^2$. We call an embedded, possibly disconnected surface $S$ in $X$ a \emph{multisection} or \emph{$n$-section} if $f|_S \colon X \to S^2$ is an $n$-fold branched cover with only simple branched points. We assume that both Lefschetz critical points and branched points conform to local complex models, that is, we work with \textit{positive} Lefschetz fibrations and \textit{positive} branched points. Precise definitions and the basic background material are given in Section~\ref{Preliminaries} below.

The first main result of our article is the description of multisections and their ambient topology via positive factorizations in framed mapping class groups, given in detail in Theorem~\ref{mainthm}. The framings amount to working with a new mapping class group of a compact oriented surface with marked boundary circles (one marked point on each boundary component), which consists of isotopy classes of orientation-preserving self-diffeomorphisms that are allowed to swap boundary components while matching the marked points. This group is naturally isomorphic to the mapping class group of a \textit{closed} surface with attached vectors at a finite \textit{set} of marked points, so as to frame a tubular neighborhood of the multisection $S$ (where the end points of these vectors, and equivalently the marked points on the boundaries trace a push-off of $S$). In Section~\ref{MCG}, leading to the proof of this theorem, we introduce the notion of positivity for monodromy factorizations in this more general setting. As we will show, from these positive factorizations, for each multisection $S$, one can easily read off the degree (i.e. the number of times $S$ intersects the fibers), topology (number of components and genera of each component of $S$) and the self-intersection numbers of the components of $S$. Here is our main theorem of Section~\ref{MCG}, stated for a connected $S$ for simplicity ---while the reader might want to turn to that section for the description of various mapping classes appearing in the below factorizations: 

\begin{theorem} \label{firstquote}
A genus-$g$ Lefschetz fibration $(X,f)$ with a self-intersection $m$ connected $n$-section $S \subset X$ with $k$ branched points away from $\Crit(f)$, and $r$ branched points at Lefschetz singularities corresponding to vanishing cycles $c_1,\ldots, c_r$, among $c_1, \ldots, c_l$ yields to a lift of the monodromy factorization
\begin{equation*}
\tilde{\tau}_{\alpha_k} \cdots \tilde{\tau}_{\alpha_1} \cdot t_{\widetilde{c_{l}}} \cdots t_{\widetilde{c_{r+1}}}\cdot \widetilde{t_{c_{r}}}\cdots \widetilde{t_{c_{1}}} = t_{\delta_1}^{a_1}\cdots t_{\delta_n}^{a_n}, 
\end{equation*}
in $\Mod(\Sigma_{g}^{n}; \{u_1,\ldots, u_n\})$, where $\{u_1,\ldots, u_n\}$ is a subset of $\partial \Sigma_{g}^{n}$ which covers all the elements of $\pi_0(\partial \Sigma_{g}^{n})$, $\tilde{\tau}_{\alpha_i}$ is a lift of half twist as described in Figure~\ref{lift_monodromy1},  and $\widetilde{t_{c_i}}$ is a lift of the Dehn twist $t_{c_i}$ as described in Figure~\ref{lift_monodromy2}, such that all $u_j$ are hit by the collection of $\alpha_i$ and $c_i$. On the other hand, $\widetilde{c_j}$ is a simple closed curve in $\Sigma_g^n$ which is isotopic to $c_j$ via the inclusion $i:\Sigma_g^n\hookrightarrow \Sigma_g$, and $\{\delta_1, \ldots, \delta_n\}$ is a set of simple closed curves parallel to $\partial \Sigma_{g}^{n}$, where 
\[g(S)= \frac{1}{2}(k+r) - n + 1 \, \ \text{and} \ m=-(\Sigma_{i=1}^n a_i) + 2k+ r \, .\]
Conversely, from any such relation in $\Mod(\Sigma_{g}^{n}; \{u_1,\ldots, u_n\})$, subject to the conditions listed above, one can construct a genus-$g$ Lefschetz fibration $(X,f)$ with a connected $n$-section $S$ of genus $g(S)$ and self-intersection $m$ as above, whose monodromy factorization is given by the image of the factorization on the left hand side under the homomorphism forgetting $\{u_1,\ldots, u_n\}$.
\end{theorem}

\noindent Moreover, up to an extended set of Hurwitz moves, there exists a one-to-one correspondence between Lefschetz fibrations with multisections and positive factorizations in the above framed mapping class group.\cite{BH2}

As observed by Donaldson and Smith \cite{Donaldson_Smith_2003}, any, possibly disconnected symplectic surface $S$ in $(X, \omega)$ can be indeed realized as a multisection of a high enough degree Lefschetz pencil on $X$, which does not go through any Lefschetz critical points. Our theorem therefore extends the combinatorial interpretation of a symplectic \linebreak $4$-manifold, which couples the results of Donaldson and Gompf with the earlier works of Kas and Matsumoto, to that of a symplectic $4$-manifold and disjoint symplectic surfaces in it in terms of  ordered tuples of interior curves $\widetilde{c_1}, \ldots, \widetilde{c_l}$ and arcs $\alpha_1, \ldots, \alpha_r$ with end points on marked points $u_1, \ldots, u_n$ on distinct boundary components of $\Sigma_g^n$ (corresponding to the factors $t_{\widetilde{c_i}}$, and $\tilde{\tau_{\alpha_j}}$, respectively). 
On the other hand, it was shown by Loi and Piergallini \cite{LP} that any compact Stein surface $(X,J)$ can be obtained as a covering of the unit $4$-ball $D^4$ branched along a \textit{braided surface} $S$, such that the composition $f: X \to D^4 \to D^2$ is an allowable Lefschetz fibration, along with the obvious converse result. 
In this case, we obtain a similar combinatorial description of a Stein surface $(X,J)$ together with the branched locus $S$ in terms of \textit{pairs} of arcs $\alpha'_1, \alpha''_1, \ldots, \alpha'_r$, each with end points on the marked points on the same pair of distinct boundary components of $\Sigma_g^n$ (corresponding to the factors $\widetilde{t_{c_i}}$). 
Although we note that the latter can be always perturbed (Remark~\ref{rem_perturbation}) to a factorization which only consists of factors $t_{\widetilde{c_i}}$ 
and $\tilde{\tau_{\alpha_j}}$, we present our results in this full generality so as to not only mark the case of compact Stein surfaces, but also because we often find it useful to first produce factorizations containing multisections going through Lefschetz critical points. 

We should make two remarks here. First, although the framed mapping class group $\Mod(\Sigma_{g}^{n}; \{u_1,\ldots, u_n\})$ has a large set of generators, which for instance involve boundary pushing maps (lifts of point pushing maps), what we have manifested in Theorem~\ref{firstquote} is that for the geometric situations discussed above, it suffices to work with boundary twists and usual Dehn twists. Secondly, our framed mapping class group is larger than the mapping class group of a surface with boundary which consists of isotopy classes of self-diffeomorphisms that fix \textit{each boundary component}, the latter being the natural mapping class group to work with when dealing with \textit{$n$ disjoint sections}. The distinction between these two groups is analogous to that of the framed surface braid group versus the framed pure surface braid group, which have appeared in two recent works that are worth mentioning here: Bellingeri and Gervais studied the exact sequences relating these braid groups \cite{BG}, whereas Massuyeau, Oancea and Salamon used the same groups to describe the monodromy action of the fundamental group on the first homology of the fiber in terms of the Picard-Lefschetz intersection data associated to vanishing cycles of a given Lefschetz fibration \cite{MOS}. 


Combining the seminal work of Taubes \cite{T, T2} and Donaldson \cite{Donaldson}, and following the ideas of Donaldson and Smith \cite{Donaldson_Smith_2003} mentioned above, a blow-up of any given symplectic $4$-manifold $X$ with $b^+(X) > 1$ admits a Lefschetz fibration with respect to which all Seiberg-Witten basic classes are multisections (called \emph{standard surfaces} in \cite{Donaldson_Smith_2003}), as discussed in the Appendix~\ref{SW}. Translating this to positive mapping class group factorizations as we prescribed in Section~\ref{MCG}, we conclude that symplectic $4$-manifolds and their Seiberg-Witten basic classes can be \emph{a priori} represented combinatorially in terms of our positive factorizations. Section~\ref{genus3counter} contains many examples of Kodaira dimension zero symplectic $4$-manifolds, where all the Seiberg-Witten basic classes are represented by a collection of $(-1)$-multisections  (as dictated by the blow-up formula for Seiberg-Witten invariants) of the constructed Lefschetz fibrations on them. In the Appendix~\ref{knotsurgery}, we present Kodaira dimension $1$ examples; namely, we carry out a sample calculation of monodromy factorizations of Lefschetz fibrations on the knot surgered elliptic surfaces which capture all their Seiberg-Witten basic classes. 

The remaining Sections~$4$--$7$ of the article gather a variety of applications, relying on the constructive converse direction of our main theorem. Each section focuses on a different problem related to the topology of symplectic $4$-manifolds and Lefschetz fibrations on them, yet what is in common for all is the essential use of our mapping class group techniques involving multisections. The novelty of ideas and techniques employed in our constructions of examples in Section~$5$--$7$ can easily be seen to amount to \textit{recipes} one can employ with the right set of of Lefschetz fibrations and multisections in hand, where we will be focusing on producing the \textit{smallest} genus examples of each kind, which are the hardest to obtain in our experience. 

In Section~\ref{SCY} we provide an alternate approach to the \textit{smooth} classification of symplectic $4$-manifolds of Kodaira dimension zero, i.e. (blow-ups) of symplectic Calabi-Yau $4$-manifolds. The only known examples of Kodaira dimension zero symplectic $4$-manifolds are torus bundles over tori, the $\K$ and the Enriques surfaces, which, up to diffeomorphisms, conjecturally exhaust all the possibilities. Using our work from Section~\ref{MCG} and the following theorem we prove in Section~\ref{SCY} (Theorem~\ref{SCYLF}), we translate the problem to a combinatorial one (Theorem~\ref{SCYLF}, Corollary~\ref{SCYcor} and Question~\ref{SCYconjecture}):

\begin{theorem}
Let $(X,f)$ be a genus-$g$ Lefschetz fibration with $g \geq 2$, and $X$ be neither rational nor ruled. Then, there exists a symplectic form $\omega$ on $X$ compatible with $f$ such that $(X,\omega)$ is a (blow-up of) a symplectic Calabi-Yau $4$-manifold, if and only if there is a disjoint collection of $(-1)$-spheres that are $n_j$-sections of $(X,f)$ such that $\sum_j n_j = 2g-2$.
\end{theorem}

\noindent Motivated by this, we introduce new techniques based on certain symmetries, to lift better understood relations from genus $0$ and $1$ surfaces to higher genus surfaces under involutions, so as to construct explicit monodromy factorizations of Lefschetz pencils on symplectic Calabi-Yau $\K$ and Enriques surfaces, i.e. minimal symplectic $4$-manifolds of Kodaira dimension zero, homeomorphic to $\K$ and Enriques surfaces, respectively; see Propositions~\ref{lem_signature_genus3LF} and~\ref{lem_pi1_genus2LF}.

In Section~\ref{Stipsicz} we turn to an interesting conjecture of Stipsicz on fiber sum indecomposable Lefschetz fibrations, which can be regarded as prime building blocks of any Lefschetz fibration via the fiber sum operation. In \cite{Stipsicz}, having proved the converse statement, Stipsicz conjectured that any fiber sum indecomposable Lefschetz fibration admits a $(-1)$-sphere section, an affirmative answer to which would allow one to think of any Lefschetz fibration to be obtained from Lefschetz pencils through blow-ups and fiber sums. Curiously, up to date, there was only one known counter-example to this conjecture, which was a genus-$2$ Lefschetz fibration constructed by Auroux, as observed by Sato in \cite{Sato_2008}. In Lemma~\ref{lem_lantern_relation}, we introduce a generalization of the lantern relation involving multisections, which allows us to braid exceptional sections into exceptional multisections of a new Lefschetz fibration obtained by a rational blow-down of the underlying symplectic $4$-manifold. Relying on this key lemma, and our special monodromy factorizations of symplectic Calabi-Yau Lefschetz fibrations obtained in Section~\ref{SCY}, where one can keep track of \textit{all} exceptional classes and sections, we prove that the above counter-example is not a mere exception (Theorems~\ref{thm_example_genus3} and \ref{thm_example_genus2}):
\begin{theorem}
There are several genus $3$ and genus $2$ fiber sum indecomposable Lefschetz fibrations on blow-ups of symplectic Calabi-Yau $4$-manifolds, which do not admit any $(-1)$ sphere sections. 
\end{theorem}

Section~\ref{Nonisomorphic} deals with the \textit{diversity} of Lefschetz pencils/fibrations on a symplectic $4$-manifold. Namely, we prove that blow-ups of symplectic Calabi-Yau $\K$ surfaces can be supported by non-isomorphic Lefschetz pencils of the same genera and same number of base points, which have ambiently homeomorphic fibers. Park and Yun used monodromy groups to construct pairs of non-isomorphic Lefschetz fibrations on knot surgered elliptic surfaces, which are Kodaira dimension $1$ symplectic $4$-manifolds \cite{ParkYun}, and more recently, the first author proved that blow-ups of any symplectic $4$-manifold which is not rational or ruled carry arbitrarily large number of non-isomorphic Lefschetz fibrations \cite{BaykurLuttingerLF}. Here we show that for a certain configuration of Lefschetz vanishing cycles and $(-1)$-sphere sections, one can perform a pair of monodromy substitutions which amount to rational-blowdowns that ``mirror'' each other's topological effect. These result in Lefschetz fibrations on the \textit{same} symplectic $4$-manifold with ambiently homeomorphic fibers. Building on our examples of monodromy factorizations tailored specifically to contain such configurations, we then obtain the following on symplectic $4$-manifolds of Kodaira dimension $0$ (Theorem~\ref{nonisomLFs}): 

\begin{theorem} 
There are pairs of genus-$g$ relatively minimal non-isomorphic Lefschetz pencils $(X, f_i)$, $i=1,2$, where $g$ can be taken as small as $3$, or arbitrarily large. 
\end{theorem}

In Section~\ref{Exotic} we investigate a natural question: does the \textit{topology} of a Lefschetz pencil (fiber genus, number of separating/non-separating vanishing cycles and base points) uniquely determine the diffeomorphism type of a symplectic $4$-manifold within its homeomorphism class? Here we answer this question in the negative by constructing the first examples of pairwise homeomorphic but not diffeomorphic symplectic $4$-manifolds, supported by Lefschetz pencils with the same topology. To the best of our knowledge, the only previously known examples of this type were the Lefschetz \textit{fibrations} of Fintushel and Stern on knot surgered elliptic surfaces --all of Kodaira dimension $1$, again. Here we construct the first examples of such pencils (Theorem~\ref{ExoticLP}):

\begin{theorem} 
There are genus-$3$ exotic Lefschetz pencils $(X_i,f_i)$, $i=0,1$, with symplectic Kodaira dimension $\kappa(X_i)=i$, where $X_i$ are homeomorphic to $\K \# \CPb$. Moreover, there are similar examples with arbitrarily high genus and the same topology for the singular fibers on higher blow-ups of homotopy $\K$s.
\end{theorem}

Lastly, in the same section, we also show that a careful application of the same circle of ideas provide a new way of constructing exotic embeddings of surfaces in $4$-manifolds, i.e. $F_i$ in $X$, $i=1,2$, such that there are ambient homeomorphisms taking one to the other, but there exist no such ambient diffeomorphisms (Theorem~\ref{exoticknotting}):

\begin{theorem} 
There are exotic embeddings of genus-$3$ surfaces $F_i$ in a blow-up of a symplectic Calabi-Yau $\K$ surface such that $F_i$ is symplectic with respect to deformation equivalent symplectic forms $\omega_i$ on $X$, for $i=1,2$. 
\end{theorem}

\smallskip
We will finish with noting a further motivation for our study of multisections. The rather explicit description of a $4$-manifold obtained via monodromy of a Lefschetz fibration on it very often allows one to detect various configurations of symplectic surfaces in it; disjoint copies of fibers and sections, as well as matching pairs of Lefschetz vanishing cycles are a few examples of this sort. Coupled with the non-triviality of Seiberg-Witten invariants on symplectic $4$-manifolds, this has been the most essential source of producing new symplectic and smooth $4$-manifolds in the past few decades. (See for instance \cite{FS-survey} for an excellent survey of such construction methods.) A close look at these constructions shows that sections and multisections of such Lefschetz fibrations feature a key role. We therefore expect that the monodromy factorizations, which involve multisections to produce interesting configurations of surfaces (such as the ones we used in our rational-blowdowns in Sections~\ref{Nonisomorphic} and \ref{Exotic}), will be useful for building new symplectic and exotic $4$-manifolds. We plan to investigate this direction in future work.

\section{Preliminaries} \label{Preliminaries} 

In this article, we assume that all manifolds are compact, connected, smooth and oriented, and all the maps between them are smooth. 


\vspace{0.15cm}
\subsection{Lefschetz fibrations and multisections} \

Let $X$ and $\Sigma$ be compact manifolds (possibly with boundary) of dimensions $4$ and $2$, respectively. 

A smooth map $f: X\rightarrow \Sigma$ is a \emph{Lefschetz fibration} if $\Crit{f}$ is a discrete set in the interior of $X$ such that for any $p_i \in \Crit(f)$, we can take a complex coordinate $(U, \varphi)$ (resp. $(V, \psi)$) of $p_i$ (resp. $f(p_i)$) compatible with the orientation of $X$ (resp. of $\Sigma$) so that: 
\[
\psi \circ f \circ \varphi^{-1}(z_1, z_2) = z_1z_2. 
\] 
We furthermore assume that for each point $q_i \in C=f(\Crit(f))$, the \emph{singular fiber} $f^{-1}(q_i)$ contains exactly one critical point $p_i \in X$ of $f$. Any point $p_i \in \Crit{f}$ is called a \emph{Lefschetz singularity}, and for $g$ the genus of a regular fiber of $f$, $f: X\rightarrow \Sigma$ is called a \emph{genus-$g$ Lefschetz fibration}. Each critical point $p_i$ locally arises from shrinking a simple loop $c_i$ on $F$, called the \emph{vanishing cycle}. A singular fiber of a Lefschetz fibration is called \emph{reducible} (resp. \emph{irreducible}) if $c_i$ is separating (resp. nonseparating). In particular, if $c_i$ is null-homotopic in $F$, it gives rise to a \linebreak $(-1)$-sphere contained in the singular fiber, which can be blow-down preserving the rest of the fibration. We will always work with \textit{relatively minimal} Lefschetz fibrations, which do not contain any $(-1)$-spheres in the fibers. 

Given any fibration with only Lefschetz critical points, after a small perturbation one can always guarantee that there is at most one critical point on each fiber, as we built into our definition above. It shall be clear that $f$ restricts to a genus-$g$ surface bundle over $\Sigma \setminus C$. Lastly, an \emph{achiral Lefschetz fibration} is defined in the same way as above except that the local coordinate $(U,\varphi)$ is allowed to be incompatible with the orientation of $X$.  

Lefschetz fibrations arise naturally from pencils, where the domain $4$-manifold is closed and the target surface is $S^2$. A \emph{Lefschetz pencil} on a closed $4$-manifold $X$ is a Lefschetz fibration $f: X \setminus B \rightarrow S^2$, defined on the complement of a \textit{non-empty} discrete set $B$ in $X$, such that around any point $b_j \in B$, $f$ is locally modeled (again in a manner compatible with orientations) as $(z_1, z_2) \to z_1/ z_2$. Blowing-up all the points in $B$, one obtains an honest Lefschetz fibration $\tilde{f}: \tilde{X} \to S^2$ with $|D|$ distinct $(-1)$-sphere sections $S_j$, namely the exceptional spheres of the respective blow-ups. We will often use the short-hand notation $(X,f)$ for a Lefschetz fibration or pencil whenever $\Sigma = S^2$.

\begin{definition}\label{def_multisection}
Let $f:X\rightarrow \Sigma$ be a Lefschetz fibration and $S$ be an embedded surface in $X$. The surface $S$ is called a \emph{multisection} or \emph{$n$-section} of $f$ if it satisfies the following conditions:
\begin{enumerate}
\item $f|_{S}$ is an $n$-fold simple branched covering for some non-negative integer $n$; 
\item if a branched point $p \in S$ is not in $\Crit{f}$, the induced map \linebreak $df_p: N_pS \rightarrow T_{f(p)}\Sigma$ is orientation preserving isomorphism, where $N_pS$ is the fiber of the normal bundle of $S$ at $p$ which has the canonical orientation induced by that of $X$ and $S$; 
\item if a branched point $p\in S$ of $f|_{S}$ is in $\Crit(f)$, then there are complex coordinates $(U,\varphi)$ and $(V,\psi)$ as in the definition of a Lefschetz fibration above such that $\varphi(S\cap U)$ is equal to $\{(z, z) \in \mathbb{C}^2~|~z \in \mathbb{C}\}$. 
\end{enumerate}
\end{definition}

Clearly a $1$-section is an honest section of a Lefschetz fibration. Note that in both definitions we have given above, there is a \emph{positivity} imposed by requiring the compatibility with orientations in local complex models. In the language of \cite{Donaldson_Smith_2003} a multisection which is branched away from Lefschetz singularities is called a \emph{standard surface}. As it will become clear later, allowing our multisections to be branched at Lefschetz critical points as well (although subject to the local model given above), we will have a more flexible setting which makes is possible to deal with larger families of examples of Lefschetz fibrations with multisections of geometric significance. Lastly, as in the case of achiral Lefschetz fibrations, one can possibly work more generally with non necessarily positive multisections by allowing the local models to be incompatible with the orientations.

\vspace{0.15cm}
\subsection{Mapping class groups}  \

As it will become crucial in capturing the local topology of multisections (namely the self-intersections of them in the ambient $4$-manifold), we are going to set up mapping class groups relevant to our purposes in a framed fashion. 

Let $\Sigma$ be a compact, oriented and connected surface. In this paper, we regard $\Sigma$ as the zero-section of the tangent bundle $T\Sigma$. Take subsets $U_i, P \subset T\Sigma$. We define a group $\Mod_{P}{(\Sigma; U_1,\ldots,  U_n)}$ as follows: 
\[
	\Mod_{P}{(\Sigma; U_1,\ldots, U_n)} = \pi_0(\Diff^+_{P}{(\Sigma; U_1,\ldots, U_n)}), 
\]
where $\Diff^+_{P}{(\Sigma; U_1,\ldots, U_n)}$ is defined as follows: 
\[
	\Diff^+_{P}{(\Sigma; U_1,\ldots, U_n)}= \\
	\left\{T\in \Diff^+(\Sigma)~|~dT|_{P}= \id|_{P},  dT(U_i) = U_i ~\text{for all} i\right\}.
\]
Here we denote by $\Diff^+(\Sigma)$ the group of orientation preserving self-diffeomorphisms of $\Sigma$.  
For a simpler notation, we also define the groups $\Diff^+(\Sigma; U_1,\ldots, U_n)$ and \linebreak $\Mod(\Sigma; U_1, \ldots, U_n)$ as: 
{\allowdisplaybreaks
\begin{align*}
\Diff^+(\Sigma; U_1,\ldots, U_n) & = \Diff^+_{\emptyset}(\Sigma; U_1,\ldots, U_n), \\
\Mod(\Sigma; U_1, \ldots, U_n) & = \Mod_{\emptyset}(\Sigma; U_1, \ldots, U_n). 
\end{align*}
}
\noindent The group structures on all of the above are defined via compositions as maps, i.e. for $T_1,T_2\in\Diff^+_{P}(\Sigma; ,\ldots, U_n)$, we $T_1\cdot T_2 = T_1\circ T_2$, etc.

\vspace{0.15cm}
\subsection{Monodromy factorizations}  \

Let $h:X\rightarrow D^2$ be a genus-$g$ Lefschetz fibration and $C=\{p_1,\ldots, p_l\}\subset D^2$ the set of critical values of $h$. We take a regular value $q_0\in \Int(D^2)$ and an identification $\Sigma_g\cong h^{-1}(q_0)$. 
For each $i$ we also take a path $\gamma_i$ in $\Int(D^2)$ connecting $q_0$ with $q_i$ so that all $\gamma_i$'s are pairwise disjoint except at $q_0$. We give indices of these paths so that $\gamma_1,\ldots, \gamma_l$ appear in this order when we travel around $q_0$ counterclockwise. Let $a_i: S^1\rightarrow D^2\setminus C$ be a loop obtained by connecting a small circle around $q_i$ oriented counterclockwise using $\gamma_i$. The pullback $a_i^\ast h$ is a $\Sigma_g$-bundle over $S^1$ and we can obtain a self-diffeomorphism by taking a parallel transport of a flow in the total space of $a_i^\ast h$ transverse to each fiber.
  
Although a diffeomorphism depends on a choice of a flow, its isotopy class is uniquely determined from the $\Sigma_g$-bundle structure. The isotopy class is called a \emph{monodromy} of the bundle $a_i^\ast h$. 
Kas~\cite{Kas_1980} proved that the monodromy of $a_i^\ast h$ is the right-handed Dehn twist along some simple closed curve $c_i\subset \Sigma_g$, which is called a \emph{vanishing cycle} of the Lefschetz singularity $p_i$ in $h^{-1}(q_i)$. Let $a$ be a loop obtained by connecting $a_1,\ldots,a_l$ in this order. It is easy to verify that $a$ is homotopic to the boundary $\partial D^2$ in $D^2\setminus C$. The product $t_{c_l}\cdot \cdots \cdot t_{c_1}$ is the monodromy of the bundle $a^\ast h$. For a genus-$g$ Lefschetz fibration $f: X\rightarrow S^2$ over $S^2$, we take a disk $D\subset S^2$ so that $D$ contains all the critical values of $f$. The restriction $f|_{f^{-1}(D)}$ is a Lefschetz fibration over the disk. Since the monodromy of $f|_{f^{-1}(\partial D)}$ is trivial, we can obtain the following factorization of the unit element of the mapping class group $\Mod(\Sigma_g)$: 
\[
t_{c_l} \cdot \cdots \cdot t_{c_1} = 1,
\] 
where $c_i \subset \Sigma_g$ is a vanishing cycle of a Lefschetz singularity of $f$. We call this factorization a \emph{monodromy factorization} associated with $f$. 

In the case of a Lefschetz pencil, recall that blowing-up each base point $b_j$ yields to a $(-1)$-sphere section $S_j$.
The section $S_j$ provides a lift of the monodromy representation $\pi_1 (\Sigma \setminus f(C)) \to \Mod(\Sigma_g)$ to the mapping class group $\Mod_{x_j}(\Sigma_g)$, where $x_j$ is a marked point on $\Sigma_g$. One can then fix a disk neighborhood of this section preserved under the monodromy, and get a lift of the factorization to $\Mod_{\partial \Sigma} (\Sigma_g^1)$, which equals to a power of the boundary parallel Dehn twist. Doing this for each $b_j$ we get a defining word
\[
t_{c_l} \cdot \cdots \cdot t_{c_1} = t_{\delta_1} \cdot \cdots \cdot t_{\delta_m},
\]
in $\Mod_{\partial \Sigma} (\Sigma_g^m)$, where $m = |B|$, the number of base points, and $\delta_j$ are boundary parallel along distinct boundary components of $\Sigma_g^m$. The powers of the $t_{\delta_j}$ are determined by the self-intersection number $-1$ of the corresponding exceptional section.

\vspace{0.15cm}
\subsection{Symplectic $4$-manifolds and Kodaira dimension}  \

By the ground-breaking work of Donaldson every symplectic $4$-manifold $(X, \omega)$ admits a \emph{symplectic} Lefschetz pencil whose fibers are symplectic with respect to $\omega$ \cite{Donaldson}. Conversely, building on a construction of Thurston, Gompf showed that the total space of a Lefschetz fibration with a homologically essential fiber, and in particular blow-up of any pencil, always admits a \textit{compatible} symplectic form $\omega$, for which the fibers are symplectic. This holds whenever the fiber genus is at least $2$, or there are critical points. In this case $\omega$ can be chosen so that not only the fibers but also any chosen collection of disjoint sections are symplectic, and moreover, any such two symplectic forms are deformation equivalent \cite{GS}. We will use the notation $(X, \omega, f)$ to indicate that $f$ is a symplectic Lefschetz pencil/fibration with respect to $\omega$, where any explicitly discussed sections of $f$ will always be assumed to be symplectic with respect to it.  

The Kodaira dimension for projective surfaces can be extended to symplectic $4$-manifolds. Recall that a symplectic $4$-manifold $(X, \omega)$ is called \emph{minimal} if it does not contain any embedded symplectic sphere of square $-1$, and that it can always be blown-down to a minimal symplectic $4$-manifold $(X_{\text{min}}, \omega')$. Let $\kappa_{X_{\text{min}}}$ be the canonical class of $(X_{\text{min}}, \omega_{\text{min}})$. We can now define the \textit{symplectic Kodaira dimension} of $(X, \omega)$, denoted by $\kappa=\kappa(X,\omega)$ as 
\[
\kappa(X,\omega)=\left\{\begin{array}{rl}-\infty& \mbox{if
}\kappa_{X_{\text{min}}}\cdot[\omega_{\text{min}}]<0 \mbox{ or } \kappa_{X_{\text{min}}}^{2}<0 \\
0 & \mbox{if } \kappa_{X_{\text{min}}}\cdot[\omega_{\text{min}}]=\kappa_{X_{\text{min}}}^{2}=0\\ 1 &
\mbox{if }\kappa_{X_{\text{min}}}\cdot[\omega_{\text{min}}]>0\mbox{ and
}\kappa_{X_{\text{min}}}^{2}=0\\2& \mbox{if }\kappa_{X_{\text{min}}}\cdot[\omega_{\text{min}}]>0\mbox{
and }\kappa_{X_{\text{min}}}^{2}>0\end{array}\right.
\]
Importantly, $\kappa$ is independent of the minimal model $(X_{\text{min}}, \omega_{\text{min}})$ and is a smooth invariant of the $4$-manifold $X$ \cite{Li3}.

\vspace{0.2in}
\section{Multisections via mapping class groups} \label{MCG} 

In this section we explain how to capture multisections of Lefschetz fibrations and self-intersections of them in terms of mapping class groups. 

\vspace{0.15cm}
\subsection{A preliminary lemma}\label{subsec_lemmainvolution} \


Let $\Sigma$ be an oriented surface, $\eta\colon \Sigma\to \Sigma$ an orientation preserving involution, $S\subset \Sigma$ a union of components of $\partial\Sigma$ invariant under $\eta$ and $T \subset \Sigma\setminus V$ a finite set invariant under $\eta$. 
We denote the fixed points set of $\eta$ by $V\subset \Sigma$. 
We define a subgroup $C_S(\Sigma,T;\eta)$ of $\Diff^+_S(\Sigma; V, T)$ as follows: 
\[
C_S(\Sigma, T;\eta) = \left\{\left.\varphi \in \Diff^+_S(\Sigma; V,T)~\right|~\varphi \circ \eta = \eta \circ \varphi\right\}. 
\]
We denote the set $C_S(\Sigma, \emptyset; \eta)$ by $C_S(\Sigma; \eta)$. 
The following lemma will be of key use to us for producing several mapping class relations as well as for proving the main theorem:

\begin{lemma}\label{lem_involutionMCG}

The kernel of the natural map 
\[
\eta_\ast \colon \pi_0(C_S(\Sigma, T;\eta))\to \Mod_{S/\eta}(\Sigma/\eta;V/\eta, T/\eta)
\]
induced by the quotient map $/\eta\colon \Sigma \to \Sigma/\eta$ is generated by the class $[\eta]$ if $C_S(\Sigma, T;\eta)$ contains $\eta$ and is trivial otherwise. 

\end{lemma}

\begin{proof}
We first prove the statement under the assumption that $\eta$ does not have fixed points. 
Let $[\varphi]$ be a mapping class in $\Ker(\eta_\ast)$. 
We denote the self-diffeomorphism of $\Sigma/\eta$ induced by $\varphi$ by $\overline{\varphi}$. 
There exists an isotopy $H_t\colon \Sigma/\eta \to \Sigma/\eta$ such that $H_0=\overline{\varphi}$ and $H_1=\id_{\Sigma/\eta}$.
Since $/\eta\colon \Sigma \to \Sigma/\eta$ is an unbranched covering, there exists a lift $\tilde{H}_t\colon \Sigma \to \Sigma$ of $H_t$ under $/\eta$ such that $\tilde{H}_0$ is equal to $\varphi$. 
By the uniqueness of a lift under a covering map, the restrictions $\tilde{H}_t|_{S}$ and $\tilde{H}_t|_{T}$ are the identity maps for any $t$. 
It is easy to verify that the composition $/\eta \circ \tilde{H}_t\circ \eta$ is equal to $/\eta \circ \eta \circ \tilde{H}_t$. 
Since $\tilde{H}_0 =\varphi$ and $\varphi$ commutes with $\eta$, the composition $\tilde{H}_t\circ \eta$ is equal to $\eta \circ \tilde{H}_t$. 
Thus the map $t\mapsto \tilde{H}_t$ gives a path in $C_S(\Sigma,T;\eta)$. 
The map $\tilde{H}_1$ is equal to either $\id_\Sigma$ or $\eta$ since $/\eta \circ \tilde{H}_1$ is equal to $\id_{\Sigma/\eta}$. 
Hence $[\varphi]\in \pi_0(C_S(\Sigma, T;\eta))$ is represented by either the identity map or $\eta$, which completes the proof of the statement. 

We next consider an involution $\eta$ with fixed points.
Let $\nu V\subset \Sigma$ be a neighborhood of $V$ consisting of a disjoint union of disks. 
The following diagram commutes: 
\[
\begin{CD}
\pi_0(C_{S\cup \partial (\nu V)}((\Sigma\setminus \Int(\nu V)), T; \eta)) & @> \eta_\ast >>  \Mod_{S\cup \partial(\nu V)/\eta}(\Sigma\setminus \Int(\nu V)/\eta; T/\eta) \\
@V C_1 VV & @VV C_2 V \\
\pi_0(C_S(\Sigma, T;\eta)) & @> \eta_\ast >> \Mod_{S/\eta}(\Sigma/\eta;V/\eta,T/\eta), 
\end{CD}
\]
where the vertical maps are the capping maps. 
The map $C_2$ is surjective and the kernel of $C_2$ is generated by the Dehn twists along simple closed curves parallel to components of $\partial \nu V$. 
We can regard these Dehn twists are contained in $\pi_0(C_{S\cup \partial (\nu V)}(\Sigma\setminus \Int(\nu V), T; \eta))$. 
It is easy to verify that the map $C_1$ is surjective. 
Furthermore any Dehn twist along a curve parallel to a component of $\partial \nu V$ is contained in the kernel of $C_1$. 
The kernel of the induced map $\eta_\ast$ defined on $\pi_0(C_{S\cup \partial (\nu V)}(\Sigma\setminus \Int(\nu V),T; \eta))$ is generated by $[\eta]$ if $C_{S\cup \partial (\nu V)}(\Sigma\setminus \Int(\nu V),T; \eta)$ contains $\eta$ and is trivial otherwise since the quotient map $/\eta$ on $\Sigma\setminus \Int(\nu V)$ is an unbranched covering. 
Thus the kernel of the map 
\[\eta_\ast\colon \pi_0(C_S(\Sigma,T;\eta)) \to \allowbreak \Mod_{S/\eta}(\Sigma/\eta;V/\eta, T/\eta) \]
is generated by $[\eta]$ if $C_S(\Sigma, T;\eta)$ contains $\eta$ and is trivial otherwise. 
This completes the proof of Lemma~\ref{lem_involutionMCG}. 
\end{proof}

\vspace{0.15cm}
\subsection{Local model for the fibration around a regular branched point} \

Let $f_0: \mathbb{C}^2\rightarrow \mathbb{C}$ be the projection onto the first component. We take a subset $S_0\subset \mathbb{C}^2$ as follows: 
\[
S_{0} = \{(z^2, z)\in \mathbb{C}^2 \hspace{.3em} | \hspace{.3em} z\in\mathbb{C} \}. 
\]
The restriction $f_0|_{S_0}$ is a double branched covering branched at the origin. 

\begin{lemma}\label{lem_localmodel_branchedpoint1}

Let $f:X\rightarrow \Sigma$ be a Lefschetz fibration, $S\subset X$ a multisection of $f$ and $p\in S\setminus \Crit(f)$ a branched point of $f|_{S}$. Then, there exist a local coordinate $\Phi: U\rightarrow \mathbb{C}^2$ of $p$ and a local coordinate $\varphi: V\rightarrow \mathbb{C}$ of $q=f(p)$ which make the following diagram commute: 
\[
\begin{CD}
(U, U\cap S) @>\Phi >> (\mathbb{C}^2, S_0) \\
@V f VV @VV f_0 V \\
V @> \varphi >> \mathbb{C}. 
\end{CD}
\]
That is, $f_0|_{S_0}$ conforms to a local model of a branched covering map with a simple branched point at $p$. 
\end{lemma}

\begin{proof}

Since $p$ is not a critical point of $f$, there exist local coordinates $\Phi_0: U\rightarrow \mathbb{C}^2$ and $\varphi_0: V\rightarrow \mathbb{C}$ of $p$ and $f(p)$, respectively, such that $p$ is mapped to the origin of $\mathbb{C}^2$, and that the following diagram commutes: 
\[
\begin{CD}
U @>\Phi_0 >> \mathbb{C}^2 \\
@V f VV @VV f_0 V \\
V @> \varphi_0 >> \mathbb{C}. 
\end{CD}
\]
Without loss of generality, we can assume that the neighborhood $U$ does not contain any branched points of $f|_{S}$ except $p$. 
Then, the intersection $U\cap S$ is diffeomorphic to $\mathbb{C}$, and $\Phi_0(S)$ is described as follows: 
\[
\Phi_0(S) = \{(s_1(z), s_2(z))\in \mathbb{C}^2 \hspace{.3em} | \hspace{.3em} z\in \mathbb{C} \},  
\] 
where $s_i : \mathbb{C}\rightarrow \mathbb{C}$ is a smooth function ($i=1,2$). 

Since $p$ is a branched point of $f|_{S}$, the map $s_1$ is a double branched covering branched at the origin. Thus, there exist diffeomorphisms $\widetilde{\varphi_1}: \mathbb{C}\rightarrow \mathbb{C}$ and \linebreak $\varphi_1:\mathbb{C}\rightarrow \mathbb{C}$ which make the following diagram commute: 
\[
\begin{CD}
\mathbb{C} @>\widetilde{\varphi_1} >> \mathbb{C} \\
@V s_1  VV @VV (\cdot )^2 V \\
\mathbb{C} @> \varphi_1 >> \mathbb{C}. 
\end{CD}
\]

Now, as $S$ is an embedded surface in $X$, we can assume that the map $z \mapsto (s_1(z), s_2(z))$ is an embedding. 
In particular, $s_2$ is locally diffeomorphic at the origin of $\mathbb{C}$. Thus, by replacing the local coordinates with sufficiently small ones if necessary, we can take diffeomorphisms $\widetilde{\varphi_2}: \mathbb{C}\rightarrow \mathbb{C}$ and $\varphi_2: \mathbb{C}\rightarrow \mathbb{C}$ which make the following diagram commute: 
\[
\begin{CD}
\mathbb{C} @>\widetilde{\varphi_2} >> \mathbb{C} \\
@V s_2  VV @VV \id V \\
\mathbb{C} @> \varphi_2 >> \mathbb{C}. 
\end{CD}
\]
We put $\Phi_1 = \varphi_1 \times \id$ and $\Phi_2 = \id \times \varphi_2$. Now, the following diagram commutes: 
\[
\begin{CD}
U @> \Phi_0 >> \mathbb{C}^2 @>\Phi_1 >> \mathbb{C}^2 @> \Phi_2 >> \mathbb{C}^2 \\
@V s  VV @V f_0 VV @V f_0 VV @V f_0 VV \\
V @> \varphi_0 >> \mathbb{C} @> \varphi_1 >> \mathbb{C} @> \id >> \mathbb{C}. 
\end{CD}
\]
The following equality can be checked easily:
\[
\Phi_2\circ \Phi_1 \circ \Phi_0 (S\cap U) = \{(z^2,  \widetilde{\varphi_2} \circ \widetilde{\varphi_1}^{-1}( z ))\in \mathbb{C}^2~|~z\in \mathbb{C} \}. 
\]
Thus, the diffeomorphisms $\Phi = (\id \times \widetilde{\varphi_1} \circ \widetilde{\varphi_2}^{-1}) \circ \Phi_2\circ \Phi_1 \circ \Phi_0$ and $\varphi = \varphi_1 \circ \varphi_0$ satisfy the desired condition. 
This completes the proof of Lemma \ref{lem_localmodel_branchedpoint1}. 
\end{proof}

Hence we can always make a local coordinate $\varphi$ in Lemma \ref{lem_localmodel_branchedpoint1} compatible with the orientation of $\Sigma$. The branched point $p\in S\setminus \Crit(f)$ of $f|_{S}$ is positive if and only if a local coordinate $\Phi$ of $p$ obtained in Lemma \ref{lem_localmodel_branchedpoint1} is compatible with the orientation of $X$ after making $\varphi$ compatible with the orientation of $\Sigma$. 

\vspace{0.15cm}
\subsection{Standard monodromy factorization around a regular branched point} \

We are now going to study the monodromy factorization around a branched point of a multisection, which will play a key role in the proof of Theorem~\ref{mainthm} below.

We denote by $\Sigma_{g}^n$ an oriented, connected and compact surface of genus $g$ with $n$ boundary components. 
Let $S_0\subset \mathbb{C}^2$ be a standard model of a branched point away from Lefschetz singularities as explained in the previous subsection. 
We denote the subset $\{z\in \mathbb{C} \hspace{.3em} | \hspace{.3em} |z|\leq k \}$ by $B_k$. 
We consider the restriction 
\[ q = p|_{B_1\times B_2}: B_1\times B_2 \rightarrow B_1 \, . \]
The subset $S_0\cap (B_1\times B_2)$ is a bisection of $q$. This bisection, together with an identification $B_2 \cong \Sigma_{0}^{1}$, makes the monodromy $\varrho_0$ of $q|_{q^{-1}(\partial B_1)}$ be contained in the group $\Mod_{\partial\Sigma_0^1}(\Sigma_{0}^{1}; \{s_1,s_2\})$ where $s_1, s_2$ are two points in $q^{-1}(1)\cap S_0$. 
It is known that this monodromy is equal to the positive half twist along an arc between $s_1$ and $s_2$. 
Let $\varepsilon \in \mathbb{R}$ be a sufficiently small real number and we put $e= 1- \varepsilon$. 
We take another subset $S_0^\prime \subset\mathbb{C}^2$ as follows: 
\[
S_0^\prime = \{(z^2, ez)\in \mathbb{C}^2 \hspace{.3em} | \hspace{.3em} z\in\mathbb{C} \}. 
\]
The subset $S_0^\prime \cap (B_1\times B_2)$ is also a bisection of $q$. By using the bisections, we can take a lift $\widetilde{\varrho_{0}}\in \Mod_{\partial \Sigma_{0}^{1}}(\Sigma_{0}^{3}; \{u_1,u_2\})$ of the monodromy $\varrho_0$, where $u_1, u_2$ are points in $\partial \Sigma_{0}^{3}\setminus \partial \Sigma_{0}^{1}$ which cover the set $\pi_0(\partial \Sigma_{0}^{3}\setminus \partial \Sigma_{0}^{1})$.  
Note that the group $\Mod_{\partial \Sigma_{0}^{1}}(\Sigma_{0}^{3}; \{u_1,u_2\})$ is isomorphic to the group $\Mod_{\partial\Sigma_0^1}(\Sigma_{0}^{1}; \{v_1,v_2\})$, where $v_i$ is a non-zero tangent vector in $T_{s_i}\Sigma_{0}^{1}$.

\begin{lemma}\label{lem_lift_monodromy1}
The mapping class $\widetilde{\varrho_{0}}$ is represented by the map described in the Figure~\ref{lift_monodromy1}. 
\begin{figure}[htbp]
\begin{center}
\includegraphics[width=65mm]{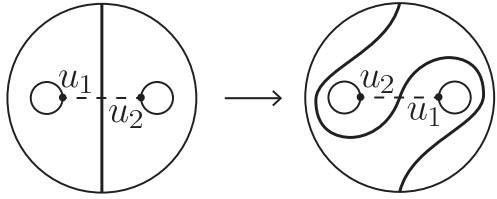}
\end{center}
\caption{The element $\widetilde{\varrho_{0}}$ interchanges the points $u_1, u_2$, and keeps the dotted arc between $u_1$ and $u_2$. }
\label{lift_monodromy1}
\end{figure}
\end{lemma}

\begin{proof}
The element $\widetilde{\varrho_{0}}$ is a lift of $\varrho_{0}$. 
Thus the bold arc in Figure~\ref{lift_monodromy1} should be sent by a representative of $\tilde{\varrho}_0$ (up to isotopy) as described in the figure. 
It is sufficient to prove an arc connecting $s_1$ and $s_2$ is preserved by some representative of $\widetilde{\varrho_{0}}$ since the group $\Mod_{\partial \Sigma_{0}^{1}}(\Sigma_{0}^{3}; \{u_1,u_2\})$ is isomorphic to the group $\Mod_{\partial\Sigma_0^1}(\Sigma_{0}^{1}; \{v_1,v_2\})$. 
We denote the arc $\{ (1, 1-2t) \in \mathbb{C}^2 \hspace{.3em} | \hspace{.3em} t\in [0,1] \}$ by $\gamma \subset q^{-1}(1)$. 
This arc connects the two points in $S_0 \cap q^{-1}(1)$. We take a horizontal distribution $\mathcal{P}$ of $q|_{q^{-1}(\nu\partial B_1)}$ so that it coincides the following distribution on $\partial B_1 \times B_{\frac{3}{2}}$: 
\[
\biggl< \biggl(\frac{\partial}{\partial x_1}\biggr) + \frac{x_3}{2}\biggl(\frac{\partial}{\partial x_3}\biggr) - \frac{x_4}{2} \biggl(\frac{\partial}{\partial x_4}\biggr), \biggl(\frac{\partial}{\partial x_2}\biggr) + \frac{x_4}{2}\biggl(\frac{\partial}{\partial x_3}\biggr) + \frac{x_3}{2} \biggl(\frac{\partial}{\partial x_4}\biggr) \biggr>, 
\]
where $(x_1,x_2,x_3,x_4)$ is a real coordinate determined by the formula 
\[(z_1,z_2) = (x_1+\sqrt{-1}x_2, x_3+ \sqrt{-1}x_4) \,. \]
We define a loop $c: [0, 2\pi]\rightarrow \partial B_1$ as follows: 
\[
c(t) = \exp(\sqrt{-1}t). 
\]
Take a point $t_0 \in [-1,1]$. It is easy to see that the horizontal lift $\tilde{c}_{t_0}(t)$ with base point $w=(1,0,t_0, 0 ) \in q^{-1}(1)$ is given by: 
\[
\tilde{c}_{t_0}(t) = \Bigr(\cos(t), \sin(t), t_{0}\cos\Bigl(\frac{t}{2}\Bigr), t_{0} \sin\Bigl(\frac{t}{2}\Bigr) \Bigl). 
\]
Thus, the arc $\gamma$ is preserved by the parallel transport along the curve $c$ with respect to $\mathcal{P}$. 
Since this parallel transport is a representative of $\widetilde{\varrho_{0}}$, this completes the proof of Lemma \ref{lem_lift_monodromy1}. 
\end{proof}

The two bisections $S_0$ and $S_0^\prime$ intersect only at the origin, but do not intersect transversely. 
In order to make the two bisections intersect transversely, we will take a small perturbation of $S_0^\prime$. 
We first take a smooth function $\rho: \mathbb{R}\rightarrow [0,\varepsilon]$ satisfying the following conditions: 

\begin{enumerate}[(a)]
\item $\rho(t) = \rho(-t)$; 
\item $\rho(t) = \varepsilon^2$ for all $t\in [0, \frac{\varepsilon}{2}]$; 
\item $\rho(t) = 0$ for all $t\in [\varepsilon, \infty)$; 
\item $-3\varepsilon <\frac{d\rho}{dt}(t) < 0$ for all $t \in [\frac{\varepsilon}{2} , \varepsilon]$. 
\end{enumerate}

\noindent We define the subset $S^\prime_{0,\rho}\subset \mathbb{C}^2$ as follows: 
\[
S^\prime_{0,\rho} = \{ (z^2, ez+ \rho(|z|^2)) \in \mathbb{C}^2 \hspace{.3em} | \hspace{.3em} z\in \mathbb{C} \}. 
\]
The two subsets $S_0$ and $S^\prime_{0,\rho}$ intersect at $(r_1^2, r_1), (r_2^2, r_2) \in \mathbb{C}^2$, where $r_1,r_2\in \mathbb{R}$ is the real numbers which satisfy the following conditions: 
\[
r_1 = \frac{\rho(r_1^2)}{\varepsilon} ~,~ r_2 = \frac{\rho(r_2^2)}{2- \varepsilon}. 
\]
We can see that $S_0$ intersects $S^\prime_{0,\rho}$ at both points transversely and positively with respect to the standard orientation of $\mathbb{C}^2$.

\vspace{0.15cm}
\subsection{Multisections branching at Lefschetz critical points} \

We will now study the local model around a branched point of a multisection coinciding with a Lefschetz critical point of the fibration. Such branched points appear \textit{exclusively} in Loi and Piergallini's description of compact Stein surfaces, up to diffeomorphisms, as total spaces of \textit{allowable Lefschetz fibrations} over the $2$-disk with bounded fibers, arising as the branched cover of the projection $D^2 \x D^2 \to D^2$ branched along a positive multisection \cite{LP} (also see \cite{AkbulutOzbagci}). Such a multisection, along with the fiber, carries the entire information one needs to describe the diffeomorphism type of any compact Stein surface.

We take two points $s_1,s_2 \in \Int(\Sigma_{0}^{2})$. We denote an involution with fixed point set $\{s_1,s_2\}$ by $\iota: \Sigma_{0}^{2}\rightarrow \Sigma_{0}^{2}$. 
The quotient space $\Sigma_0^2/\iota$ is diffeomorphic to the disk $\Sigma_0^1$. 
Denote the images of $s_1$ and $s_2$ under the quotient map $\Sigma_0^2\to \Sigma_0^2/\iota \cong \Sigma_0^1$ by $s_1^\prime$ and $s_1^\prime$, respectively. 
The group $\Mod_{\partial\Sigma_0^1}(\Sigma_0^1; \{s_1^\prime, s_2^\prime\})$ is an infinite cyclic group. 
By Lemma \ref{lem_involutionMCG} the natural map $\pi_0(C_{\partial \Sigma_0^2}(\Sigma_{0}^{2}; \iota)) \rightarrow \Mod_{\partial\Sigma_0^1}(\Sigma_0^1; \{s_1^\prime, s_2^\prime\})$ induced by the quotient map is injective. 
The inclusion map $C_{\partial \Sigma_0^2}(\Sigma_{0}^{2}; \iota)\hookrightarrow \Diff^+_{\partial \Sigma_{0}^{2}}(\Sigma_{0}^{2})$ induces the homomorphism $i: \pi_0(C_{\partial \Sigma_0^2}(\Sigma_{0}^{2}; \iota)) \rightarrow \Mod_{\partial \Sigma_{0}^{2}}(\Sigma_{0}^{2}) \cong \mathbb{Z}$. 
Since this map is surjective, the group $\pi_0(C_{\partial \Sigma_0^2}(\Sigma_{0}^{2}; \iota))$ is also an infinite cyclic group and the map $i$ is an isomorphism. 
On the other hand, the inclusion map $C(\Sigma_{0}^{2}; \iota)\hookrightarrow \Diff^+_{\partial \Sigma_{0}^{2}}(\Sigma_{0}^{2}; \{s_1,s_2\})$ also induces a homomorphism $i_{\ast}: \pi_0(C_{\partial \Sigma_0^2}(\Sigma_{0}^{2}; \iota)) \rightarrow \Mod_{\partial \Sigma_{0}^{2}}(\Sigma_{0}^{2}; \{s_1,s_2\})$. 
We denote by 
\[F_{s_1,s_2}: \Mod_{\partial \Sigma_{0}^{2}}(\Sigma_0^2; \{s_1,s_2\})\rightarrow \Mod_{\partial \Sigma_{0}^{2}}(\Sigma_0^2)\]
the forgetful map. Since the composition $F_{s_1,s_2}\circ i_{\ast}$ is equal to $i$ and $i$ is isomorphic, the map $i_{\ast}$ is injective. 
Thus, we can regard the group $\pi_0(C_{\partial \Sigma_0^2}(\Sigma_0^2; \iota)) \cong \Mod_{\partial \Sigma_{0}^{2}}(\Sigma_{0}^{2})$ as the subgroup of $\Mod_{\partial \Sigma_{0}^{2}}(\Sigma_{0}^{2}; \{s_1,s_2\})$. 
Under this identification, the Dehn twist $t_c\in \Mod_{\partial \Sigma_{0}^{2}}(\Sigma_{0}^{2})$ along the curve parallel to $\partial \Sigma_{0}^{2}$, which is the generator of this group, is regarded as an element in $\Mod_{\partial \Sigma_{0}^{2}}(\Sigma_{0}^{2}; \{s_1,s_2\})$ described in Figure~\ref{lift_Dehntwist}. 
\begin{figure}[htbp]
\begin{center}
\includegraphics[width=50mm]{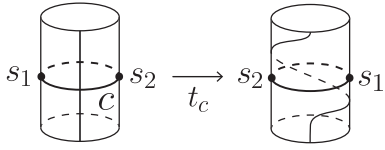}
\end{center}
\caption{The element $t_c$ interchanges $s_1$ and $s_2$. }
\label{lift_Dehntwist}
\end{figure}

We denote by $Y\subset\mathbb{C}^2$ the intersection $B_2\times B_2 \cap f^{-1}(B_1)$, where $B_k$ is the disk $\{z\in\mathbb{C}~|~\left|z\right|\leq k\}$ and $f:\mathbb{C}^2\rightarrow \mathbb{C}$ is the standard local model of a Lefschetz singularity, that is, $f$ is defined as $f(z_1,z_2)= z_1z_2$. 
Let $f_0$ be the restriction $f|_{Y}$. 
Take the standard bisection of $\Delta_0 = \{(z,z)\in Y~|~ z\in B_1\}$ of $f_0$. 
We define the involution $\eta: \mathbb{C}^2\rightarrow \mathbb{C}^2$ as follows: 
\[
\eta(z_1,z_2) = (z_2,z_1). 
\]
The fixed point set of $\eta$ is equal to $\Delta_0$. 
The regular fiber $f_0^{-1}(1)$ is the annulus $\Sigma_0^2$. 
We take an identification $f_0^{-1}(1)\cong \Sigma_0^2$ so that the restriction $\eta|_{f_0^{-1}(1)}$ equals to the involution $\iota$. 
By taking a horizontal distribution $\mathcal{P}$ of the fibration $f_0|_{Y\setminus \{0\}}$ which is along both $\Delta_0$ and $\partial Y$, we can regard the monodromy $\varrho_0$ of $\partial B_1$ as an element of the group $\Mod_{\partial \Sigma_{0}^{2}}(\Sigma_{0}^{2}; \{s_1,s_2\})$, where $\{s_1,s_2\}$ is the intersection $\Delta_0\cap f_0^{-1}(1)$. 

\begin{lemma}\label{lem_MCG_branchedonCrit}
Under the identification of $f_0^{-1}(1)$ with $\Sigma_0^2$ as above, the monodromy $\varrho_0$ is equal to the Dehn twist $t_{c}\in \Mod_{\partial \Sigma_{0}^{2}}(\Sigma_{0}^{2}; \{s_1,s_2\})$. 
\end{lemma}

\begin{proof}
We take a horizontal distribution $\mathcal{P}$ so that $\mathcal{P}$ is preserved by $\eta$. 
The monodromy $\varrho_0$ is contained in the group $\pi_0(C_{\partial \Sigma_0^2}(\Sigma_{0}^{2}, \iota))\subset \Mod_{\partial \Sigma_{0}^{2}}(\Sigma_{0}^{2}; \{s_1,s_2\})$. 
Furthermore, using the result in \cite{Kas_1980}, it is easy to see that this monodromy is sent to the Dehn twist $t_c\in \Mod_{\partial \Sigma_{0}^{2}}(\Sigma_{0}^{2})$ by $F_{s_1,s_2}$. 
%
\end{proof}

We take a disk neighborhood $D_i\subset \Sigma_{0}^{2}$ of the point $s_i$ which is preserved by $\iota$. 
We put $D=D_1\amalg D_2$ and fix an identification $\Sigma_{0}^{2}\setminus D \cong \Sigma_{0}^{4}$. 
We also take points $u_i, u_i^\prime\in \partial D_i$ so that $\iota(u_i) = u_i^\prime$. 
We can define the following homomorphism: 
\[
Cap: \Mod_{\partial \Sigma_{0}^{2}}(\Sigma_{0}^{4}; \{u_1, u_2\}) \rightarrow \Mod_{\partial \Sigma_{0}^{2}}(\Sigma_{0}^{2}; \{s_1,s_2\}), 
\]
by capping $\Sigma_{0}^{4}$ by $D$. 

We take a sufficiently small number $\varepsilon >0$ and put $\xi = \exp(\sqrt{-1}\varepsilon)$. 
We define another bisection $\Delta_0^\prime$ of $f_0$ as follows: 
\[
\Delta_0^\prime = \{(\xi z , \xi^{-1} z) \in Y~|~z\in B_1\}. 
\]
Note that $\Delta_0^\prime$ intersects $\Delta_0$ at the origin transversely. 
This bisection, together with the bisection $\Delta_0$, gives a lift $\widetilde{\varrho_0} \in \Mod_{\partial \Sigma_{0}^{2}}(\Sigma_{0}^{4}; \{u_1,u_2\})$ of the monodromy $\varrho_0$ under the map $Cap$. 

\begin{lemma}\label{lem_liftmonodromy2}

Under a suitable identification $\Sigma_0^4\cong f_0^{-1}(1)\setminus \nu \Delta_0$, the monodromy $\widetilde{\varrho_0}$ is represented by the map described in Figure~\ref{lift_monodromy2}. 

\begin{figure}[htbp]
\begin{center}
\includegraphics[width=100mm]{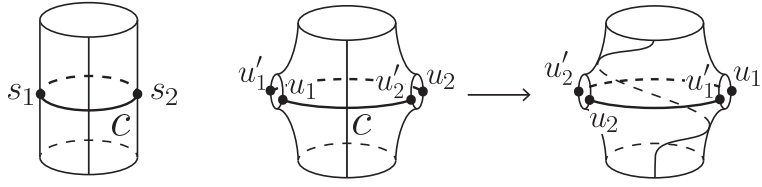}
\end{center}
\caption{The element $\widetilde{\varrho_0}$ interchanges the points $u_1,u_2$. }
\label{lift_monodromy2}
\end{figure}

\end{lemma}

\begin{proof}
The map described in Figure~\ref{lift_monodromy2} is contained in $C_{\partial \Sigma_0^2}(\Sigma_{0}^4, \{u_1,u_1^\prime,u_2,u_2^\prime\}; \iota)$.
By the same argument as in the proof of Lemma \ref{lem_MCG_branchedonCrit}, we can assume that the element $\widetilde{\varrho_0}$ is contained in the group $\pi_0(C_{\partial \Sigma_0^2}(\Sigma_{0}^4, \{u_1,u_1^\prime,u_2,u_2^\prime\}; \iota))$. 
It is easy to see that the following map is a diffeomorphism: 
	\begin{equation*}
	\begin{array}{ccc}
\mathbb{C}^2/\eta & \rightarrow & \mathbb{C}^2 \\
\rotatebox{90}{$\in$} &  & \rotatebox{90}{$\in$} \\ 
\bigr[ (z_1 , z_2) \bigl] & \mapsto     & \Bigl( z_1z_2, \dfrac{z_1+z_2}{2} \Bigr).
	\end{array}
	\end{equation*}
We identify these spaces via this diffeomorphism. 
The following diagram commutes: 
\[
\begin{CD}
(\mathbb{C}^2, \Delta_0, \Delta_0^\prime) @>/\eta >> (\mathbb{C}^2, S_0, S^\prime_0) \\
@V f VV @VV p V \\
\mathbb{C} @> \id >> \mathbb{C},  
\end{CD}
\]
where $S_0$ and $S_0^\prime$ are the subsets of $\mathbb{C}^2$ defined in the previous section (in this case, $e$ is equal to $\Re(\xi)$).  
Thus the monodromy $\widetilde{\varrho_0}$ is mapped to the mapping class described in Figure~\ref{lift_monodromy1} by the map 
\[\eta_\ast :\pi_0(C_{\partial \Sigma_0^2}(\Sigma_{0}^4, \{u_1,u_1^\prime, u_2,u_2^\prime\}; \iota))\to \Mod_{\partial \Sigma_{0}^1}(\Sigma_{0}^3; \{u_1,u_2\}) \]
induced by $/\eta$. On the other hand, we can see that the mapping class described in Figure~\ref{lift_monodromy2} is also mapped to that descried in Figure~\ref{lift_monodromy1} by $\eta_\ast$. 
Since $\eta_\ast$ is injective by Lemma \ref{lem_involutionMCG}, these two mapping classes coincide. 
%
\end{proof}

\vspace{0.15cm}
\subsection{Capturing multisections via mapping class group factorizations} \

With all the preliminary results we have obtained in the previous subsections, we are now ready to prove the main theorem of this section:

\begin{theorem}\label{mainthm}
Let $f: X \rightarrow S^2$ be a genus-$g$ Lefschetz fibration with monodromy factorization
\[ t_{c_l}\cdot \cdots \cdot t_{c_1} =1 \, . \]
Let $S\subset X$ be a genus-$g$ surface with self-intersection $m$, which is an $n$-section of $f$ with $k$ branched points away from $\Crit(f)$, and $r$ branched points at Lefschetz singularities corresponding to cycles $c_1,\ldots, c_r$. Then there exists a lift $\widetilde{c_i} \subset \Sigma_{g}^{n}$ of $c_i$ such that the following holds in $\Mod(\Sigma_{g}^{n}; \{u_1,\ldots, u_n\})$: 

\begin{equation}\label{eq_relation_multisection with intersection}
\tilde{\tau}_{\alpha_k} \cdots \tilde{\tau}_{\alpha_1} \cdot t_{\widetilde{c_{l}}} \cdots t_{\widetilde{c_{r+1}}}\cdot \widetilde{t_{c_{r}}}\cdots \widetilde{t_{c_{1}}} = t_{\delta_1}^{a_1}\cdots t_{\delta_n}^{a_n}, 
\end{equation}
where $\{u_1,\ldots, u_n\}$ is a subset of $\partial \Sigma_{g}^{n}$ which covers all the elements of $\pi_0(\partial \Sigma_{g}^{n})$, $\tilde{\tau}_{\alpha_i}$ is a lift of a half twist along the arc $\alpha_i$ as described in Figure~\ref{lift_monodromy1}, $\widetilde{t_{c_i}}$ is a lift of the Dehn twist $t_{c_i}$ as described in Figure~\ref{lift_monodromy2}, and $\{\delta_1, \ldots, \delta_n\}$ is a set of simple closed curves parallel to $\partial \Sigma_{g}^{n}$. 
Here the arcs for $\tilde{\tau}_{\alpha_1}, \ldots, \tilde{\tau}_{\alpha_k}$ and the Dehn twist curves for $\widetilde{t_{c_{1}}}, \ldots , \widetilde{t_{c_{r}}}$ should contain all $u_1, \ldots, u_n$, and the integral equalities $g(S)= \frac{1}{2}(k + r) - n + 1$ and  $m=-(\Sigma_{i=1}^n a_i) + 2k+ r$ should hold.  
 
Conversely, for any relation in $\Mod(\Sigma_{g}^{n}; \{u_1,\ldots, u_n\})$ of the form ~\eqref{eq_relation_multisection with intersection} and satisfying the conditions above, there exists a genus-$g$ Lefschetz fibration $f: X\rightarrow S^2$ with a connected $n$-section $S\subset X$ of genus $\frac{1}{2}(k+r) - n + 1$ and self-intersection $- (\Sigma_{i=1}^n a_i ) + 2k +r$,  whose monodromy factorization is given by the image of the factorization on the left hand side  of~\eqref{eq_relation_multisection with intersection} under $i_\ast:\Mod(\Sigma_{g}^{n}; \{u_1,\ldots, u_n\}) \rightarrow \Mod(\Sigma_{g})$ which is induced by the inclusion 
$i:\Sigma_{g}^{n}\hookrightarrow \Sigma_g$. 
\end{theorem}

\noindent Note that after relabeling the arcs we choose for the monodromy description of $f$, we can always assume that the first $r$ cycles are the ones corresponding to those where $S$ is branched at. 

\vspace{0.1in}
The reader might find it illuminating to look at an example before we move on to proving our theorem:

\begin{example} \label{FactExample}
A monodromy factorization of a genus-$2$ Lefschetz fibration with a $2$-section we will produce in Section~\ref{SectionGenus2Counterex} is the following:
\[ \large{t_{\delta_1}^2 t_{\delta_2}^3= (t_{d_3}t_{d_2}t_{d_1})^2 t_{y_2} \tilde{\tau}_{\alpha_2} t_z t_{y_1} \tilde{\tau}_{\alpha_1} t_{t_{c_1}^{-1} t_{c_3}^{-2} (c_2)} t_{t_{c_3}^{-1}(c_2)} (t_{c_1} t_{c_2} t_{c_3})^2}  \, , \]
in $\Mod(\Sigma_{2}^{2}; \{u_1,u_2 \})$, where the Dehn twist curves $c_i, d_i, y_j, z$ are given in blue in Figure\ref{SCYEnriques}, the $\delta_i$ are the two boundary components, and in red are the arcs $\alpha_j$ for the half-twists $\tilde{\tau}_{\alpha_j}$. (One can of course conjugate each $\tilde{\tau}_{\alpha_i}$ all the way to the left to put it into the above ``standard form''.) 
\begin{figure}[htbp]
\begin{center}
\includegraphics[width=125mm]{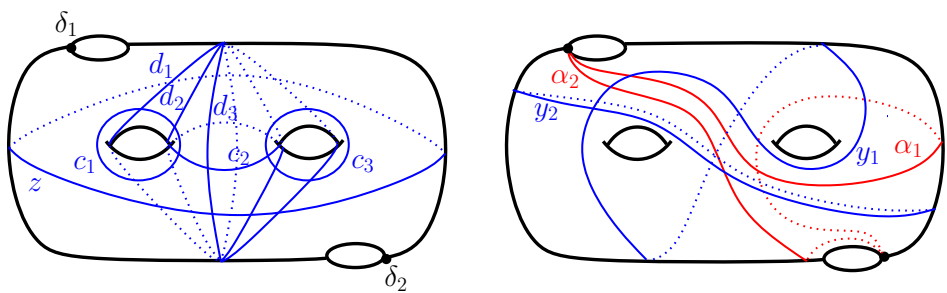}
\end{center}
\caption{Dehn twist curves $c_1, c_2, c_3, d_1, d_2, d_3, z, y_1, y_2$ and arc twist curves $\alpha_1, \alpha_2$ in $\Sigma_2^2$ with framed boundary.}
\label{SCYEnriques}
\end{figure}
Since $u_1$ and $u_2$ are connected by $\tau_i$, this gives a connected $2$-section $S$ of genus $\frac{k}{2}-n+1=0$ and self-intersection $-\sum_{i=1}^n a_i + 2k = -1$; so it is an exceptional sphere. 
\end{example}

\begin{proof}[Proof of Theorem~\ref{mainthm}]
For a given genus-$g$ Lefschetz fibration $f:X\rightarrow S^2$ with an $n$-section $S$ and its monodromy factorization $t_{c_l}\cdots t_{c_1}=1$, let $\gamma_1,\ldots, \gamma_l$ be reference paths from a regular value $q_0\in S^2$ which gives the factorization $t_{c_l}\cdots t_{c_1}=1$. 
We take reference paths $\alpha_1,\ldots, \alpha_{k}$ satisfying the following properties: 

\begin{itemize}

\item $\alpha_i$ connects $q_0$ with the image of a branched point of $S$ away from $\Crit(f)$; 

\item $\gamma_1,\ldots, \gamma_{l}, \alpha_1,\ldots, \alpha_{k}$ appear in this order when we go around $q_0$ counterclockwise. 

\end{itemize}

\noindent
We take a perturbation $S^\prime$ of $S$ so that the pair $(S, S^\prime)$ coincides with either of the pairs $(S_0, S_0^\prime)$ or $(\Delta_0, \Delta_0^\prime)$ in a small coordinate neighborhood of each branched point of $S$. 
This perturbation gives a lift of monodromies of $f$ to the group $\Mod(\Sigma_g^n; \{u_1,\ldots, u_n\})$. 
By Lemma \ref{lem_lift_monodromy1}, local monodromies obtained from paths $\alpha_i$ are lifts of half twists described in Figure~\ref{lift_monodromy1}. 
On the other hand, by Lemma \ref{lem_liftmonodromy2}, a local monodromy obtained from a path $\gamma_i$ ($i \in \{1,\ldots, r \}$) is a lift of the Dehn twist $t_{c_i}$ described in Figure~\ref{lift_monodromy2}. Thus we can obtain a factorization in Theorem \ref{mainthm}. 

Using the observation following the proof of Lemma \ref{lem_lift_monodromy1} and the fact that $\Delta_0$ intersects $\Delta_0^\prime$ at the origin transversely, it is easy to verify that this factorization satisfies the condition on the self-intersection number of $S$. 

Conversely, for a given lift of a factorization given in Theorem \ref{mainthm}, we can prescribe a genus-$g$ Lefschetz fibration $f: X\rightarrow S^2$ and an $n$-section $S$ of $f$ with desired conditions by pasting local models given in the present section according to the factorization. 

There is a correspondence between a connected multisection $S$ and that of a graph $\Gamma$ with vertices corresponding to $u_i$, and edges between $u_i$ and $u_{i'}$ corresponding to half twists 
$\tilde{\tau}_{\alpha_j}$ or Dehn twists $\widetilde{t_{c_{j}}}$ in the relation \eqref{eq_relation_multisection with intersection} interchanging them. The Euler characteristic of $S$ is then given by $2v-e$, for $v$ the number of vertices and $e$ the number of edges of $\Gamma$. Since $S$ is connected, the arcs for $\tilde{\tau}_1, \ldots, \tilde{\tau}_{k}$ and the Dehn twist curves for $\widetilde{t_{c_{1}}}, \ldots , \widetilde{t_{c_{r}}}$ should contain all of $u_1, \ldots, u_n$, and in this case, we have $g(S)= \frac{1}{2}(k +r) - n + 1$. 
\end{proof}

Per the last paragraph of the proof above, the monodromy factorization in Theorem~\ref{mainthm} can be generalized to disconnected multisections in a straightforward way ---see Subsection~\ref{SCYsubsection}. A sample calculation of a monodromy factorization of a Lefschetz fibration with its multisections is given in the Appendix, and many more examples can be found in Sections $4$--$6$. 

\begin{remark} 
After a small modification of the proof above we can similarly obtain a monodromy factorization for a \textit{not necessarily positive} multisection, where each negative branched point away from $\Crit(f)$ contributes $-2$ and each branched point at a negative critical point contributes $-1$ to the total count of the self-intersection of the multisection. 
\end{remark}

\begin{remark} \label{rem_perturbation} \
We shall note that, although multisections going through Lefschetz critical points are of particular interest in certain contexts (for instance for allowable Lefschetz fibrations on Stein surfaces), it is in fact always possible to perturb any given multisection of a Lefschetz fibration so as to obtain one which is branched completely away from the Lefschetz critical points. This can be achieved by the following perturbation around each branched point on a Lefschetz singularity: 
\[
\Delta_\varepsilon = \{(z+\varepsilon, z-\varepsilon)\in \mathbb{C}^2~|~ z\in\mathbb{C} \}, 
\]
where $\varepsilon$ is a sufficiently small positive number. 
\begin{figure}[htbp]
\begin{center}
\includegraphics[width=25mm]{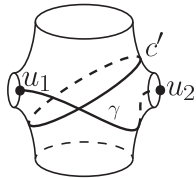}
\end{center}
\caption{Simple closed curves and paths in $\Sigma_0^4$. }
\label{interpretation_lift}
\end{figure}
%
In this perturbation, a branched point on a Lefschetz singularity is substituted for a positive branched point. 
Indeed, we can easily verify the following relation in the group $\Mod_{\partial \Sigma_0^2}(\Sigma_0^4; \{u_1,u_2\})$ (using the Alexander method \cite[Proposition 2.8]{Farb_Margalit_2011}, for example): 
\begin{equation}\label{eq_substitution}
\widetilde{t_{c}} = t_{\delta}^{-1} t_{c^\prime} \tau_{\gamma}, 
\end{equation}
where $\delta$ is a simple closed curve parallel to the boundary component containing $u_2$, $c^\prime$ is a simple closed curve described in Figure~\ref{interpretation_lift} and $\tau_{\gamma}$ is a lift of a half twist preserving the path $\gamma$ in Figure~\ref{interpretation_lift}. 

Using Theorem \ref{mainthm}, we can then make a multisection avoiding Lefschetz singularities by substituting a lift $\widetilde{t_c}$ in a lift of a factorization \eqref{eq_relation_multisection with intersection} of the right hand side of \eqref{eq_substitution}.
\end{remark}

\begin{remark}[Hurwitz equivalence for Lefschetz fibrations with multisections] 
It is well-known that for $g \geq 2$, there is a one-to-one correspondence between genus--$g$ Lefschetz fibrations \textit{up to isomorphisms} and monodromy factorizations in 
$\Mod(\Sigma_{g})$ \textit{up to Hurwitz moves and global conjugations}. It is possible to extend this correspondence to our setting, by considering positive factorizations in the framed mapping class group $\Mod(\Sigma_{g}^{n}; \{u_1,\ldots, u_n\})$ up to usual Hurwitz moves, global conjugations that are allowed to swap boundary components, and an additional move which compensates for the ambiguity in boundary framings. A detailed study will be given in \cite{BH2}.
\end{remark}

\vspace{0.2in}
\section{Lefschetz fibrations on symplectic Calabi-Yau $4$-manifolds} \label{SCY}
%

Symplectic $4$-manifolds of negative Kodaira dimension are classified up to symplectomorphisms, which are precisely the rational and ruled surfaces \cite{Li3}. The next compelling target has been the symplectic $4$-manifolds of Kodaira dimension zero, which are analogues of the Calabi-Yau surfaces that have torsion canonical class \cite{Li1}. With a slight abuse of language, we will thus call $(X, \omega)$ with $\kappa=0$ a \textit{symplectic Calabi-Yau}, referring to its minimal model having a torsion canonical class. It has been shown by Li, and independently by Bauer \cite{Li2, Li1, Bauer}, that the rational homology type of any \textit{minimal} symplectic Calabi-Yau $4$-manifold is that of either a torus bundle over a torus, the $\K$ surface, or the Enriques surface. In particular, a folklore conjecture states that a symplectic Calabi-Yau with $b_1=0$ is diffeomorphic to a (blow-up of) either the Enriques surface or the $\K$ surface. 

With this conjecture in mind, below we will determine the defining properties for a Lefschetz pencil\,/\, fibration to be on a (blow-up of) a symplectic Calabi-Yau $4$-manifold, essentially relying on Taubes' seminal work \cite{T, T2}.  We will then construct two model examples, a genus-$3$ Lefschetz pencil on a symplectic Calabi-Yau $\K$ surface, and a genus-$2$ pencil on a symplectic Calabi-Yau Enriques surface. (That is, these are symplectic Calabi-Yaus \textit{homeomorphic} to $\K$ and Enriques surfaces, respectively.)

\vspace{0.15cm}
\subsection{Characterizing Lefschetz fibrations on symplectic Calabi-Yaus} \label{SCYsubsection} \

Fibers of a symplectic Lefschetz fibration $(X,\omega, f)$ are $J$-holomorphic with respect to any almost complex structure $J$ compatible with $\omega$. It follows from Taubes' seminal work on the correspondence between Gromov and Seiberg-Witten invariants on symplectic $4$-manifolds with $b^+(X)>1$ that exceptional classes $e_j$ in $H_2(X)$ are represented by disjoint $J$-holomorphic $(-1)$-spheres $S_j$ \cite{T, T2}, and by the work of Li and Liu \cite{LiLiu}, the same holds even when $b^+(X)=1$, as long as $X$ is not a rational or ruled surface. We then conclude from the positivity of intersections for $J$-holomorphic curves that each $S_j$ is an $n_j$-section, intersecting the genus $g \geq 2$ fiber $F$ positively at exactly $n_j= S \cdot F \geq 1$ points. Moreover, in this case, we can use the Seiberg-Witten adjunction inequality to show that $\sum n_j = (\sum S_j) \cdot F \leq 2g-2$. Note that this can fail to be true only for rational and ruled surfaces; a Lefschetz pencil on a rational or ruled surface can have more than $2g-2$ base points. We will show that the equality is sharp precisely for Lefschetz fibrations on symplectic Calabi-Yaus. 

We can now characterize Lefschetz fibrations on minimal symplectic Calabi-Yau $4$-manifolds and their blow-ups, using Sato's work in \cite{Sato_2013}:

\begin{theorem}\label{SCYLF}
Let $(X,f)$ be a genus-$g$ Lefschetz fibration with $g \geq 2$, and $X$ be neither rational nor ruled. Then, there exists a symplectic form $\omega$ on $X$ compatible with $f$ such that $(X,\omega)$ is a (blow-up of) a symplectic Calabi-Yau $4$-manifold, if and only if there is a disjoint collection of $(-1)$-spheres that are $n_j$-sections of $(X,f)$ such that $\sum_j n_j = 2g-2$.
\end{theorem}

\begin{proof}
If $X$ is minimal, then $(X, \omega, f)$ can be a symplectic Calabi-Yau only if the fiber genus is $1$, by the adjunction formula. We can thus assume that $X$ is not minimal and $g \geq 2$. In this case, $(X, \omega)$ should be an $m \geq 1$ times blow-up of a minimal symplectic Calabi-Yau, so $c_1(X, \omega)$ is Poincar\'{e} dual to $\sum_{j=1}^m e_j$, where $e_j$ are the exceptional classes. We have $m$ disjoint $n_j$-sections $S_j$, representing the exceptional classes $e_j$. Then the adjunction formula for the symplectic fiber dictates that
\[  \sum_j n_j = (\sum S_j) \cdot F = \kappa_X \cdot F = 2g-2 - F^2 = 2g -2 \, . \]

Conversely, it is shown by Sato \cite[Theorem~5.5.]{Sato_2013} that for a genus $g\geq2$ Lefschetz fibration on a non-minimal $4$-manifold $(X,f)$, where $X$ is not rational or ruled, if the maximal collection of exceptional classes $e_j$ intersect the fiber exactly at $2g-2$ times, then $c_1(X, \omega)$ is Poincar\'{e} dual to $\sum e_j$. Although there is an oversight in this observation, which for instance contradicts with the case of Lefschetz fibrations on blow-ups of the Enriques surface (such examples for homotopy Enriques surfaces are given in the later sections), Sato's proof in \cite{Sato_2013}, which is obtained by a thorough analysis of  intersections between pseudo-holomorphic curves, goes through for a \textit{rational} homology class, i.e. modulo torsion. With this corrected statement in mind, blowing down all $e_j$ yields a minimal symplectic model for $(X, \omega)$ with torsion canonical class.
\end{proof}

We will thus call $(X,f)$ a genus-$g$ \textit{symplectic Calabi-Yau Lefschetz fibration} if $X$ is not rational or ruled, and there is a disjoint collection of $(-1)$-spheres that are $n_j$-sections of $(X,f)$ with $\sum_j n_j = 2g-2$. Note that not every symplectic Calabi-Yau Lefschetz fibration is a blow-up of a Lefschetz pencil on a minimal symplectic Calabi-Yau, the examples of which we will provide in sections $4$--$6$.

Let $W$ be a factorization of the multitwist $t_{\delta_1}^{a_1}\cdots t_{\delta_{n}}^{a_{n}}$ in $\Mod(\Sigma_{g}^{2g-2}; \{u_1,\ldots, u_{n}\})$
\begin{equation} \label{scyfactorization}
\tilde{\tau}_{\alpha_k} \cdots \tilde{\tau}_{\alpha_1} \cdot t_{\widetilde{c_{l}}} \cdots t_{\widetilde{c_1}} = t_{\delta_1}^{a_1}\cdots t_{\delta_{n}}^{a_{n}}, 
\end{equation}
for $n= 2g-2$. Recall that by Remark~\ref{rem_perturbation} we can simplify the right-hand side as above so that there are no $\widetilde{t_{c_j}}$. Consider the associated graph $\Gamma= \Gamma_W$ whose vertices correspond to $u_i$ and edges to half twists $\tilde{\tau}_{\alpha_j}$ interchanging them. After relabeling $\delta_j$ if needed, we can assume that the connected components \,$\Gamma_1, \ldots, \Gamma_s$\, of $\Gamma$ have vertices $\{u_1, \ldots u_{j_1}\}, \{u_{j_1+1}, \ldots, u_{j_2}\}, \ldots, \{u_{j_{s-1}+1 }, \ldots, u_n\}$, respectively, for a subsequence $j_1, \ldots, j_{s-1}$ of $1, \ldots, n$. Let $k_t$ be the corresponding number of $\tilde{\tau}_{\alpha_j}$ involved in the points $u_i$ in each $\Gamma_t$. We then impose the following additional conditions: 
\begin{itemize}
\item for each $\Gamma_t$, $2v_t - e_t=2$, and \ 
\item $-\left(\Sigma_{j=j_{t-1}+1}^{j_{t}} a_j\right) + 2k_t = -1$, \, for every $t = 1, \ldots, s$ . \
\end{itemize}
Observe that these two conditions translate to each connected component of the corresponding multisection to be a $2$-sphere and of self-intersection $-1$, respectively. Isolated vertices of $\Gamma_W$ amount to exceptional \textit{sections}, which can be blown-down to a pencil.

Let $G(W)$ be the quotient of $\pi_1(\Sigma_g)$ by $N(c_1, \ldots, c_l)$, the subgroup normally generated by $c_i$, and denote by $b_1(W)$ the first Betti number of $G$. 
Let $\sigma(W)$ be the signature of the image of the positive word $t_{c_1} \cdot \ldots \cdot t_{c_l}$ in $\Mod(\Sigma_{g})$ under the boundary capping homomorphism, and  $\eu(W)=4-4g+l$ be the associated Euler characteristic.  We can then set $b^+(W)= \frac{1}{2}(\eu(W)+2b_1(W)-2+ \sigma(W))$.

We obtain a characterization of monodromy factorizations of symplectic Calabi-Yau Lefschetz fibrations:

\begin{corollary} \label{SCYcor}
Let $W$ be a factorization in the framed mapping class group for \linebreak $g \geq 2$, such that either $G(W)$ is not a surface group, or $G(W)=1$ but $b^+(W) \neq 1$. If the associated graph $\Gamma_W$ satisfies the properties listed above, and in addition, has at least one isolated vertex, then the reduced word $t_{c_l} \cdots t_{c_1}$ is a monodromy factorization of a symplectic Calabi-Yau Lefschetz pencil. Conversely, on any symplectic Calabi-Yau $4$-manifold, one can find a Lefschetz pencil with a monodromy lift like such $W$.
\end{corollary}
\noindent As demonstrated by our examples in Section~\ref{Stipsicz}, one can also have SCY Lefschetz fibrations with \textit{no} exceptional sections. The first direction of the corollary can be extended to include such examples, too, provided a little care is given to the calculation of $G(W)$ (and thus $b_1(W)$) if no other lifts with pure sections are known.

Motivated by the \textit{conjectural} smooth classification of symplectic Calabi-Yau \linebreak $4$-manifolds, we can formulate a parallel problem for groups that can possibly be fundamental groups of SCYs \cite{FriedlVidussi}:

\begin{question}[Symplectic Calabi-Yau groups via mapping class factorizations] \label{SCYconjecture}
For any positive factorization $W$ of the boundary multitwist as in the corollary, is it always the case that $G(W)=1$, $\Z/2\Z$, or an infrasolvmanifold\footnote{which covers all other known SCYs with $b_1 \neq 0$ \cite{FriedlVidussi}} fundamental group? 
\end{question}

\noindent A negative answer to this question amounts to the existence of \textit{new} symplectic Calabi-Yaus. Whereas for a positive answer, since it suffices to work with pencils on minimal SCYs, one can restrict to positive factorizations $W$ in $\Mod_{\partial \Sigma_g^{2g-2}}(\Sigma_{g}^{2g-2})$ with no $\tau_{\alpha_i}$ on the left and all $a_i=1$ on the right side of the Equality~\ref{scyfactorization}. Thus, understanding all possible SCY groups is equivalent to understanding $G(W)$, where $W$ runs through all possible positive Dehn twist factorizations of the boundary multi-twist $t_{\delta_1} \cdots t_{\delta_{2g-2}}$.

Constructing mapping class group factorizations of the boundary multitwist in $\Mod(\Sigma_{g}^{2g-2};U)$ however is a rather challenging task in general. The next two subsections will demonstrate two successful cases: we will construct Lefschetz pencils on symplectic Calabi-Yau $\K$ and Enriques surfaces with \emph{explicit monodromy factorizations}, respectively. These will serve as sources of various interesting fibrations we will derive from them via surgical operations in the later sections of our paper. What is of key importance here is the \emph{special configurations of Lefschetz vanishing cycles} in the factorizations we get, and to produce them we will appeal to several symmetries of surfaces and lift better known mapping class relations on spheres or tori with boundaries to higher genera surfaces. 

\vspace{0.15cm}
\subsection{A genus-$3$ pencil on a symplectic Calabi-Yau $\K$ surface} \ 

We now construct an explicit monodromy for a genus-$3$ Lefschetz fibration with exactly $4$ disjoint $(-1)$-sphere sections on a $4$ times blown-up symplectic Calabi-Yau $\K$ surface, thus a pencil on a symplectic Calabi Yau $4$-manifold homeomorphic to the $\K$ surface. The equivalent monodromy factorization we derive at the end will be used in further constructions, and is the main motivation for us to go after this particular factorization.

\begin{lemma}\label{lem_relation_genus3}

The following relation holds in $\Mod_{\partial \Sigma_3^4}(\Sigma_3^4)$: 
\begin{equation}\label{eq_relation_genus3}
(t_{c_1}t_{c_7}t_{c_3}t_{c_5}t_{c_2}t_{c_6}t_{a_1}t_{a_2}t_{b_1}t_{b_2}t_{c_1}t_{c_7}t_{c_3}t_{c_5}t_{b_1}t_{b_2}t_{c_2}t_{c_6})^2 = t_{\delta_1}t_{\delta_2}t_{\delta_3}t_{\delta_4},
\end{equation}
where $a_i, b_j, c_k, \delta_l\subset \Sigma_3^4$ are simple closed curves shown in Figure~\ref{fig_scc_genus3}. 
\begin{figure}[htbp]
\begin{center}
\includegraphics[width=115mm]{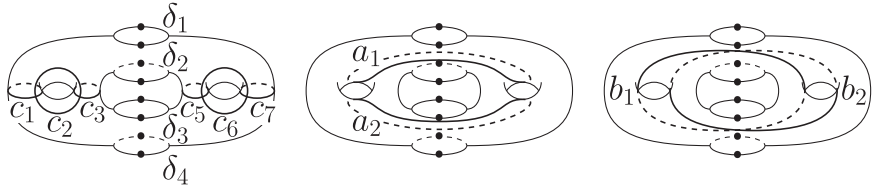}
\end{center}
\caption{Simple closed curves in $\Sigma_3^4$. The curve $\delta_i$ is parallel to a boundary component. }
\label{fig_scc_genus3}
\end{figure}

\end{lemma}

\begin{proof}
Let $\lambda$ be an involution on $\Sigma_3^4$ as shown in Figure~\ref{fig_lantern_relation_quotient}. 
\begin{figure}[htbp]
\includegraphics[width=60mm]{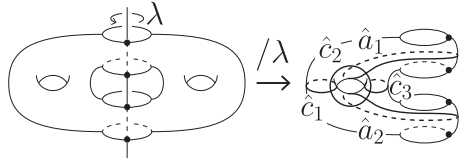}
\caption{The quotient map by an involution $\lambda$. }
\label{fig_scc_genus3_quotient}
\end{figure}
The quotient space $\Sigma_3^4$ is diffeomorphic to the surface $\Sigma_1^4$. 
We describe the images of curves under $/\lambda$ by hatted symbols. 
Using the relation given in \cite[Section 3.4]{Korkmaz_Ozbagci_2008} we obtain the following relation in $\Mod_{\partial\Sigma_1^4}(\Sigma_1^4)$: 
{\allowdisplaybreaks
\begin{align}
& \left(t_{\hat{c}_1}t_{\hat{c}_3}t_{\hat{c}_2}t_{\hat{a}_1}t_{\hat{a}_2}t_{\hat{c}_2}\right)^2 = t_{\hat{\delta}_1}t_{\hat{\delta}_2}t_{\hat{\delta}_3}t_{\hat{\delta}_4} \nonumber\\
\Leftrightarrow & \left(t_{\hat{c}_1}t_{\hat{c}_3}t_{\hat{c}_2}t_{\hat{a}_1}t_{\hat{a}_2}t_{\hat{c}_2}t_{\hat{c}_1}t_{\hat{c}_3}t_{\hat{c}_2}t_{\hat{a}_1}t_{\hat{a}_2}t_{\hat{c}_2}\right)^2 = t_{\hat{\delta}_1}^2t_{\hat{\delta}_2}^2t_{\hat{\delta}_3}^2t_{\hat{\delta}_4}^2 \nonumber\\
\Leftrightarrow & \left(t_{\hat{c}_1}t_{\hat{c}_3}t_{\hat{c}_2}t_{\hat{a}_1}^2t_{\hat{a}_2}^2t_{\hat{b}_1}t_{\hat{c}_1}t_{\hat{c}_3}t_{\hat{b}_1}t_{\hat{c}_2}\right)^2 = t_{\hat{\delta}_1}^2t_{\hat{\delta}_2}^2t_{\hat{\delta}_3}^2t_{\hat{\delta}_4}^2,  \label{eq_genus1_quotient}
\end{align}
}
\noindent where the last relation holds since $t_{\hat{b}_1}$ is equal to $(t_{\hat{a}_1}t_{\hat{a}_2})^{-1}t_{\hat{c}_2}t_{\hat{a}_1}t_{\hat{a}_2}$. 
The quotient map $/\lambda:\Sigma_3^4 \to \Sigma_1^4$ induces the following homomorphism: 
\[
\lambda_\ast : \pi_0(C_{\partial \Sigma_3^4}(\Sigma_3^4;\lambda)) \to \Mod_{\partial \Sigma_1^4}(\Sigma_1^4). 
\]
By Lemma~\ref{lem_involutionMCG} this homomorphism is injective. 
Furthermore it is not hard to see that the following equalities hold: 
{\allowdisplaybreaks
\begin{align*}
t_{\hat{c}_1} = \lambda_\ast(t_{c_1}t_{c_7}), \hspace{.4em}t_{\hat{c}_2} = \lambda_\ast(t_{c_2}t_{c_6}), \hspace{.4em} t_{\hat{c}_3} = \lambda_\ast(t_{c_3}t_{c_5}), \\
t_{\hat{a}_i}^2 = \lambda_\ast(t_{a_i}), \hspace{.4em} t_{\hat{b}_1} = \lambda_\ast(t_{b_1}t_{b_2}), \hspace{.4em}t_{\hat{\delta}_j}^2 = \lambda_\ast(t_{\delta_j}). 
\end{align*}
}
Thus we can obtain the relation \eqref{eq_relation_genus3} using a homomorphism 
\[\pi_0(C_{\partial \Sigma_3^4}(\Sigma_3^4;\lambda)) \to \Mod_{\partial \Sigma_3^4}(\Sigma_3^4) \]
induced by the inclusion $C_{\partial \Sigma_3^4}(\Sigma_3^4;\lambda) \hookrightarrow \Diff^+_{\partial \Sigma_3^4}(\Sigma_3^4)$. 
This completes the proof of Lemma \ref{lem_relation_genus3}. 
\end{proof}

The relation \eqref{eq_relation_genus3} gives rise to a genus-$3$ Lefschetz fibration $(X,f)$ over the \linebreak $2$-sphere with four disjoint $(-1)$-sphere sections.

\begin{proposition} \label{lem_signature_genus3LF}
$(X,f)$ is a genus-$3$ symplectic Calabi-Yau Lefschetz fibration, where the minimal model of $X$ is homeomorphic to a $\K$ surface.
\end{proposition}

\begin{proof}

The topological invariants of $X$ can be calculated using the monodromy factorization. The Euler characteristic of $X$ is the easiest to derive from $\eu(X) = 4g-4+m = 28$, where $g=3$ is the genus of the fibration and $m=36$ is the number of critical points. On the other hand, as the fibration $f \colon X \to S^2$ has a section, $\pi_1(X) \cong\pi_1(\Sigma_3)/\mathcal{C}$, where $\mathcal{C}$ is the normal subgroup of $\pi_1(\Sigma_3)$ generated by the vanishing cycles of $f$. The subcollection $c_1, c_2,c_5, c_6, c_7, b_1$ of vanishing cycles of $f$, taken with base points on the fiber in a straightforward fashion, generates the group $\pi_1(\Sigma_3)$. It follows that $\mathcal{C} = \pi_1(\Sigma_3)$, and in turn, $\pi_1(X)=1$. 

The signature calculation is more involved, and will constitute the rest of the proof. Note that if we knew $b^+(X)>1$ at this point (or simply that $X$ was not a rational surface) then Theorem~\ref{SCYLF} would imply that $(X,f)$ is a symplectic Calabi-Yau Lefschetz fibration which is a $4$ times blow-up of its minimal model. Since it should then have a minimal model which has the rational type of $\K$ surface, by Freedman's topological classification, the minimal model of $X$ should be homeomorphic to $\K$. However, we will be able to derive the simple, yet essential conclusion $b^+(X)>1$ after coupling our signature calculation to follow with our knowledge of $\eu(X)$ and vanishing $b_1(X)$.

Let $S_i \subset X$, $i=1,\ldots, 4$ denote the disjoint sections of self-intersection $-1$ corresponding to boundary components parallel to $\delta_i$ in the factorization in Lemma\ref{lem_relation_genus3}. Each Dehn twist on the left side of~\eqref{eq_relation_genus3} corresponds to a Lefschetz singularity of $f$. We take two disks $D_1, D_2\subset S^2$ which contain two consecutive critical values of $f$ corresponding to the elements $t_{a_1}t_{a_2}$. We also take disk neighborhoods $S_1,\ldots,S_{32}\subset S^2$ of the other critical values of $f$ and a small disk $D_0\subset S^2$ away from critical values of $f$. We decompose the surface $S^2 \setminus (\Int(D_0)\amalg_i \Int(S_i) \amalg_j \Int(D_j))$ into pants $P_1,\ldots, P_{33}$ (which are surfaces diffeomorphic to $\Sigma_0^3$) as follows: 
\begin{itemize}

\item $\partial P_1$ contains $\partial S_1$ and $\partial S_2$. 
We denote the circle $\partial P_1 \setminus (\partial S_1 \cup \partial S_2)$ by $L_1$;

\item $\partial P_i$ contains $L_{i-1}$ and $\partial S_{i+1}$ for each $i=2,\ldots, 31$. 
We denote the circle $\partial P_i \setminus (L_{i-1} \cup \partial S_{i+1})$ by $L_i$; 

\item $\partial P_j$ contains $L_{j-1}$ and $\partial D_{j-31}$ for $j=32,33$. 
We denote the circle $\partial P_j \setminus (L_{j-1} \cup \partial D_{j-31})$ by $L_j$. 

\end{itemize}
We can deduce the following equality by the Novikov additivity: 
{\allowdisplaybreaks
\begin{align*}
\sigma(X) = & \sigma\left(f^{-1}\left(S^2 \setminus (\Int(D_0)\amalg_i \Int(S_i) \amalg_j \Int(D_j))\right)\right) + \sigma(f^{-1}(D_0)) \\
& + \sum_{i=1}^{32} \sigma(f^{-1}(S_i)) + \sum_{j=1}^{2} \sigma(f^{-1}(D_j)) \\
= & \sigma\left(f^{-1}\left(S^2 \setminus (\amalg_i \Int(S_i) \amalg_j \Int(D_j))\right)\right) -2 \\
= & \sum_{i=1}^{33} \sigma(f^{-1}(P_i)) -2,  
\end{align*}
}
where the second equality holds since we have $\sigma(f^{-1}(D_0)) = \sigma(f^{-1}(S_i)) = 0$ and $\sigma(f^{-1}(D_j)) = -1$. 

Let $\iota$ be an involution of $\Sigma_3$ shown in Figure~\ref{fig_hyp_inv}. 

\begin{figure}[htbp]
\begin{center}
\includegraphics[width=40mm]{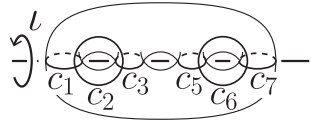}
\end{center}
\caption{The involution $\iota$. }
\label{fig_hyp_inv}
\end{figure}

We denote by $\mathcal{H}_3$ the image of the homomorphism $\pi_1(C(\Sigma_3;\iota)) \to \Mod(\Sigma_3)$ induced by the inclusion $C(\Sigma_g;\iota) \hookrightarrow \Diff^+(\Sigma_3)$. 
Let $\phi_3\colon \mathcal{H}_3\to \frac{\mathbb{Z}}{7}$ be a class function given in \cite[Proposition 3.1]{Endo_2000}. 
For a simple closed curve $L \subset S^2\setminus \Crit(f)$ we denote the monodromy along $L$ by $\rho(L)$, which is a conjugacy class of an element in $\Mod(\Sigma_3)$. 
By the configuration of vanishing cycles of $f$ we can regard $\rho(\partial S_i)$, $\rho(\partial D_j)$ and $\rho(L_k)$ as conjugacy classes in $\mathcal{H}_3$. 
We obtain the following equality by \cite[Satz 1]{Meyer} and \cite[Proposition 3.1]{Endo_2000}: 
{\allowdisplaybreaks
\begin{align*}
\sigma(X) = & \phi_3(\rho(L_1)) - \phi_3(\rho(\partial S_1)) - \phi_3(\rho(\partial S_2)) \\
& + \sum_{i=2}^{31} \left( \phi_3(\rho(L_i)) - \phi_3(\rho(\partial S_{i+1})) - \phi_3(\rho(L_{i-1})) \right) \\
& + \sum_{j=32}^{33} \left( \phi_3(\rho(L_j)) - \phi_3(\rho(\partial D_{j-31})) - \phi_3(\rho(L_{j-1})) \right) -2 \\
= & -\sum_{i=1}^{31} \phi_3(\rho(\partial S_i)) - \sum_{j=1}^{2}\phi_3(\rho(\partial D_j)) + \phi_3(\rho(L_{33})) -2.  
\end{align*}
}
Since the curve $L_{33}$ is homotopic to $\partial D_0$ the monodromy $\rho(L_{33})$ is trivial. 
In particular $\phi_3(\rho(L_{33}))$ is equal to $0$. 
The monodromy $\rho(\partial S_i)$ is a Dehn twist along non-separating simple closed curve in $\Sigma_3$. 
Thus the value $\phi_3(\rho(\partial S_i))$ is equal to $\frac{4}{7}$ (see \cite[Lemma 3.3]{Endo_2000}). 
Since the monodromy $\rho(\partial D_1)$ is equal to $\rho(\partial D_2)$, the value $\phi_3(\rho(\partial D_1))$ is equal to $\phi_3(\rho(\partial D_2))$. 

The manifold $S^2\times T^2 \sharp 8 \overline{\mathbb{CP}^2}$ admits a genus-$3$ Lefschetz fibration $f$ with the following properties (see \cite{Korkmaz_2001}):
\begin{itemize}
\item $f$ has $16$ Lefschetz singularities with non-separating vanishing cycles; 

\item a monodromy factorization of $f$ contains four pairs of Dehn twists along bounding pairs; 

\item the other eight vanishing cycles are preserved by an involution $\iota$. 

\end{itemize}
Since the function $\phi_3$ is preserved by conjugation we obtain the following equality: 
{\allowdisplaybreaks
\begin{align*}
& -8 = \sigma(S^2\times T^2 \sharp 8 \overline{\mathbb{CP}^2}) = 8\left(-\frac{4}{7}\right) + 4\phi_3(\rho(\partial D_1))-4 \\
\Rightarrow & \phi_3(\rho(\partial D_1)) = \frac{1}{4} \left( -8 +4 + \frac{32}{7} \right) = \frac{1}{7}. 
\end{align*}
}
As a result we can calculate the signature of $X$ as 
{\allowdisplaybreaks
\begin{align*}
\sigma(X) = & -31\cdot \frac{4}{7} -2\cdot \frac{1}{7} -2 = -20.
\end{align*}
}
%
%
We can now use 
\[ b^+(X)-b^-(X)=\sigma(X)=-20 \, \text{and} \, 2-2b_1(X)+b^+(X)+b^-(X)= \eu(X) = 28 \, , \]
where $b_1(X)=0$, to conclude that $b^+(X)=3$. Hence, per our discussion preceding the signature calculation, $(X,f)$ is a symplectic Calabi-Yau Lefschetz fibration, where the minimal model of $X$ is homeomorphic to a $\K$ surface.
\end{proof}

\vspace{0.15cm}
\subsection{A genus-$2$ pencil on a symplectic Calabi-Yau Enriques surface} \

We now construct a genus-$2$ Lefschetz fibration with exactly two disjoint \linebreak $(-1)$-sphere sections on a $2$ times blown-up symplectic Calabi-Yau Enriques surface, that is, a genus-$2$ Lefschetz pencil on a symplectic Calabi-Yau $4$-manifold homeomorphic to the Enriques surface.

\begin{lemma}\label{lem_lift_matsumoto}

The following relation holds in $\Mod_{\partial \Sigma_2^2}(\Sigma_2^2)$: 
\[
(t_{d_4}t_{d_3}t_{d_2})^2 t_{d_+}t_{d_-} = t_{\delta_1}t_{\delta_2}, 
\]
where $d_\ast, \delta_i\subset \Sigma_2^2$ are simple closed curves shown in Figure~\ref{fig_scc_matsumoto}. 
\begin{figure}[htbp]
\includegraphics[width=70mm]{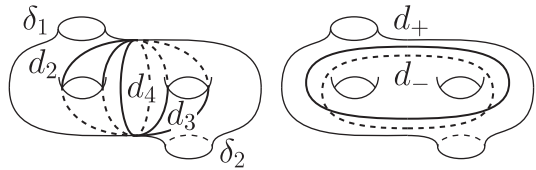}
\caption{Simple closed curves in $\Sigma_2^2$.}
\label{fig_scc_matsumoto}
\end{figure} 

\end{lemma}

\begin{proof}

We regard $\Sigma_2^4$ as a subsurface of $\Sigma_2^2$ in the obvious way. 
We take simple closed curves $d_1, \Gamma_i\subset \Sigma_2^4$, a point $u_i\in \partial \Sigma_2^4$ and involutions $\kappa$ and $\iota$ as shown in Figure~\ref{fig_scc_matsumoto_quotient}.

\begin{figure}[htbp]
\includegraphics[width=110mm]{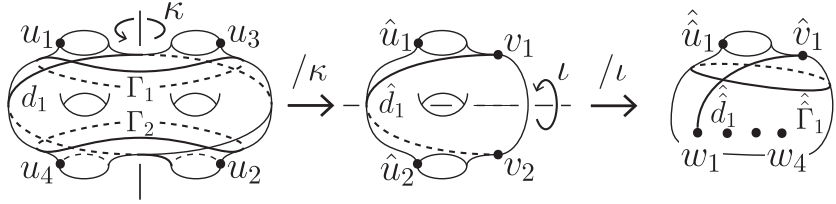}
\caption{Symmetries on surfaces. }
\label{fig_scc_matsumoto_quotient}
\end{figure}
Denote a simple closed curve in $\Sigma_2^4$ parallel to the component of $\partial\Sigma_2^4$ containing $u_i$ by $\delta_i$ and fixed points of $\kappa$ (resp.~$\iota$) by $v_i$ (resp.~$w_i$). 
We add a hat to any symbols to describe the image of quotient maps. 
Let $\alpha_i\subset \Sigma_0^1\setminus \{w_1,\ldots,w_4\}$ be a loop with the base point $\hat{v}_1$ obtained by connecting $\hat{\hat{d}}_i$ with a counterclockwise circle around $w_i$. 
The following equality holds in $\Mod_{\partial \Sigma_0^1}(\Sigma_0^1, \{w_1,\ldots,w_4\})$: 
\[
\Push(\alpha_4)\Push(\alpha_3)\Push(\alpha_2)\Push(\alpha_1) = t_{\hat{\hat{\Gamma}}_1}^{-1} t_{\hat{\hat{\delta}}_1}. 
\]
By Lemma~\ref{lem_involutionMCG}, the following homomorphism is injective:
\[
\iota_\ast \colon \pi_0(C_{\partial \Sigma_1^2}(\Sigma_1^2;\iota)) \to \Mod_{\partial \Sigma_0^1}(\Sigma_0^1;\{w_1,\ldots,w_4 \}).
\]
It is easy to verify (by using Alexander's lemma, for example) the following equalities holds: 
\[
\Push(\alpha_i)=\iota_\ast(\tau_{\hat{d}_i}), \hspace{.3em} t_{\hat{\hat{\Gamma}}_1} = \iota_\ast(t_{\hat{\Gamma}_1}t_{\hat{\Gamma}_2}), \hspace{.3em}
t_{\hat{\hat{\delta}}_1} = \iota_\ast (t_{\hat{\delta}_1}t_{\hat{\delta}_2}). 
\]
Thus we obtain the following relation in $\pi_0(C_{\partial \Sigma_1^2}(\Sigma_1^2;\iota))$: 
\begin{equation}\label{eq_genus1_matsumoto}
\tau_{\hat{d}_4}\tau_{\hat{d}_3}\tau_{\hat{d}_2}\tau_{\hat{d}_1} = t_{\hat{\Gamma}_1}^{-1}t_{\hat{\Gamma}_2}^{-1} t_{\hat{\delta}_1}t_{\hat{\delta}_2}. 
\end{equation}
Using the inclusion $C_{\partial \Sigma_1^2}(\Sigma_1^2;\iota) \hookrightarrow \Diff^+_{\partial \Sigma_1^2}(\Sigma_1^2;\{v_1,v_2\})$ we obtain the same relation as \eqref{eq_genus1_matsumoto} in $\Mod_{\partial \Sigma_1^2}(\Sigma_1^2;\{v_1,v_2\})$. 

The involution $\kappa$ induce the following homomorphism which is injective by Lemma~\ref{lem_involutionMCG}: 
\[
\kappa_\ast \colon\pi_0(C_{\partial\Sigma_2^4}(\Sigma_2^4;\kappa)) \to \Mod_{\partial\Sigma_1^2}(\Sigma_1^2;\{v_1,v_2\}). 
\]
It is easy to see that the images of the mapping classes $t_{d_i}, t_{\Gamma_i}$ and $t_{\delta_i}t_{\delta_{i+2}}$ under the homomorphism $\kappa_\ast$ are $\tau_{\hat{d}_i}, t_{\hat{\Gamma}_i}^2$ and $t_{\hat{\delta}_i}$, respectively. 
Thus we obtain the following relation in $\pi_0(C_{\partial \Sigma_2^4}(\Sigma_2^4; \kappa))$: 
{\allowdisplaybreaks
\begin{align}
& (t_{d_4}t_{d_3}t_{d_2}t_{d_1})^2 = t_{\Gamma_1}^{-1}t_{\Gamma_2}^{-1}t_{\delta_1}^2t_{\delta_2}^2t_{\delta_3}^2t_{\delta_4}^2 \nonumber \\
\Leftrightarrow & (t_{d_4}t_{d_3}t_{d_2})^2 ((t_{d_4}t_{d_3}t_{d_2})^{-1}t_{d_1}(t_{d_4}t_{d_3}t_{d_2})) t_{d_1} = t_{\Gamma_1}^{-1}t_{\Gamma_2}^{-1}t_{\delta_1}^2t_{\delta_2}^2t_{\delta_3}^2t_{\delta_4}^2. \label{eq_genus2_matsumoto} 
\end{align}
}
We can obtain the same relation as \eqref{eq_genus2_matsumoto} via the inclusion $C_{\partial \Sigma_2^4}(\Sigma_2^4;\kappa) \hookrightarrow \Diff^+_{\partial \Sigma_2^4}(\Sigma_2^4)$. 
Let $C\colon \Mod_{\partial \Sigma_2^4}(\Sigma_2^4)\to \Mod_{\partial \Sigma_2^2}(\Sigma_2^2)$ be a homomorphism obtained by capping the components of $\partial \Sigma_2^4$ containing $u_3$ and $u_4$ with punctured disks. 
It is easy to see that the image of the left (resp.~right) side of the equality~\eqref{eq_genus2_matsumoto} under $C$ is equal to the left (resp.~right) side of the equality in the statement. 
This completes the proof of Lemma~\ref{lem_lift_matsumoto}. 
\end{proof}

Let $c_2, c_3, c_4\subset \Sigma_2^2$ be simple closed curves shown in Figure~\ref{fig_scc_genus2surface_2}. 
\begin{figure}[htbp]
\includegraphics[width=30mm]{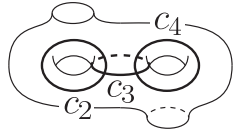}
\caption{Simple closed curves in $\Sigma_2^2$.}
\label{fig_scc_genus2surface_2}
\end{figure}
By the chain relation (see \cite[Proposition 4.12]{Farb_Margalit_2011}) and Lemma~\ref{lem_lift_matsumoto} we obtain the following relation in $\Mod_{\partial \Sigma_2^2}(\Sigma_2^2)$: 
\[
(t_{d_4}t_{d_3}t_{d_2})^2 (t_{c_2}t_{c_3}t_{c_4})^4 = t_{\delta_1}t_{\delta_2}. 
\]
This relation gives rise to a genus-$2$ Lefschetz fibration $h\colon Y \rightarrow S^2$ with two \linebreak $(-1)$-sphere sections.

\begin{proposition}\label{lem_pi1_genus2LF}
$(Y,g)$ is a genus-$2$ symplectic Calabi-Yau Lefschetz fibration, where the minimal model of $Y$ is homeomorphic to an Enriques surface.
\end{proposition}

\begin{proof}

Since $h$ has $18$ Lefschetz singularities, $\eu(Y)=14$. Thanks to genus-$2$ mapping class group being hyperelliptic, we can also calculate $\Sign(Y)$ easily by using Matsumoto's signature formula \cite[Theorem 3.3(2), Proposition 3.6]{Matsumoto3} as 
\[
\Sign(Y) = 16\left(-\frac{3}{5}\right) + 2\left(-\frac{1}{5}\right) = -10. 
\]

The calculation that is more involved this time is that of $\pi_1(Y)$, since the vanishing cycles of $(Y,h)$ do not kill all the generators of the fundamental group of the fiber, $\pi_1(\Sigma_2)$. 
%
To calculate $\pi_1(Y)$, let us take oriented based loops $\alpha_i, \beta_i$ in $\Sigma_2$ as shown in Figure~\ref{fig_generator_genus2_pi1}. 
\begin{figure}[htbp]
\begin{center}
\includegraphics[width=40mm]{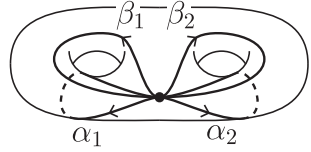}
\end{center}
\caption{A generator of $\pi_1(\Sigma_2)$. }
\label{fig_generator_genus2_pi1}
\end{figure}
We also use the symbols $\alpha_i, \beta_i$ to represent the homotopy classes of them. 
The fundamental group $\pi_1(\Sigma_2)$ has the presentation $\left<\alpha_1,\beta_1,\alpha_2,\beta_2~|~[\alpha_1,\beta_1][\beta_2,\alpha_2]\right>$. 
The curves $c_2, c_3$ and $c_4$ are homotopic to the curves $\beta_1, \alpha_1\overline{\alpha_2}$ and $\beta_2$, respectively, where $\overline{\alpha_2}$ is the based loop $\alpha_2$ with the opposite orientation. 
It is easy to see that the curves $d_2, d_3$ and $d_4$ are homotopic to $\alpha_1\overline{\beta_1}\beta_2\alpha_2, \alpha_1\beta_2\alpha_2\overline{\beta_2}$ and $[\alpha_1,\beta_1]$, respectively. 
Since the fibration $h$ has a section the fundamental group of $Y$ is calculated as follows:
{\allowdisplaybreaks
\begin{align*}
\pi_1(Y) & \cong \pi_1(\Sigma_2)/\left<c_2,c_3,c_4,d_2,d_3,d_4\right> \\
& \cong \left<\alpha_1,\beta_1,\alpha_2,\beta_2 \left|[\alpha_1,\beta_1][\beta_2,\alpha_2],\beta_1, \alpha_1\overline{\alpha_2}, \beta_2, \alpha_1\overline{\beta_1}\beta_2\alpha_2, \alpha_1\beta_2\alpha_2\overline{\beta_2},[\alpha_1,\beta_1] \right.\right> \\
& \cong \left<\alpha_1,\alpha_2 \left|\alpha_1\overline{\alpha_2}, \alpha_1\alpha_2 \right.\right> \\
& \cong \left<\alpha \left|\alpha^2 \right.\right> \cong \mathbb{Z}/2\mathbb{Z}. 
\end{align*}
}
Now, since $Y$ is not rational or ruled, it is a symplectic Calabi-Yau, and has only one exceptional sphere. Blowing it down, we arrive at the minimal symplectic $Y'$. Since its signature  is not divisible by $16$, $w_2(Y')$ does not vanish. 
It is easily seen that the universal cover $\widetilde{Y'}$ of $Y'$ has a torsion canonical class, and is trivial since $H^2(\widetilde{Y'})$ has no torsion. Since the modulo $2$ reduction of the canonical class coincides with the second Stiefel-Whitney class, $w_2(\widetilde{Y'})$ vanishes. 
Thus, $Y'$ has the same $w_2$--type \cite{HambletonKreck_1988} as that of the Enriques surface. By \cite[Theorem C]{HambletonKreck}, we conclude that $Y'$ is homeomorphic to the Enriques surface.

\end{proof}

As we are unable to detect whether the total spaces of our pencils are \textit{diffeomorphic} to the $\K$ and \textit{the} Enriques surfaces, we finish with highlighting this question: 

\begin{question}
Are the symplectic Calabi-Yau manifolds $X$ and $Y$ we have constructed above diffeomorphic to blow-ups of the $\K$ and the Enriques surfaces, respectively? 
\end{question}


\vspace{0.2in}
\section{Fiber sum indecomposability and Stipsicz's conjecture} \label{Stipsicz} 

A common way of constructing new Lefschetz fibrations from given ones is the \emph{fiber sum} operation, defined as follows: Let $(X_i, f_i)$, $i=1,2$, be genus-$g$ Lefschetz fibrations with regular fiber $F$. The \emph{fiber sum} \, $(X_1, f_1)\#_{F, \Phi}(X_2,f_2)$ is a genus-$g$ Lefschetz fibration obtained by removing a fibered tubular neighborhood of a regular fiber from each $(X_i, f_i)$ and then identifying the resulting boundaries via a fiber-preserving, orientation-reversing diffeomorphism $\Phi$. In terms of monodromy factorizations, this translates to a monodromy factorization with a proper subfactorization of the identity in the mapping class group of $F$. A Lefschetz fibration $(X,f)$ is called \emph{fiber sum indecomposable} (or \emph{indecomposable} in short) if it cannot be expressed as a fiber sum of any two nontrivial Lefschetz fibrations. These can be regarded as prime building blocks of Lefschetz fibrations.

Stipsicz in \cite{Stipsicz}, and Smith in \cite{Smith} independently proved that Lefschetz fibrations admitting $(-1)$-sphere sections are fiber sum indecomposable. Moreover, Stipsicz conjectured in the same work that the converse is also true, that is, every fiber sum indecomposable Lefschetz fibration contains a $(-1)$-sphere section, an affirmative answer to which would suggest blow-ups of Lefschetz pencils as the elementary building blocks of Lefschetz fibrations through fiber sums. However, Sato showed that a genus-$2$ Lefschetz fibration constructed by Auroux on a once blown-up minimal symplectic $4$-manifold provided a counter-example to this conjecture, by showing that the exceptional class could only be represented by a $2$-section of this fibration \cite{Sato_2008}. 

Here we will show that Auroux's genus-$2$ fibration is not a mere exception, by constructing further counter-examples to Stipsicz's conjecture via explicit monodromy factorizations of Lefschetz fibrations with their multisections. To be able to detect how all exceptional classes lie with respect to a Lefschetz fibration $(X, f)$ (and that none is a section), we will work with Lefschetz fibrations on blow-ups of symplectic Calabi-Yau $4$-manifolds, which are the perfect fit to our purposes because of two reasons: any symplectic Calabi-Yau $(X,f)$ is fiber sum indecomposable (\cite{BaykurDecomp, Usher}), and, as we have reviewed earlier, each $2$-sphere $S_j$ representing an exceptional class is an $s_j$-section of $(X, f)$, where the \textit{inequality} $\sum s_j = (\sum S_j) \cdot F \leq 2g-2$ in this case is sharp. The strategy of our proof then boils down to constructing an explicit monodromy factorization for a symplectic Calabi-Yau $(X,f)$ detecting all exceptional multisections, and then applying monodromy substitutions which turn all exceptional classes to $s_j$-sections with $s_j \geq 2$ for all $j$. This monodromy substitution comes from a generalization of the lantern relation to the framed mapping class group, which we present next.

\vspace{0.15cm}
\subsection{A lantern relation for multisections}\label{braidedlantern} \

We generalize the lantern relation to that in the mapping class group with commutative boundary components: 

\begin{lemma}[Braiding lantern relation] \label{lem_lantern_relation}

Let the curves $a,b,c,d,x, \delta_1, \delta_2$, pairs of arcs $y,z$ and points $u_1, u_2$ in $\Sigma_0^6$ be as shown in Figure~\ref{fig_lantern_relation}, where $a,b,c,d,\delta_1,\delta_2$ are parallel to boundary components. Denote the boundary components parallel to $\delta_i$ by $S_i$. The following relation holds in $\Mod_{\partial \Sigma_0^6\setminus (S_1\sqcup S_2)}(\Sigma_0^6; \{u_1,u_2\})$: 
\[
\widetilde{t_z} t_{x} \widetilde{t_{y}} = t_a t_b t_c t_d t_{\delta_2}. 
\]
\begin{figure}[htbp]
\begin{center}
\includegraphics[width=35mm]{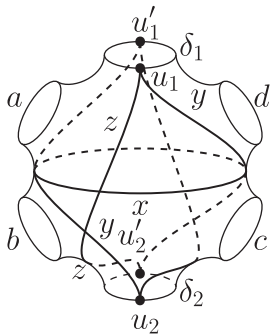}
\end{center}
\caption{Curves in $\Sigma_0^6$. }
\label{fig_lantern_relation}
\end{figure}
\end{lemma}
\begin{proof}
Let $\eta$ be an involution of $\Sigma_0^6$ defined as the $\pi$-degree rotation as shown in Figure~\ref{fig_lantern_relation_quotient}. 
\begin{figure}[htbp]
\begin{center}
\includegraphics[width=125mm]{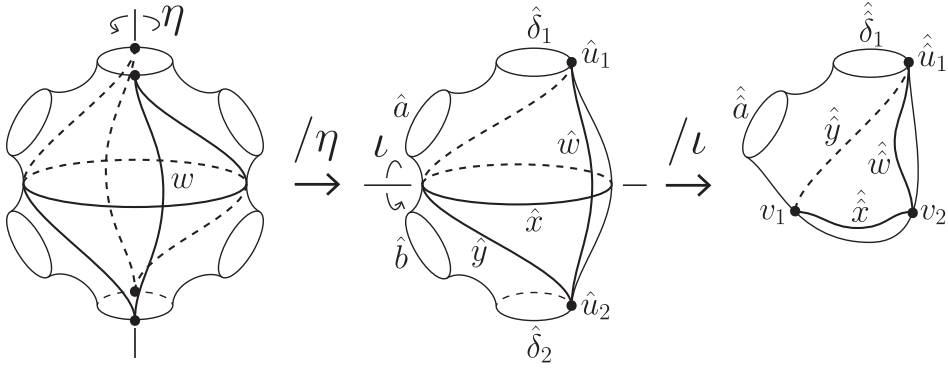}
\end{center}
\caption{The quotient map $/\eta:\Sigma_0^6 \to \Sigma_0^6/\eta \cong \Sigma_0^4$. }
\label{fig_lantern_relation_quotient}
\end{figure}
The quotient space $\Sigma_0^6/\eta$ is homeomorphic to the surface $\Sigma_0^4$. 
We take a pair of arcs $w$ in $\Sigma_0^6$ as shown in Figure~\ref{fig_lantern_relation_quotient}. 
Let $\hat{w}\subset \Sigma_0^4$ be the image of $w$ under the quotient map $/\eta$. 
We define other symbols in the same way (see Figure~\ref{fig_lantern_relation_quotient}). 

We further take an involution $\iota$ on $\Sigma_0^4$ as shown in Figure~\ref{fig_lantern_relation_quotient}. 
The quotient space $\Sigma_0^4/\iota$ is homeomorphic to the surface $\Sigma_0^2$. 
Denote the image of the fixed points of $\iota$ under the quotient map $/\iota: \Sigma_0^4\to \Sigma_0^2$ by $v_1, v_2$. 
We use double-hatted symbols to describe images of curves and arcs in $\Sigma_0^4$ under the map $/\iota$ (see Figure~\ref{fig_lantern_relation_quotient}).  
We denote by $\hat{\hat{S}}$ the boundary component of $\Sigma_0^2$ parallel to the curve $\hat{\hat{\delta}}_1$. 
Let $Y$ and $W$ be simple closed curves in $\Sigma_0^2$ which bound regular neighborhoods of the unions $\hat{\hat{y}}\cup \hat{\hat{S}}$ and $\hat{\hat{w}}\cup \hat{\hat{S}}$, respectively, and $X$ a simple closed curve which bounds the arc $\hat{\hat{x}}$. 
By the lantern relation \cite[Proposition 5.1]{Farb_Margalit_2011}, we obtain the following relation in $\Mod_{\partial \Sigma_0^2}(\Sigma_0^2; v_1, v_2)$: 
\begin{align}
& t_{W} t_{X} t_{Y} = t_{\hat{\hat{a}}}t_{\hat{\hat{\delta}}_1}. \label{eq_Lantern_relation_quotient1}
\end{align}
%
Consider the following homomorphism induced by the quotient map $/\iota:\Sigma_0^4 \to \Sigma_0^2$: 
\[
\iota_\ast : \pi_0(C_{\partial \Sigma_0^4\setminus (\hat{S}_1\sqcup \hat{S}_2)}(\Sigma_0^4;\iota)) \to \Mod_{\partial \Sigma_0^2}(\Sigma_0^2; v_1, v_2). 
\]
By Lemma~\ref{lem_involutionMCG} this homomorphism is injective. 
Since the paths $\hat{y}, \hat{w}$ and a loop $\hat{x}$ are invariant under $\iota$, we can regard $\tau_{\hat{y}}, \tau_{\hat{w}}$ and $t_{\hat{x}}$ as elements in $\pi_0(C_{\partial \Sigma_0^4\setminus (\hat{S}_1\sqcup \hat{S}_2)}(\Sigma_0^4;\iota))$, where $\tau_{\hat{y}}$ (resp. $\tau_{\hat{w}}$) is the half twist along $\hat{y}$ (resp. $\hat{w}$) interchanging the boundary components $\hat{S}_1$ and $\hat{S}_2$. 
It is not hard to see (by using the Alexander method \cite[Section 2.3]{Farb_Margalit_2011}, for example) that the following equalities hold: 
{\allowdisplaybreaks
\begin{align*}
\iota_\ast(\tau_{\hat{y}}) = t_{Y}, \hspace{.4em}
\iota_\ast(\tau_{\hat{w}}) = t_{W}, \hspace{.4em} 
\iota_\ast(t_{\hat{x}}^2) = t_{X}, \\
\iota_\ast(t_{\hat{a}}t_{\hat{b}}) = t_{\hat{\hat{a}}}, \hspace{.4em}
\iota_\ast(t_{\hat{\delta}_1}t_{\hat{\delta}_2}) = t_{\hat{\hat{\delta}}_1}. 
\end{align*}
} %
Thus we obtain the following relation in $\pi_0(C_{\partial \Sigma_0^4\setminus (\hat{S}_1\sqcup \hat{S}_2)}(\Sigma_0^4;\iota))$ from the relation \eqref{eq_Lantern_relation_quotient1}: 
\begin{equation}\label{eq_lantern_relation_quotient2}
\tau_{\hat{w}} t_{\hat{x}}^{2} \tau_{\hat{y}} = t_{\hat{a}}t_{\hat{b}} t_{\hat{\delta}_1} t_{\hat{\delta}_2}.
\end{equation}
Since the inclusion $C_{\partial \Sigma_0^4\setminus (\hat{S}_1\sqcup \hat{S}_2)}(\Sigma_0^4; \iota) \hookrightarrow \Diff^+_{\partial \Sigma_0^4\setminus (\hat{S}_1\sqcup \hat{S}_2)}(\Sigma_0^4; \{u_1,u_2\})$ induce a homomorphism $\pi_0(C_{\partial \Sigma_0^4\setminus (\hat{S}_1\sqcup \hat{S}_2)}(\Sigma_0^4; \iota))\to \Mod_{\partial \Sigma_0^4\setminus (\hat{S}_1\sqcup \hat{S}_2)}(\Sigma_0^4; \{u_1,u_2\})$, we can regard the relation \eqref{eq_lantern_relation_quotient2} as that in $\Mod_{\partial \Sigma_0^4\setminus (\hat{S}_1\sqcup \hat{S}_2)}(\Sigma_0^4; \{u_1,u_2\})$. 

We denote by $t_{\delta_i}^{1/2} \in \Mod_{\partial \Sigma_0^6\setminus (S_1\sqcup S_2)}(\Sigma_0^6; \{u_1,u_1^\prime, u_2, u_2^\prime\})$ the square root of the Dehn twist along $\delta_i$ which fixes outside of a neighborhood of $S_i$ and interchanges the points $u_i$ and $u_i^\prime$. 
In a way quite similar to that in the previous paragraph, we obtain the following relation in $\Mod_{\partial \Sigma_0^6\setminus (S_1\sqcup S_2)}(\Sigma_0^6; \{u_1,u_1^\prime, u_2, u_2^\prime\})$: 
{\allowdisplaybreaks
\begin{align*}
& \widetilde{t_{w}} t_{x} \widetilde{t_{y}} = t_a t_b t_c t_d t_{\delta_1}^{1/2} t_{\delta_2}^{1/2} \\
\Leftrightarrow & t_{\delta_2}^{1/2}\widetilde{t_{w}}  t_{x} \widetilde{t_{y}} t_{\delta_1}^{-1/2} = t_a t_b t_c t_d t_{\delta_2} \\
\Leftrightarrow & (t_{\delta_2}^{1/2}\widetilde{t_{w}}t_{\delta_2}^{-1/2})  t_{x} \widetilde{t_{y}} = t_a t_b t_c t_d t_{\delta_2} \\
\Leftrightarrow & \widetilde{t_{z}} t_{x} \widetilde{t_{y}} = t_a t_b t_c t_d t_{\delta_2}.
\end{align*}
}
This completes the proof of Lemma \ref{lem_lantern_relation}. 
\end{proof}

\begin{remark} 
As the braiding lantern relation allows us to perform a local substitution in a monodromy factorization, it can be used to pass from one Lefschetz fibration to another while braiding two given sheets of (multi)sections. If the boundary components $\delta_1, \delta_2$ correspond to two sections $S_1, S_2$ of self-intersections $s_1$ and $s_2$, the substitution will hand a new Lefschetz fibration with a $2$-section $S_{12}$ which is an embedded $2$-sphere of self-intersection $s_1+s_2+1$. In particular, if $S_1, S_2$ are exceptional classes, so is $S_{12}$, which will play a crucial role in our applications to follow.
\end{remark}

\begin{remark}
Let us sketch a toy example of our braiding lantern substitution: Consider the trivial rational fibration on $S^2 \x S^2$, and blow-up one of the fibers $4$ times so it now consists of a $(-4)$-sphere $V$ and $4$ exceptional spheres. Let $S_1, S_2$ be two disjoint self-intersection $0$ sections of this fibration on $S^2 \x S^2 \# 4 \CPb \cong \CP \#5 \CPb$, each intersecting $V$ once. The braiding lantern substitution along the vanishing cycles and the two boundary components for the sections $S_1, S_2$ amounts to rationally blowing-down $V$, and the result will be a new rational Lefschetz fibration on $\CP \# 4 \CPb$ with $3$ vanishing cycles and a $2$-sphere bi-section $S_{12}$ of self-intersection $+1$. The latter is equivalent to the blow-up of the degree-$2$ pencil on $\CP$, where $S_{12}$ is identified with $\mathbb{CP}^{1} \subset \CP$.
\end{remark}

\vspace{0.15cm}
\subsection{Genus--$3$ counter-examples to Stipsicz's conjecture}\label{genus3counter} \

In the previous section, we have obtained the following monodromy factorization for a genus-$3$ Lefschetz fibration with $4$ $(-1)$-sphere sections on a symplectic Calabi-Yau $\K$ surface: 
\begin{equation*}
(t_{c_1}t_{c_7}t_{c_3}t_{c_5}t_{c_2}t_{c_6}t_{a_1}t_{a_2}t_{b_1}t_{b_2}t_{c_1}t_{c_7}t_{c_3}t_{c_5}t_{b_1}t_{b_2}t_{c_2}t_{c_6})^2 = t_{\delta_1}t_{\delta_2}t_{\delta_3}t_{\delta_4},
\end{equation*}

We will now derive a factorization Hurwitz equivalent to this one. In fact, the motivation behind all of our construction in that section is indeed to arrive at this next configuration containing various lantern curves.

Denote by $y_1, y_2, z_1, z_2$ pairs of arcs shown in Figure~\ref{genus3_lantern}. 
\begin{figure}[htbp]
\begin{center}
\includegraphics[width=80mm]{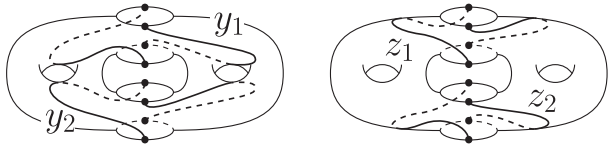}
\end{center}
\caption{Pairs of arcs in $\Sigma_3^4$. }
\label{genus3_lantern}
\end{figure}
By Lemma \ref{lem_lantern_relation}, we obtain the following relations in $\Mod(\Sigma_3^4;\{u_1,u_2,u_3,u_4\})$: 
\begin{align}
\widetilde{t_{z_1}} t_{a_1} \widetilde{t_{y_1}} t_{\delta_1}^{-1} &= t_{c_1}t_{c_3} t_{c_5} t_{c_7}, \label{genus3_relation1}\\
\widetilde{t_{z_2}} t_{a_2}\widetilde{t_{y_2}} t_{\delta_4}^{-1} &= t_{c_1}t_{c_3} t_{c_5} t_{c_7}. \label{genus3_relation2}
\end{align}

The monodromy factorization of $f$ given in Lemma \ref{lem_relation_genus3} can be then changed by elementary transformations as follows: 
{\allowdisplaybreaks
\begin{align}
t_{\delta_4}t_{\delta_3}t_{\delta_2}t_{\delta_1} &= \underbrace{t_{c_1}t_{c_7}t_{c_3}t_{c_5}}_{(a)}t_{c_2}t_{c_6}t_{a_1}t_{a_2}t_{b_1}t_{b_2}\underbrace{t_{c_1}t_{c_7}t_{c_3}t_{c_5}}_{(b)}t_{b_1}t_{b_2}t_{c_2}t_{c_6} \nonumber\\
& \cdot t_{c_1}t_{c_7}t_{c_5} t_{t_{c_3}(c_2)}t_{c_6}\underbrace{t_{c_3}t_{a_1}t_{a_2}t_{c_3}}_{(c)}t_{b_1}t_{t_{c_3}^{-1}(b_2)}t_{c_1}t_{c_7}t_{c_5}t_{b_1}t_{b_2}t_{c_2}t_{c_6} \label{genus3factor} 
\end{align}
}
We will denote this fibration as $(X_{(1,1,1,1)}, f_{(1,1,1,1)})$ in regards to the intersection numbers $n_j$ of the  exceptional classes with the fiber, which are all honest sections in this case.

Let $S_1, S_2, S_3, S_4$ be the exceptional sections of $(X,f)$. By Lemma~\ref{lem_lantern_relation} we can derive the following new symplectic Calabi-Yau Lefschetz fibrations: The monodromy substitution by the relation \eqref{genus3_relation1} at the part (a) gives rise to a new genus-$3$ Lefschetz fibration $(X_{(2,1,1)},f_{(2,1,1)})$ with a $2$-section $S_{12}$ derived from $S_1$ and $S_2$ and the two sections $S_3$, $S_4$ inherited from $(X,f)$. By Theorem \ref{mainthm}, $S_{12}$ is a $2$-sphere with self-intersection equal to $-1$. In fact, since a substitution by the lantern relation corresponds to a rational blow-down along $C_2$ \cite{Endo_Gurtas_2010}, $b^+(X_{(2,1,1)})=b^+(X)$ and $b^-(X_{(2,1,1)})= b^-(X)-1$. Note that we could also apply the lantern substitution along (c) to arrive at a similar Lefschetz fibration. 

We can further apply the substitution by the relation \eqref{genus3_relation2} at the part (c) to $(X_{(2,1,1)},f_{(2,1,1)})$. This substitution now gives rise to a genus-$3$ Lefschetz fibration $(X_{(2,2)},f_{(2,2)})$ with two $(-1)$-spheres $S_{12}$ and $S_{34}$ that are $2$-sections. Once again, $b^+(X_{(2,2)})=b^+(X)$ and $b^-(X_{(2,2)})= b^-(X)-2$. We have therefore obtained a fiber sum indecomposable Lefschetz fibration on $X_{(2,2)}$, which is not a rational or ruled surface, where there are no other exceptional classes other than $S_{12}$ and $S_{34}$, and thus, there are no $(-1)$ sphere sections due to formula $\sum n_j = (\sum S_j) \cdot F = 4$. 
 
We can further apply a substitution at (b) to arrive at a symplectic Calabi-Yau Lefschetz fibration $(X_{(4)},f_{(4)})$
with a single exceptional class represented by a sphere $4$-section $S_{1234}$. By the same arguments as above, this is a fiber sum indecomposable fibration without any $(-1)$ sphere sections. Note that we could change the order of substitutions and braid $S_1$, $S_2$ and $S_3$ using simultaneous substitutions along (a) and (b) first into a $3$-section $S_{123}$, and then apply a substitution along (c) to arrive at $(X_{(4)},f_{(4)})$.

Hence we have obtained a pair of counter-examples to Stipsicz's conjecture:

\begin{theorem}\label{thm_example_genus3}
Lefschetz fibrations $(X_{(2,2)},f_{(2,2)})$ and $(X_{(4)},f_{(4)})$ are fiber sum indecomposable, but do not admit any $(-1)$ sphere sections. 
\end{theorem}

\begin{remark} 
As shown above, fiber sum indecomposable Lefschetz fibrations without $(-1)$-sphere sections do appear when the fiber genus $g>2$ as well. Furthermore, invoking Endo's signature formula for hyperelliptic Lefschetz fibrations, one can easily observe that these fibrations are not hyperelliptic, as opposed to any genus-$2$ example one can produce.
\end{remark}

\vspace{0.15cm}
\subsection{A new genus-$2$ counter-example} \label{SectionGenus2Counterex} \

The following is the monodromy factorization for a genus-$2$ symplectic Calabi-Yau Lefschetz fibration obtained in the previous section:

{\allowdisplaybreaks
\begin{align*}
t_{\delta_1}t_{\delta_2} & = (t_{d_4}t_{d_3}t_{d_2})^2 (t_{c_2}t_{c_3}t_{c_4})^4 \\
& = (t_{d_4}t_{d_3}t_{d_2})^2 \underbrace{t_{c_2}^2t_{c_4}^2}_{(d)}t_{t_{c_2}^{-1}t_{c_4}^{-2}(c_3)}t_{t_{c_4}^{-1}(c_3)}(t_{c_2}t_{c_3}t_{c_4})^2.
\end{align*}
}
%
We can apply the braiding lantern substitution at the part (d) above as we applied in the previous subsection. 
This substitution changes the two exceptional sections into a sphere bisection with self-intersection $-1$. 
We denote the resulting fibration by $h_{(2)}:Y_{(2)}\to S^2$. Using \cite[Lemma 5.1]{Gompf} we can prove that $Y_{(2)}$ is diffeomorphic to the blow-down of $Y_{(1,1)}$ (see also Proposition~\ref{Gompflemma}). In particular $Y_{(2)}$ is not rational or ruled. Thus, by the same argument as in the previous subsection we obtain the following theorem: 

\begin{theorem}\label{thm_example_genus2}
Lefschetz fibration $(Y_{(2)},h_{(2)})$ is fiber sum indecomposable, but does not admit any $(-1)$ sphere sections. 
\end{theorem}

\noindent Tracing the braided lantern curves in the above monodromy substitution, one can verify that the explicit positive factorization of this fibration, along with its exceptional bisection, is the one we have given earlier in Example~\ref{FactExample}.

\begin{remark}[More examples]
Following the same recipe, one can also obtain counter-examples to Stipsicz's conjecture from the already known $2$-boundary chain relation on the genus--$2$ surface:
\[ (t_{c_1} t_{c_2} t_{c_3} t_{c_4} t_{c_5})^6=1  \, , \]
which prescribes a Lefschetz fibration $f$ on $X=\K \# 2 \CPb$. Through Hurwitz moves, one can easily find a lantern configuration in an equivalent factorization. Since the $4$ curves of the lantern configuration yield a symplectic $2$-sphere $V$ of self-intersection $-4$ \cite{GayMark}, and since $(X,f)$ is an SCY, we know that there are $2$ exceptional sections $S_1, S_2$ of $f$ each hitting $V$ once. In turn, this a priori tells that we have a lift of this factorization where one can apply the braiding lantern relation of Lemma~\ref{lem_lantern_relation} to produce a new SCY Lefschetz fibration, which is fiber sum indecomposable but doesn't have any exceptional sections. In fact, Auroux's genus-$2$ Lefschetz fibration, the first known counter-example to Stipsicz's conjecture \cite{Sato_2008}, can be seen to arise in this way as well; see \cite{BH2}.
\end{remark}

Intrinsic to our strategy to argue that we obtain true counter-examples is that they are all symplectic Calabi-Yau Lefschetz fibrations. Although one would expect the answer to be affirmative, a natural question that arises is:

\begin{question}
Are there any fiber sum indecomposable Lefschetz fibrations with no $(-1)$-sphere sections on symplectic $4$-manifolds of non-zero Kodaira dimension?
\end{question}

As we discussed earlier, by Usher's theorem \cite{Usher}, any Lefschetz fibration on a non-minimal symplectic $4$-manifold is necessarily fiber sum indecomposable. (A short alternative proof of this particular fact was given in \cite{BaykurDecomp} making use of multisections.) To the best of our knowledge, there are no known examples of fiber sum indecomposable --nontrivial-- Lefschetz fibrations on \textit{minimal} symplectic $4$-manifolds, where there are many examples of fiber sum decomposable ones such as the Lefschetz fibrations on knot surgered elliptic surfaces \cite{FSKnotSurgeryLF}. We end with noting this curious question:

\begin{question}
Are there any fiber sum indecomposable Lefschetz fibrations on minimal symplectic $4$-manifolds?
\end{question}

\vspace{0.2in}
\section{Rational blow-downs and non-isomorphic Lefschetz fibrations} \label{Nonisomorphic}

Two Lefschetz pencils\,/\,fibrations $(X, f_i)$ are called  \emph{isomorphic} if there are \linebreak orientation-preserving self-diffeomorphisms of the $4$-manifold and the base $S^2$ which make the two commute. In particular, the fiber genera, as well as the number of base points in the case of pencils, should match. This translates to the two associated monodromy factorizations to be equivalent up to global conjugations and Hurwitz moves. Park and Yun used the latter approach to show that there are pairs of odd genus $g \geq 5$ inequivalent Lefschetz \textit{fibrations} on certain knot surgered elliptic surfaces \cite{ParkYun} of Fintushel and Stern. These are all fiber sum \textit{decomposable}, in particular do not blow-down to pencils, and are easily seen to be equivalent via partial conjugations ---conjugations applied to a subword of monodromy factorizations. (Another isolated example on $T^2 \x \Sigma_2 \# 9 \CPb$, as we learn from the same authors, is given by Smith in his thesis.) More recently, the first author constructed arbitrary number of non-isomorphic Lefschetz pencils\,/\,fibrations on blow-ups of any symplectic $4$-manifold which is not a rational or ruled surface \cite{BaykurLuttingerLF}.

Here we will introduce a new approach and construct the first examples of non-isomorphic Lefschetz \textit{pencils} with fiber genus as low as $3$. To do so, we will employ rational blow-down operations corresponding to certain monodromy substitutions, which we briefly review below. The key ingredient in our arguments will be an observation originally due to Gompf which we review in the next subsection.

\vspace{0.15cm}
\subsection{A mirror rational-blowdown operation} \ 

The \textit{rational blow-down} operation introduced by Fintushel-Stern \cite{FSrationalblowdown} is defined as follows: Let $p \geq  2$ and $C_p$ be the smooth $4$-manifold obtained by plumbing disk bundles over the $2$-sphere according to the linear diagram in Figure~\ref{fig_graph_blowdown}, where each vertex $u_{i}$ of the linear diagram represents a disk bundle over $2$-sphere with the indicated Euler number. 
\begin{figure}[htbp]
\centering
\includegraphics[width=40mm]{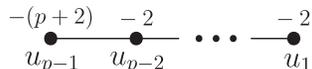}
\caption{A plumbing diagram representing $C_p$.}
\label{fig_graph_blowdown}
\end{figure}

The boundary of $C_p$ is the lens space $L(p^2, 1 - p)$, which also bounds a rational ball $B_p$ with $\pi_1(B_p) = {\mathbb{Z}}/p\Z$ and $\pi_1(\partial B_p) \rightarrow  \pi_1(B_p)$ surjective. Whenever $C_p$ is embedded in a $4$-manifold $X$, we can thus produce a new $4$-manifold 
\[X_p= (X \setminus C_p) \cup B_p . \]
Algebraic topological invariants of $X_p$ are easily derived from those of $X$: 
\[b^{+}(X_p) = {b}^{+}(X) \, , \, \eu(X_p) = \eu(X) - (p-1) \, , \sigma(X_p) = \sigma(X) + (p-1) \, , \]
\[{c_1}^{2}(X_p) = {c_1}^2(X) + (p-1) \, . \]
If $X$ and $X \setminus C_p$ are simply-connected, then so is $X_p$.

A \textit{monodromy substitution} is the trading of a subword in a given monodromy factorization by positive Dehn twists with another subword of the same type. We will employ this operation more generally to monodromy factorizations capturing multisections as well, i.e. not only positive Dehn twists but also positive arc twists for multisections will be allowed in the subwords. As shown by Endo and Gurtas, a particular substitution in a monodromy factorization of a Lefschetz fibration which trades $4$ Dehn twists with $3$ Dehn twists by the lantern relation on a $2$-sphere with $4$ holes corresponds to produce a new Lefschetz fibration corresponds to the simplest possible rational blow-down operation for the underlying $4$-manifolds: blowing-down a $C_2$ configuration which is a $(-4)$-sphere formed by the $4$ vanishing cycles on the fiber. Notably, the two fibrations can be supported by symplectic structures in a natural way \cite{GayMark}.

We will be interested in a special configuration, which can be blow-down in $3$ different ways: 
two disjoint $(-4)$-spheres $V_1, V_2$ and a $(-1)$-sphere $S$ intersecting each of the $(-4)$-sphere positively at one point. One can then blow-down $S$ or rationally blow-down $V_i$. As it will become apparent in the proof of Theorem \ref{nonisomLFs}, we will be interested in monodromy substitutions that correspond to this ``mirror'' blow-downs of $V_1$ and $V_2$ in the presence of a $(-1)$-sphere section intersecting both. The essential ingredient here is a result of Gompf \cite{Gompf} (also see \cite{Dorfmeister}) quoted below, for which we will sketch a handlebody proof. Thus the $4$-manifolds that result from rationally blowing-down $V_1$ or $V_2$ and $S$ are diffeomorphic, yielding diffeomorphic $4$-manifolds for all $3$ blow-downs.

\begin{proposition}[{\cite[Lemma 5.1]{Gompf}}] \label{Gompflemma}
Let $V_i$ be disjoint embedded $(-4)$-spheres, each intersecting a $(-1)$-sphere $S$ in $X$ once, and let $X_i$ denote the $4$-manifold obtained by a rational blowdown of $V_i$ in $X$, $i=1,2$, and $X_0$ be the one obtained by the blowdown of $S$. Then the $4$-manifolds $X_i$, $i=0,1,2$, are all diffeomorphic.
\end{proposition}

\begin{proof}
Let $X_0$ denote the manifold obtained by blowing-down $S$. 
It is sufficient to prove that each of $X_1$ and $X_2$ is diffeomorphic to $X_0$. We will verify the blow-downs along $V_i$ and $S$ give rise to diffeomorphic $4$-manifolds using handlebody diagrams. 

Rational blowdown is equivalent to removing a tubular neighborhood of a \linebreak $(-4)$-sphere and pasting $\CP \setminus \nu C$, where $C$ is a non-singular quadratic curve and $\nu C$ is its tubular neighborhood. 
Following a procedure for drawing a diagram of the surface complement in \cite[Section 6.2]{GS}, we can obtain a diagram of $\CP\setminus \nu C$ which is shown in Figure~\ref{fig_curvecomplement1} (the configuration of $C$ in the diagram of $\CP$ was discussed in \cite{Akbulut_Kirby_1980}).
\begin{figure}[htbp]
\centering
\subfigure[]{\includegraphics[height=20mm]{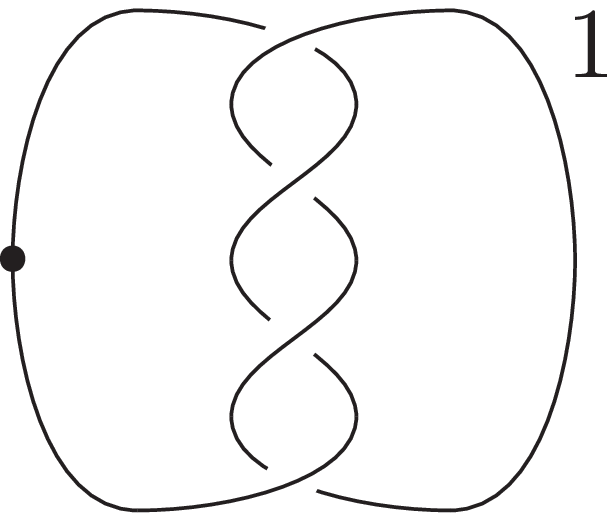}
\label{fig_curvecomplement1}}
\subfigure[]{\includegraphics[height=20mm]{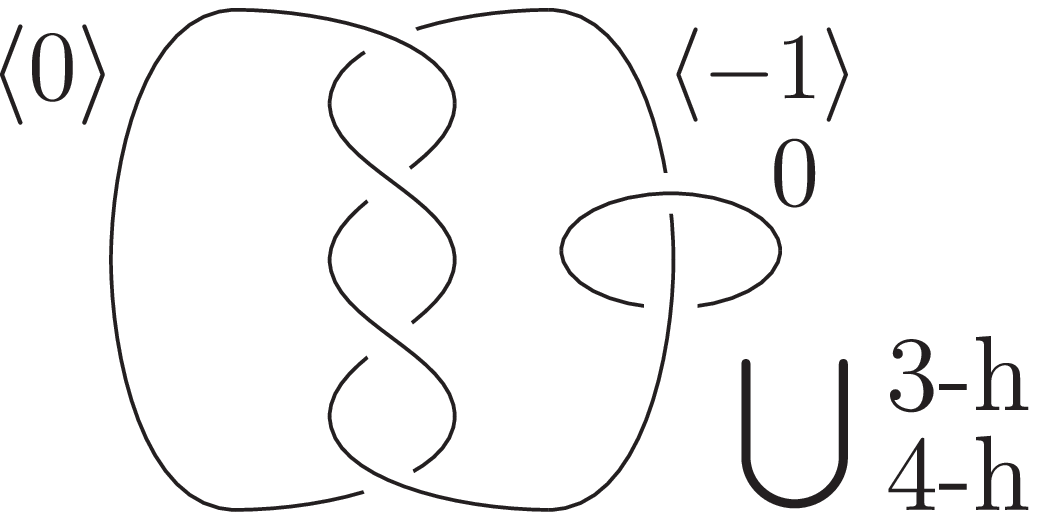}
\label{fig_curvecomplement2}}
\subfigure[]{\includegraphics[height=20mm]{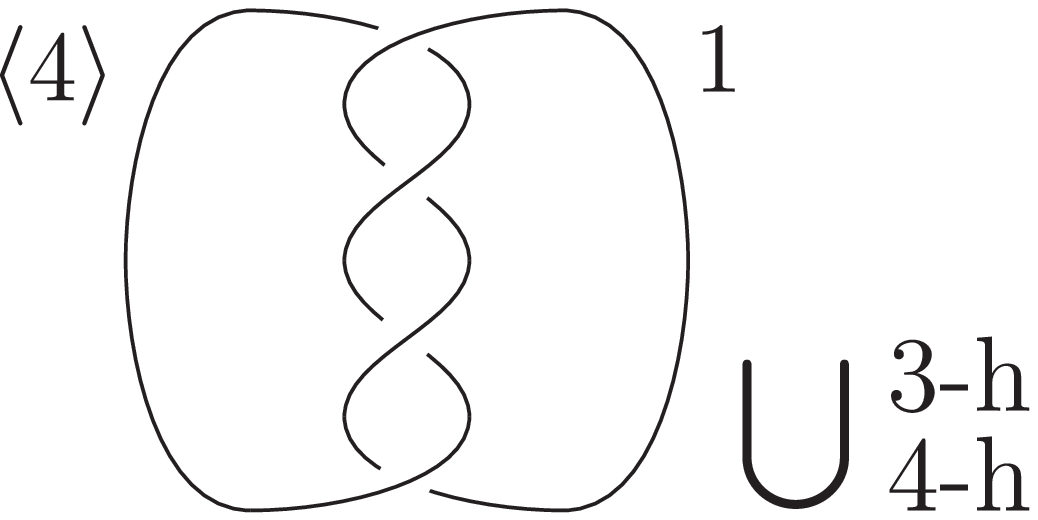}
\label{fig_curvecomplement3}}
\caption{Diagrams of $\CP\setminus \nu C$.}
\label{fig_curvecomplement}
\end{figure}
By turning the handlebody corresponding to the diagram up side down, we obtain the diagram in Figure~\ref{fig_curvecomplement2}. 
The diagram in Figure~\ref{fig_curvecomplement3} can be obtained by blow-down. 
Thus attaching $\CP\setminus \nu C$ to the boundary of a tubular neighborhood of a $(-4)$-sphere is equivalent to attaching handles as shown in Figure~\ref{fig_curvecomplement3}. 
We eventually obtain the diagram of the manifold obtained by rationally blowing-down a regular neighborhood of $V_i\cup S$ along $V_i$ as shown in Figure~\ref{fig_rationalblowdown_plumbing}. (Note that the $0$-framed handle in this diagram coincides with the $(-1)$-sphere $S$). 
\begin{figure}[htbp]
\centering
\subfigure[]{\includegraphics[height=24mm]{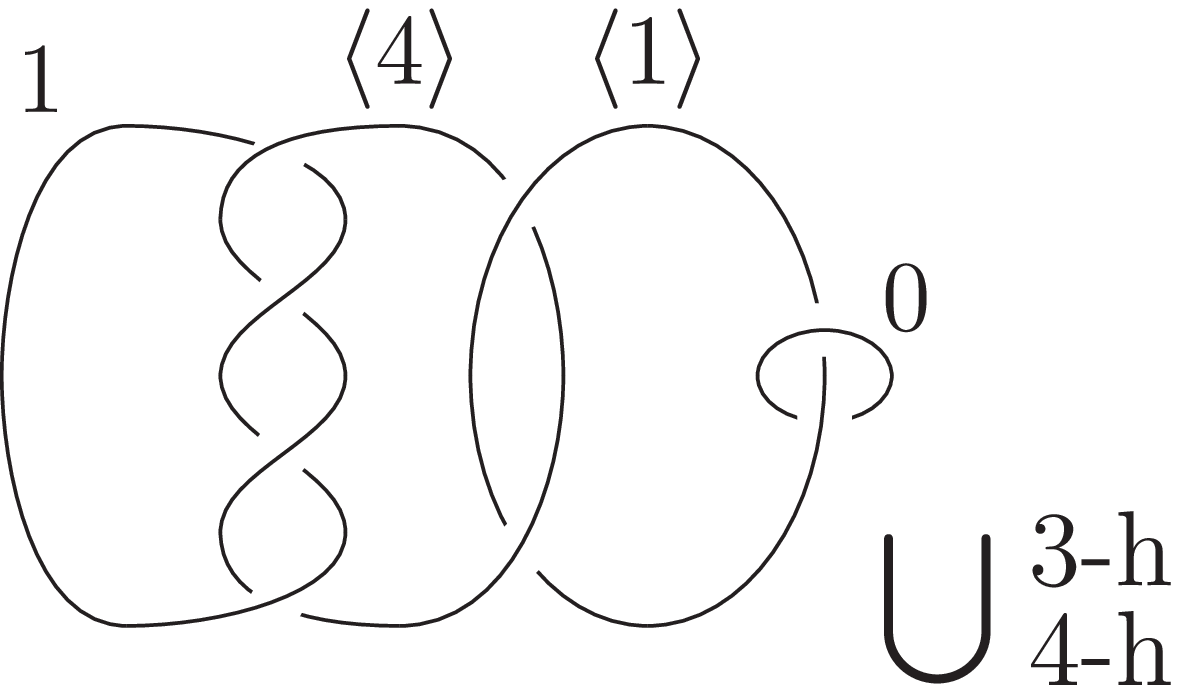}
\label{fig_rationalblowdown_plumbing}}
\hspace{2em}
\subfigure[]{\includegraphics[height=24mm]{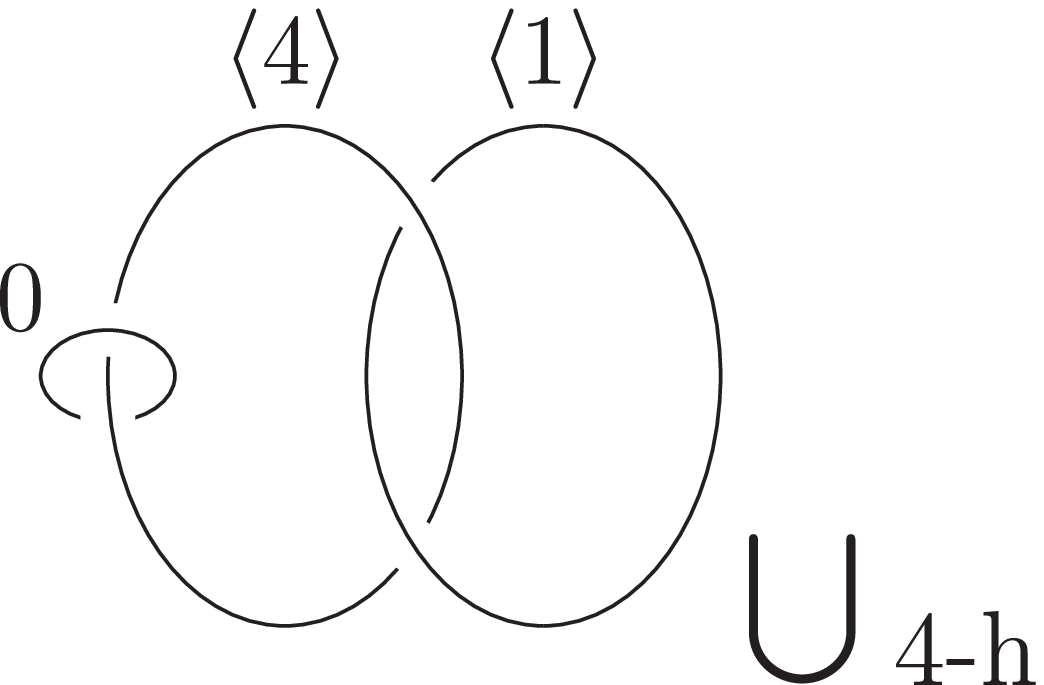}
\label{fig_blowdown_plumbing}}
\caption{The manifolds obtained by blowing down a neighborhood of $V_i\cup S$. }
\end{figure}
On the other hand, Figure~\ref{fig_blowdown_plumbing} shows a diagram of the manifold resulting from blowing-down a regular neighborhood of $V_i\cup S$ along $S$. 
These two manifolds are diffeomorphic relative to the boundaries as shown in Figure~\ref{fig_calculus_diffeomorphism}. 
The diagram in Figure~\ref{fig_calculus_diffeomorphism1} can be obtained by sliding the $0$-framed $2$-handle in Figure~\ref{fig_rationalblowdown_plumbing} to the knot with the label $\left<1\right>$.  
Sliding the resulting $(-1)$-framed handle to that with framing $1$ yields the diagram in Figure~\ref{fig_calculus_diffeomorphism2}. 
The resulting $0$-framed handle becomes a meridian of the knot with the label $\left<4\right>$ after sliding the $1$-framed handle to the knot with the label $\left<4\right>$, especially we can obtain the diagram in Figure~\ref{fig_calculus_diffeomorphism3}. 
Lastly, sliding the $1$-framed handle to the knot with the label $\left<1\right>$ and removing the canceling pair yields the diagram in Figure~\ref{fig_blowdown_plumbing}. 
\begin{figure}[htbp]
\centering
\subfigure[]{\includegraphics[height=24mm]{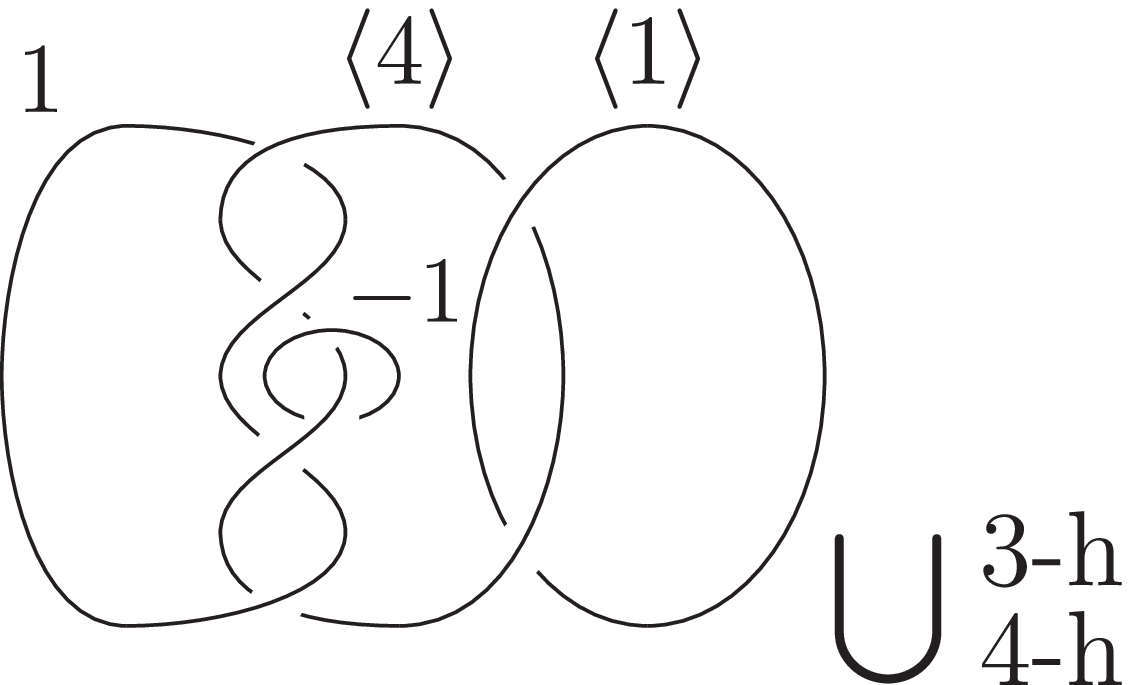}
\label{fig_calculus_diffeomorphism1}}
\subfigure[]{\includegraphics[height=24mm]{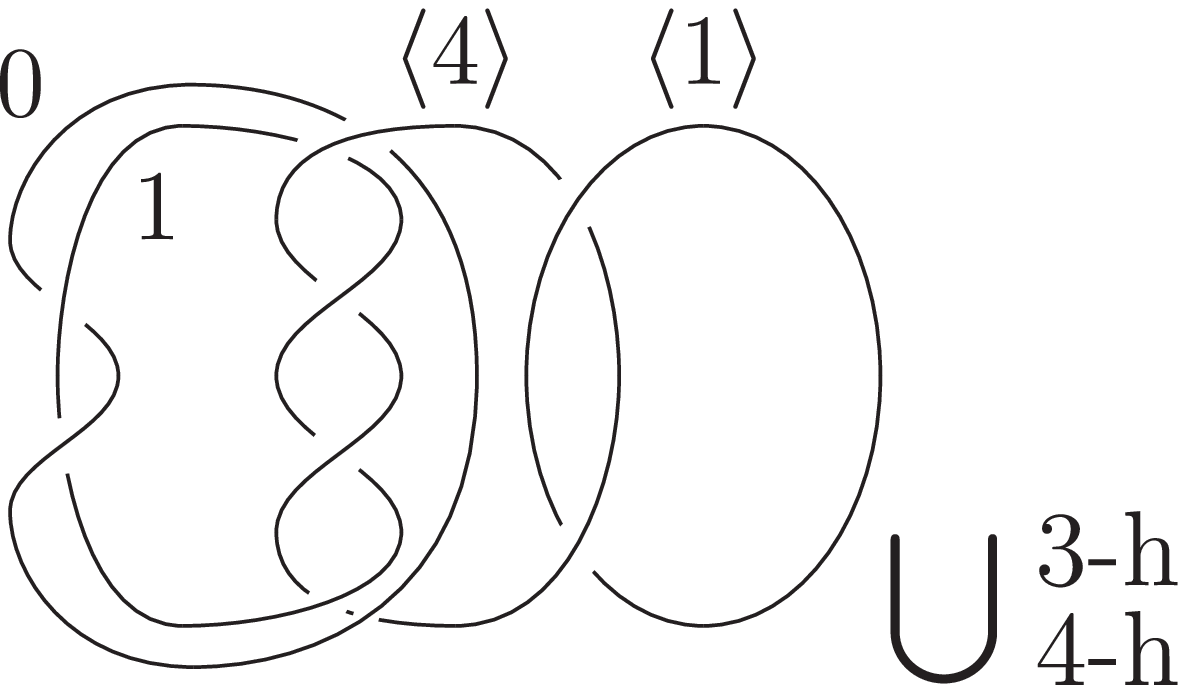}
\label{fig_calculus_diffeomorphism2}}
\subfigure[]{\includegraphics[height=24mm]{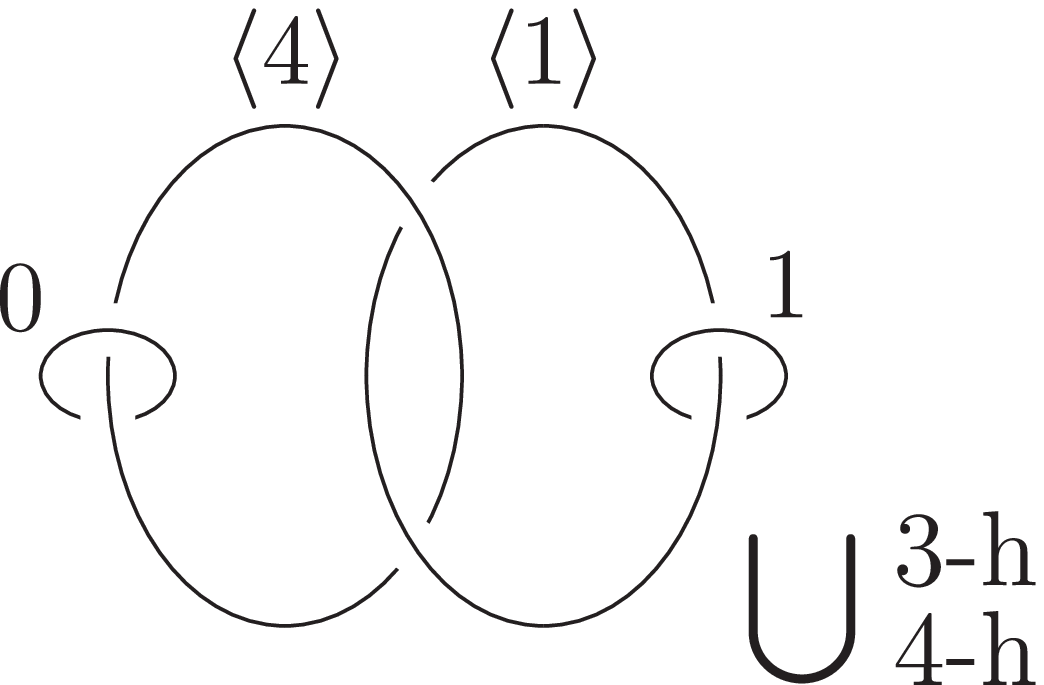}
\label{fig_calculus_diffeomorphism3}}
\caption{A sequence of handleslides.}
\label{fig_calculus_diffeomorphism}
\end{figure}
\end{proof}

%
\vspace{0.15cm}
\subsection{A new construction of non-isomorphic Lefschetz fibrations} \label{NonisomConstruction} \

Here we prove the main theorem of this section:

\begin{theorem} \label{nonisomLFs}
There are pairs of genus-$g$ relatively minimal Lefschetz pencils $(X, f_i)$, $i=1,2$, which are non-isomorphic, where $g$ can be taken as small as $3$, and arbitrarily large. 
\end{theorem}

\begin{proof}
Let $(X_{(2,1,1)}, f_{(2,1,1)})$ be the genus $3$ symplectic Calabi Yau Lefschetz fibration constructed in Subsection~\ref{genus3counter} by a monodromy substitution along (a) in Equation~\ref{genus3factor}. 
Alternatively, we can construct another genus $3$ symplectic Calabi Yau Lefschetz fibration $(X_{(2',1,1)}, f_{(2',1,1)})$ by a monodromy substitution along (c). Note that the exceptional section $S_4$ descends to both fibrations. The configurations of the lantern curves for these are as in Figure~\ref{fig_lantern_configurations}.
\begin{figure}[htbp]
\centering
\subfigure[]{
\includegraphics[height=33mm]{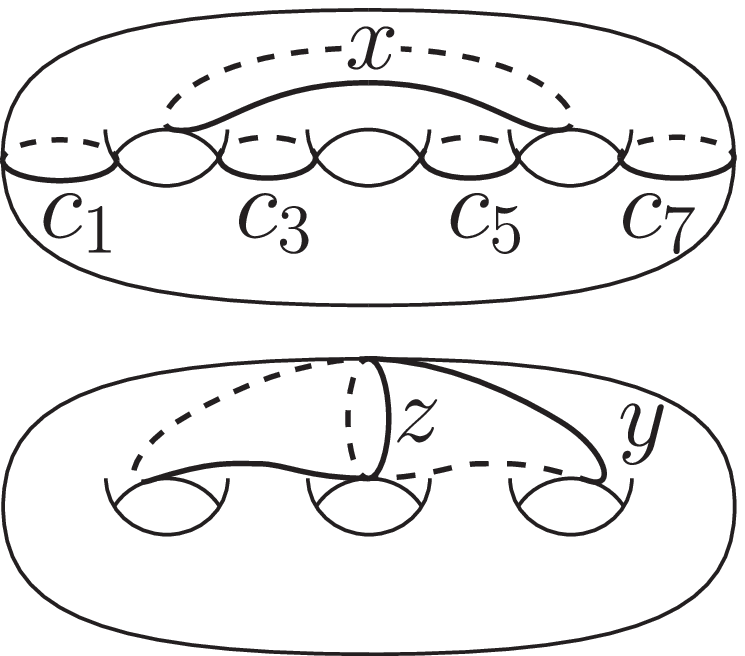}\label{fig_lantern_configuration1}
}
\hspace{10mm}
\addtocounter{subfigure}{1}
\subfigure[]{
\includegraphics[height=33mm]{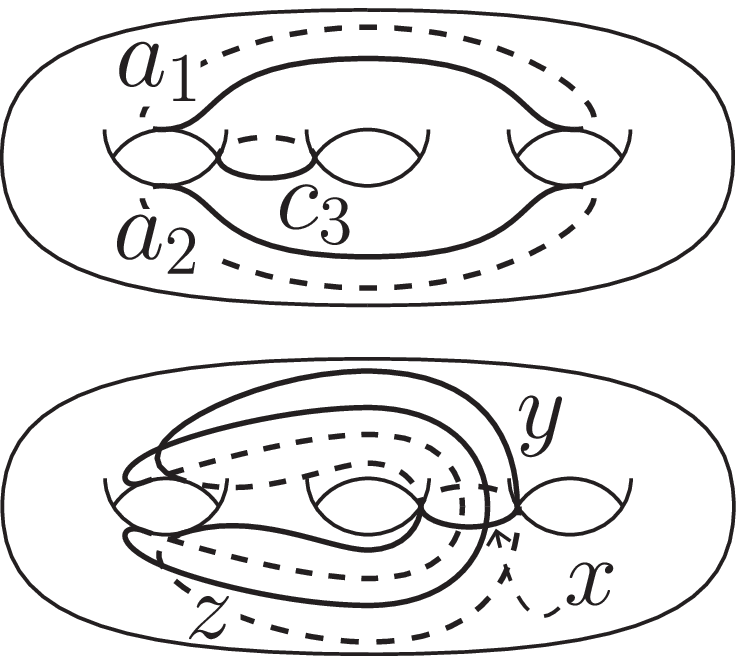}\label{fig_lantern_configuration2}
}
\caption{The lantern curves appearing in the substitution at the parts (a) and (c).}
\label{fig_lantern_configurations}
\end{figure}
Since the section $S_2$ intersects both $(-4)$-spheres $V_1$ and $V_2$ corresponding to the monodromy substitutions along (a) and (c), respectively, this provides us the configuration of surfaces in $X_{(1,1,1,1)}$ studied in the previous paragraph. 
As we observed, the $4$-manifolds $X_{(2,1,1)}$ and $X_{(2',1,1)}$ are diffeomorphic. Since only one of them contains a separating vanishing cycle, they are non-isomorphic genus-$3$ Lefschetz pencils on a once blown-up symplectic Calabi-Yau $\K$ surface. 

An infinite family of pairs of non-isomorphic Lefschetz pencils of arbitrarily high fiber genera can be produced by applying the degree doubling construction (as discussed in detail in the next section) simultaneously to both $(X, f_i)$, at each step changing the genus and number of base points by $\tilde{g}=2g+m-1$ and $\tilde{m}=4m$. The separating vanishing cycle in $(X,f_2)$, obtained from $(X_{(2',1,1)}, f_{(2',1,1)})$ by blowing-down the $(-1)$-sphere section $S_4$ splits the fiber to genus-$1$ and genus $2$ components, only one of which is hit by $S_4$. As seen by the explicit monodromy factorization obtained by Auroux and Katzarkov \cite{AK}, this separating vanishing cycle will then contribute to a separating vanishing cycle of the resulting pencil. On the other hand, doubling produces no new separating vanishing cycles. Therefore, iterated doubles of $(X, f_i)$, for $i=1,2$ at each step will remain to be non-isomorphic, as one will contain a reducible fiber and the other one will not. 
\end{proof}

\begin{remark}[An interlude on smallest possible fiber genera of non-isomorphic Lefschetz fibrations/pencils]\label{R:genera nonisomLF} 
Here we would like to make a few remarks on how small the genus of non-isomorphic Lefschetz pencils\,/\,fibrations can be, discussed by varying the additional features of such examples that we know of so far.

First of all, if we are after obtaining non-isomorphic Lefschetz fibrations, for $g \leq 1$ this phenomenon does not appear. If we ask for non-isomorphic examples with differing number of reducible fibers as we have in Theorem~\ref{nonisomLFs}, our $g=3$ fibrations can be seen to be the smallest possible genera representatives of this kind. For, any genus-$2$ Lefschetz fibration is hyperelliptic, the signature formula of Endo shows that the signature contributions of reducible and irreducible singular fibers are different, and thus their total spaces cannot have the same total space. 

As for having non-isomorphic Lefschetz fibrations in general (so both can have the same number of reducible fibers, or even none), the smallest possible genus among Park and Yun's examples of non-isomorphic Lefschetz fibrations is $g=5$ for a pair of knot surgered $E(2n)$ with $n=1$ where the knots are $2$-bridge knots of genus $2$. This is now improved to $g=3$ by our examples, leaving out $g=2$ as the smallest possible genus. 

We speculate that $g=2$ examples with transitive monodromy and without reducible fibers do not exist. 
Indeed, it is known that any such a Lefschetz fibration is isomorphic to a fiber-sum whose components are either of the two basic Lefschetz fibrations. 
Furthermore, the number of the two basic fibrations in the fiber-sum decomposition is uniquely determined by the number of Lefschetz singularities in the original fibration (see \cite[Corollary 0.2]{SiebertTian} for details), in particular the isomorphic class of such a fibration is uniquely determined by the number of Lefschetz singularities. This rigidity then implies that the only counter-examples can come from fibrations with reducible fibers or intransitive monodromies. This concludes our interlude.

\end{remark}


The examples of non-isomorphic Lefschetz fibrations with topologically isotopic fibers we have constructed here are all on symplectic $4$-manifolds of $\kappa =0$. We hope to address the same question for the remaining Kodaira dimensions in a future work. 

On the other hand, given higher genus fibrations with many exceptional sections and several (braiding) lantern factorizations embedded in them (which are hard to produce!), one can turn our strategy of the proof above into a recipe to produce examples of \textit{arbitrarily many} pairwise non-isomorphic Lefschetz fibrations with pairwise ambiently homeomorphic fibers. 

Nevertheless, the following question remains open: 

\begin{question} \cite{ParkYun}
Are there infinitely many non-isomorphic Lefschetz fibrations of the same genera on any symplectic $4$-manifold? 
\end{question}


\vspace{0.2in}
\section{Exotic Lefschetz pencils and exotic embeddings of surfaces} \label{Exotic}

A pair of $4$-manifolds $X_i$, $i=0,1$, that are pairwise homeomorphic but not diffeomorphic, is commonly called an \textit{exotic pair} of $4$-manifolds. 
Pairs that are both symplectic are particularly interesting in regards to the symplectic botany problem, which asks about the diversity of symplectic structures supported in the same homeomorphism class. Similarly, a pair of Lefschetz pencils\,/\,fibrations $(X_i,f_i)$ is called exotic if $X_i$ constitute an exotic pair of symplectic $4$-manifolds and $f_i$ have the  same fiber genus and the same number of base points. Up to date, the only known examples are some particular families of exotic Lefschetz \textit{fibrations}: for $X_i=E(n)_{K_i}$ knot surgered elliptic surfaces, it was shown by Fintushel and Stern that for $K_i$  fibered knots with the same genus $g$ but different Alexander polynomials, one obtains genus-$(2g+n-1)$ exotic Lefschetz fibrations $(X_i,f_i)$ --- which however do not yield pencils. (These are discussed in detail in our Appendix.) However, although every symplectic $4$-manifold admits a Lefschetz \textit{pencil} by Donaldson, there are no known exotic pairs of Lefschetz pencils up to date. Our goal in this final section is to present the first examples of this kind.

\begin{theorem} \label{ExoticLP}
There are genus-$3$ exotic Lefschetz pencils $(X_i,f_i)$, $i=0,1$, with symplectic Kodaira dimension $\kappa(X_i)=i$, where $X_i$ are homeomorphic to $\K \# \CPb$. Moreover, there are similar examples with arbitrarily high genus and the same topology for the singular fibers on higher blow-ups of homotopy $\K \# \CPb$s.
\end{theorem}

\noindent Note that since any symplectic $4$-manifold $X$ with $\kappa = -\infty$ is \textit{diffeomorphic} to a rational or ruled surface, and any minimal $X$ with $\kappa=2$ has $c_1^2>0$, this is the best possible result one can obtain for varying the Kodaira dimensions within the same homeomorphism class.

Lastly, in Subsection~\ref{Knottings}, we will show that similar techniques can be employed to produce exotic embeddings of surfaces.

\vspace{0.15cm}
\subsection{Constructing pairs of exotic Lefschetz pencils}  \ 

Recall the monodromy factorization of the symplectic Calabi Yau Lefschetz fibration $(X_{(1,1,1,1)}, f_{(1,1,1,1)})$ we have produced in Equation~\ref{genus3factor}. A monodromy substitution at (a) (resp. (b)) amounts to rationally blowing-down a $(-4)$-sphere split off from the regular fiber by the $4$ vanishing cycles in (a) (resp. (b)), for which there are two possibilities: one can blow-down the $(-4)$-sphere intersecting the sections $S_1, S_2$ or the one intersecting $S_3, S_4$. (This choice was implicitly made when producing our earlier examples by each time indicating which sections we were braiding.) Thus, we can blow-down two disjoint $(-4)$-spheres in $(X_{(1,1,1,1)}, f_{(1,1,1,1)})$ \textit{both} intersecting $S_1, S_2$ by monodromy substitutions along (a) and (b) simultaneously to produce a new Lefschetz fibration $(X',f')=(X_{([2],1,1)}, f_{([2],1,1)})$. The Euler characteristic, the signature and the fundamental group of $X_{([2],1,1)}$ is easily calculated from those of $X_{(1,1,1,1)}$ under rational and regular blowdowns, allowing us to conclude that it is \textit{homeomorphic} to $\K \# 2\CPb$. Using Theorem~\ref{mainthm} we see that $S_1, S_2$ together turn into a self-intersection $0$ torus bisection of $f_{([2],1,1)}$, whereas the $(-1)$-sphere sections $S_3$ and $S_4$ descends to sections of the new fibration as well. Since the minimal model of an SCY with $b^+=3$ should have the same rational homology as the $\K$ surface, by Theorem~\ref{SCYLF} this cannot be a symplectic Calabi-Yau Lefschetz fibration. Blowing-down $S_3$, we get a Lefschetz pencil $(X_1, f_1)$ on a symplectic $4$-manifold $X'$. Calculating $c_1^2=2\eu + 3 \sigma=0$ on the minimal model of $X_1$ we note that $\kappa(X_1)=1$.

On the other hand, let $(X_0, f_0)$ be the pencil obtained by blowing-down the SCY Lefschetz fibration $(X_{(3,1)}, f_{(3,1)})$ we constructed earlier along the only $(-1)$-sphere section, so $\kappa(X_0)=0$. Thus $(X_i, f_i)$, for $i=0,1$ is a pair of genus-$3$ Lefschetz pencils promised in Theorem~\ref{ExoticLP}. 

The only caviat in our construction of these exotic Lefschetz pencils is that $(X_1, f_1)$ has no reducible fibers, whereas $(X_0, f_0)$ has one reducible fiber. Below, we will show that,  if we compromise on the smallness of the pencil genus, we can also produce exotic Lefschetz pencils both having only irreducible fibers. 

In the arguments to follow, we will need a variant of the degree doubling procedure \cite{Smith2,AK}, introduced in \cite{BaykurLuttingerLF}. Degree doubling construction produces a new genus $\tilde{g}$ symplectic Lefschetz pencil $(X,\omega, \tilde{f})$ with $\tilde{m}$ base points from a given genus $g$ symplectic Lefschetz pencil $(X,\omega, f)$ with $m$ base points, where $\tilde{g}=2g+m-1$ and $\tilde{n}=4m$. It is described for Donaldson's pencils in Smith's work \cite{Smith2}, for pencils obtained via branched coverings of $\CP$ by Auroux and Katzarkov in \cite{AK}, and for arbitrary topological pencils by the first author in \cite{BaykurLuttingerLF} based on \cite{Smith2,AK} ---which is the one that suits to pencils constructed via monodromy factorizations. We define a \emph{partial double along $m \geq k \geq 1$ points} as the Lefschetz pencil one gets by \textit{first} symplectically blowing-up $(X, \omega, f)$ at $m-k$ points and then taking the double of the resulting pencil on $(\tilde{X}, \tilde{\omega}, \tilde{f})$, where $\tilde{X} = X \# (m-k) \CPb$. Moreover, if $(\tilde{X}, \tilde{\omega}, \tilde{f})$ is obtained from $(X, \omega, f)$ by a sequence of partial doublings, where in the very last step we in addition blow-up all the base points, then both the smooth $4$-manifold $\tilde{X}$ and the genus $\tilde{g}$ of $\tilde{f}$ are \emph{uniquely} determined by the initial pencil $(X, \omega, f)$ and the ordered tuple of integers $k_1, \ldots, k_d$, for each partial doubling along $k_j$ points. We can then blow-down the $(-1)$-sphere sections to produce a pencil. 

Following \cite{BaykurLuttingerLF}, we denote the latter sequence by $D=[k_1, \ldots, k_d]$, which is only subject to the condition $4k_{j} \geq k_{j+1} \geq 1$ for all $j$. The next lemma is a simple variation of \cite[Lemma~3.1]{BaykurLuttingerLF} proved in an identical way:

\begin{lemma} \label{matching}
Let $f$ and $f'$ be genus $g_0$ and $g'_0$ Lefschetz pencils on homeomorphic $4$-manifolds $X$ and $X'$ with $m_0$ and $m'_0$ base points, respectively. Two partial doubling sequences 
\[ D=[k_1, \ldots, k_d] \, \, \text{and} \, \, D'=[k'_1, \ldots, k'_{d'}] \, \]
applied to $f$ and $f'$, respectively, result in Lefschetz fibrations on $X \# M \CPb $ and $X' \# M \CPb$ with the same fiber genus $g$ if and only if
\vspace{-0.5cm} 
\begin{center}
\begin{align*}
M &= m_0+ 3\sum_{i=1}^d k_i = m'_0+ 3\sum_{i=1}^{d'} k'_i  \, \, \, \, \text{and} \\ 
g &= 2^d g_0 + \sum_{i=1}^d 2^{d-i} (k_i-1)= 2^{d'} g'_0 + \sum_{i=1}^{d'} 2^{d'-i} (k'_i-1) \, . 
\end{align*}
\end{center}
\end{lemma}

Now let $(X, f)$ be a genus-$8$ Lefschetz pencil on the $\K$ surface with $14$ base points \cite{Smith2}, so $\kappa(X)=0$. We can then apply Lemma~\ref{matching} to $(X,f)$ and $(X',f')$ using the (very short!) partial doubling sequences 
\[ D=[1] \, \, \text{and} \, \, D'=[2,3] \, , \]
to produce a pair of genus $g$ Lefschetz fibrations on the topological $4$-manifold $\K \# M \CPb$ with $g=16$ and $M=17$. Blowing-down the same number of \linebreak $(-1)$-sphere sections (and at most $4$ of them, as the doubling sequence $D$ results in $4$ base points) in both we obtain the desired exotic pair of Lefschetz pencils $(X_0, f_0)$ and $(X_1, f_1)$ (overriding our earlier picks of $(X_i, f_i)$) where $X_0$ now denotes (a blow-up) of $\K \# 13 \CPb$ and $X_1$ is homeomorphic to it. 

Applying further simultaneous doublings to any one of the exotic Lefschetz pencils $(X_i,f_i)$ we produced above give us exotic pairs of pencils of arbitrarily high genera. This completes the proof of Theorem~\ref{ExoticLP}. 

\begin{remark}
After one more substitution along part (b) in the monodromy of $(X_{([2],1,1)}, f_{([2],1,1)})$, we can obtain another \textit{fibration} $(X_1,f_1)=(X_{([3],1)}, f_{([3],1)})$. Letting $(X_0, f_0)=(X_{(4)}, f_{(4)})$ be the SCY Lefschetz \textit{fibration} produced in Section~\ref{Stipsicz}, we then obtain a pair of exotic genus-$3$ Lefschetz fibrations $(X_i, f_i)$ with $\kappa(X_i)=i$, both of which having one reducible fiber --- and thus, with exact same topology.
\end{remark}

\vspace{0.15cm}
\subsection{Exotic embedings of symplectic surfaces} \label{Knottings} \

We call two surfaces $F_i \subset X$, $i=1,2$, \textit{exotically embedded} in $X$ if there exists an ambient homeomorphism of $X$ taking $F_1$ to $F_2$ but there exists no such diffeomorphism. Such symplectic surfaces are harder to produce: for instance, the work of Siebert and Tian shows that up to isotopy there is a unique symplectic surface in the homology class of an algebraic curve of degree $\leq 17$ in $\CP$ \cite{SiebertTian}. In contrast, Finashin \cite{F}, and H.-J.\,Kim \cite{K} \,(also see \cite{KR}) constructed knotted surfaces in $\CP$ that are not isotopic to algebraic curves, which can be seen to be not symplectic. The latter rely on a construction method of Fintushel and Stern \cite{FSsurfaces}, called (twisted) rim-surgery, and up to date this has been the only way of producing exotic embeddings of surfaces -- and curiously, only producing symplectic tori when asked to lie in the same homology class. The purpose of this section is to present a new way of constructing exotically embedded orientable surfaces: 

\begin{theorem} \label{exoticknotting}
There is a pair of genus-$3$ surfaces $F_i$ exotically embedded in a blow-up of a symplectic Calabi-Yau $\K$ surface such that $F_i$ is symplectic with respect to deformation equivalent symplectic forms $\omega_i$ on $X$, for $i=1,2$. 
\end{theorem}

\begin{proof}
In Section~\ref{genus3counter}, we constructed a Lefschetz fibration $(X_{(2,1,1)},f_{(2,1,1)})$ by a braiding lantern substitution at part (a) of \eqref{genus3factor}. We can now apply another braiding lantern substitution at part (c), which yields to the genus-$3$ Lefschetz fibration $(X_{(2,2)},f_{(2,2)})$ with two $(-1)$-sphere bi-sections $S_{12}$ and $S_{34}$ we obtained earlier, or at part (b), which yields to a new genus-$3$ Lefschetz fibration $(X_{(3,1)},f_{(3,1)})$ with $(-1)$-sphere $3$-section $S_{123}$ and a $(-1)$-section $S_4$. 

Since $X_{(2,2)}$ and $X_{(3,1)}$ are obtained from $X_{(2,1,1)}$ by rational blowdowns along $(-4)$-spheres $V_1$ and $V_2$ (prescribed by the Lantern curves in parts (c) and (b)) both intersecting the exceptional sphere $S_3$ at one point, they are diffeomorphic by Proposition~\ref{Gompflemma}. Let $F_1, F_2$ be regular fibers of $f_{(2,2)}$ and $f_{(3,1)}$, respectively. 

There exists a pairwise homeomorphism between $(X_{(2,2)}, F_1)$ and $(X_{(3,1)}, F_2)$: To see this, we first observe that the vanishing cycles in the complement of $X_{(2,2)} \setminus F_1$, and respectively of $X_{(3,1)} \setminus F_2$, allows us to easily compute $\pi_1(X_{(2,2)} \setminus F_1) = 1 = \pi_1(X_{(3,1)} \setminus F_2)$.  So both homology classes $[F_i]$ are indivisible. Moreover, $F_1 \cdot S_{12} = 2$ but $S_{12}^2 = -1$, whereas for $F_2$, $f_{(3,1)}$ has a reducible fiber component $R$, so $F_1 \cdot R =0$ but $R^2 = -1$. So both $[F_i]$ are not characteristic. Since $b_2 - \sigma \geq 4$ and $\pi_1=1$ for $X_{(2,2)} \cong X_{(3,1)}$, by Wall's theorem on automorphisms of the intersection form and Freedman's topological h-cobordism theorem (see for example \cite{Scorpan} p.152-153), we get a homeomorphism between $X_{(2,2)}$ and $X_{(3,1)}$ matching the homology classes of $F_1$ and $F_2$. Finally, viewing the two surfaces in the same manifold under this homeomorphism, we can invoke \cite{Sunukjian} to find a topological isotopy between them, which yields the desired homeomorphism between the pairs $(X_{(2,2)}, F_1)$ and $(X_{(3,1)}, F_2)$.

On the other hand, since $X_{(2,2)}, X_{(3,1)}$ are symplectic Calabi-Yaus, and thus not rational or ruled, by Li's work in \cite{Li_1999}, any diffeomorphism between them maps exceptional classes to exceptional classes in the same homology classes. However, in $X_{(2,2)}$ the two exceptional classes $S_{12}, S_{34}$ intersect $F_1$ both twice, whereas in $X_{(3,1)}$ we have two exceptional classes $S_{123}, S_4$ intersecting $F_2$ thrice and once. Hence, there is no pairwise diffeomorphism between $(X_{(2,2)}, F_1)$ and $(X_{(3,1)}, F_2)$.

Now if we let $X=X_{(3,1)}$ and identify $F_1$ with its image under the diffeomorphism between $X_{(2,2)}$ and $X_{(3,1)}$, we conclude that $F_1, F_2$ is a pair of exotically embedded surfaces in $X$. Lastly, to prove our additional claim on the existence of deformation equivalent symplectic forms $\omega_i$ on $X$ with respect to which $F_i$ are symplectic, we first perturb the Lefschetz fibration $f_{(2,1,1)}$ so that each quadruple of vanishing cycles appearing in part (c) and (b) of the monodromy factorization lie on the same singular fiber, forming a reducible  $(-4)$-sphere fiber component $V_i$. We can then equip $(X_{(2,1,1)},f_{(2,1,1)})$ with a compatible symplectic form with respect to which both $V_i$s and the section $S_3$ are symplectic. Now, by the work of McDuff and Symington, who showed that Gompf's diffeomorphism we employed here can be interpreted as a symplectic $4$-sum operation, the symplectic $4$-manifolds we produce by rational blowdowns of $V_i$ are symplectic deformation equivalent \cite{MS}, with $F_i$ symplectic surfaces in them. We thus obtain the desired symplectic forms on $X$ by pulling-back the latter form on $X_{(2,2)}$. 
\end{proof}

\vspace{0.5in}
\appendix

\section{Seiberg-Witten basic classes of homotopy $\K$ surfaces via mapping class group factorizations} \label{SWapp} 

\subsection{Seiberg-Witten basic classes of symplectic $4$-manifolds} \label{SW} \

Let $X$ be a symplectic $4$-manifold with $b^+(X) > 1$. We further assume that it has an integral symplectic form $\omega$, which can always be achieved by replacing a given form with a multiple of a rational symplectic form approximating it. By Taubes, for a generic almost complex structure $J$ on $(X, \omega)$, any Seiberg-Witten basic class $\beta \in H_2(X; \Z)$ can be represented by a sum of $J$-holomorphic curves $C_i$ in $X$ \cite{T, T2}. Moreover, each component of the representative of $\beta = \Sigma_i [C_i]$ is an embedded smooth curve unless it is a torus of self-intersection zero (in which case the image of the curve is still smoothly embedded, but the parametrization is a multiple cover) or a sphere of negative self-intersection. Since $J$ is $\omega$ tamed, each $C_i$ is a symplectic surface in $(X, \omega)$. 

Since the number of basic classes of a $4$-manifold is finite, so is the collection of the symplectic surfaces $C_i$, sums of which represent the basic classes in $(X, \omega)$. As noted by Donaldson and Smith \cite[Proposition 2.9]{Donaldson_Smith_2003} replacing $\omega$ with a sufficiently high multiple $k \omega$, we can then assume that there exists a symplectic Lefschetz pencil on $X$ for which all $C_i$ are multisections (``standard surfaces'' in the language of \cite{Donaldson_Smith_2003}). By the blow-up formula for Seiberg-Witten classes, we conclude that after passing to a blow-up of $X$ we get a symplectic Lefschetz fibration $f: \tilde{X} \to S^2$ where all basic classes are represented by a collection of symplectic surfaces $C_i$ and the exceptional spheres $E_j$. Hence, each Seiberg-Witten basic class of $\tilde{X}$ is represented by a multisection (possibly with several components). 

To sum up, combining the works of Taubes and Donaldson, after passing to a blow-up $\tilde{X}$, one can represent all Seiberg-Witten classes of a symplectic $4$-manifold $X$ as multisections with respect to a Lefschetz fibration. We shall note that this is merely an existence result, as the construction of such a Lefschetz fibration is not explicit. 

\vspace{0.15cm}
\subsection{Sample calculation: basic classes of knot surgered elliptic surfaces} \label{knotsurgery} \

We will now present explicit monodromy factorizations in the framed mapping class group capturing all basic classes of knot surgered elliptic surfaces as multisections of certain Lefschetz fibrations on them.  

Here is a quick review of the knot surgery construction: Let $X$ be a smooth $4$-manifold and $T\subset X$ an embedded torus with self-intersection $0$. For a fibered knot $K\subset S^3$, let $M_K$ denote the $3$-manifold obtained by $0$-surgery along $K$ from $S^3$, then $M_K$ admits a natural fibration over $S^1$, where fibers are capped of Seifert surfaces. In turn, $S^1 \x M_K$ is a genus-$g$ symplectic surface bundle over $T^2$, with $g$ the Seifert genus of $K$. For $\mu_K$ the meridian of $K$ in $S^3$, we obtain a torus $S^1 \x \mu_K$ as a symplectic section of this bundle. We then define a \emph{knot surgered $4$-manifold} $X_K$ as the generalized fiber sum $X_K = X\setminus \nu T \cup_{S^1 \x \mu_K} S^1\times M_K$, which can be performed symplectically. (When $K$ is not fibered, the same construction --for $M_K$ admitting an $S^1$-valued Morse function this time-- results in a new $4$-manifold which is not necessarily symplectic.) Fintushel and Stern \cite{Fintushel_Stern_1998} introduced this operation and proved that a Laurent polynomial associated with the Seiberg-Witten invariant of $X_K$ is the product of that of $X$ and the symmetrized Alexander polynomial of the knot $K$ for homologically essential $T$ in $X$. For $X=E(n)$,  all basic classes arise as multiples of the image of the elliptic fiber $T$ of $X$ in $X_K$. Moreover, assuming $K$ is a fibered knot with Seifert genus-$g$, the knot surgery $4$-manifold $E(n)_K$ admits a genus-$(2g+n-1)$ Lefschetz fibration $(E(n)_K, f_{n,K})$ \cite{Fintushel_stern_2004}. It is easy to see that $T$ becomes a bisection (i.e. a $2$-section) of this fibration. Capturing all basic classes of $X_K$ in this case therefore comes to identifying disjoint copies of $T$ via a monodromy factorization of an appropriate lift of $f_{n,K}$ to the framed mapping class group.

Let $A_1,\ldots, A_{2n-2}, B_1,\ldots, B_{2g+1}, C_1,C_2$ be simple closed curves in $\Sigma_{2g+n-1}$ as described in Figure~\ref{scc_knotsurgerybisection}. 

\begin{figure}[htbp]
\begin{center}
\includegraphics[width=120mm]{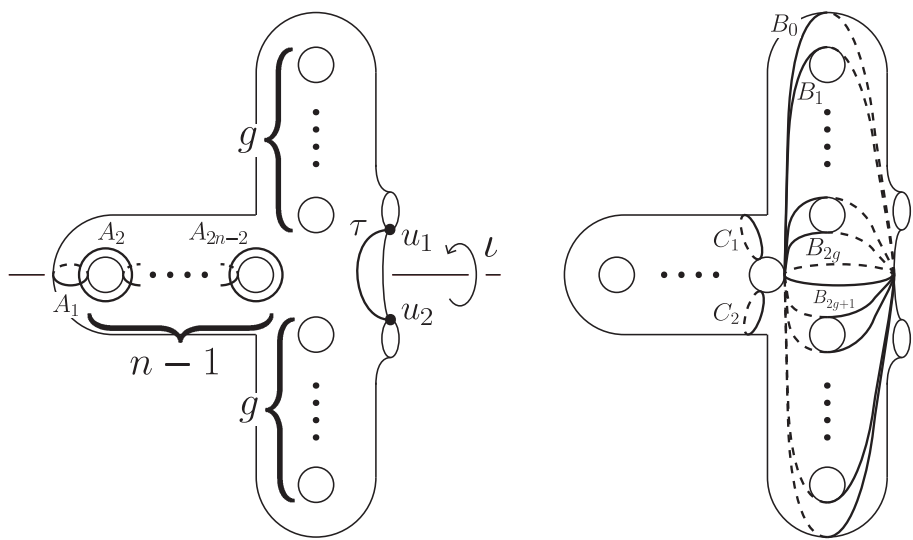}
\end{center}
\caption{Simple closed curves and a path in $\Sigma_{2g+n-1}^2$. }
\label{scc_knotsurgerybisection}
\end{figure}

We remove two disks $D_1, D_2$ from $\Sigma_{2g+n-1}$ as in Figure~\ref{scc_knotsurgerybisection} and take points $u_1, u_2$ on each boundary component of $\Sigma_{2g+n-1}^2 = \Sigma_{2g+n-1} \setminus (D_1\amalg D_2)$. 
Let $K$ be a fibered knot with genus-$g$ and $\varphi_K\in \Mod(\Sigma_g)$ a monodromy of $K$. 
We decompose $\Sigma_{2g+n-1}$ into three pieces: the upper $\Sigma_g$, the lower $\Sigma_g$ and the central $\Sigma_{n-1}$ in Figure~\ref{scc_knotsurgerybisection}, so that both of the disks $D_1$ and $D_2$ are contained in $\Sigma_{n-1}$. 
Let $\Phi_K$ be an element $\Mod(\Sigma_{2g+n-1})$ defined as follows: 
\[
\Phi_K = \varphi_K \# \id \# \id : \Sigma_g \#\Sigma_{n-1} \# \Sigma_g \rightarrow \Sigma_g \#\Sigma_{n-1} \# \Sigma_g. 
\]
The genus-$(2g+n-1)$ Lefschetz fibration $f_{n,K}: E(n)_K \rightarrow S^2$ mentioned above has the following monodromy factorization (see \cite{Fintushel_stern_2004}):
\[
\eta_{n,g} \eta_{n,g} \Phi_K(\eta_{n,g}) \Phi_K(\eta_{n,g}) = 1, 
\]
where $\eta_{n,g}$ is equal to $t_{A_{2n-2}}\cdots t_{A_1}t_{A_1} \cdots t_{A_{2n-2}} t_{B_0}\cdots t_{B_{2g+1}}$ and $\Phi_K(\eta_{n,g})$ is a factorization obtained from $\eta_{n,g}$ by substituting $A_i$ and $B_j$ in $\eta_{n,g}$ for $\Phi_K(A_i)$ and $\Phi_K(B_j)$, respectively. 

\begin{proposition}\label{prop_relation_surface}

The following equality holds in $\Mod(\Sigma_{2g+n-1}^2; \{u_1,u_2\})$: 
\[
t_{A_{2n-2}}\cdots t_{A_1}t_{A_1} \cdots t_{A_{2n-2}} t_{B_0}\cdots t_{B_{2g+1}} = t_{\delta_1} t_{\delta_2} \tilde{\tau}^{-1} \iota, 
\]
where $\delta_i$ is a simple closed curve in $\Sigma_{2g+n-1}^2$ parallel to the boundary component containing $u_i$, $\tilde{\tau}$ is a lift of a half twist along a path given in Figure~\ref{scc_knotsurgerybisection} as described in Figure~\ref{lift_monodromy1} and $\iota$ is an involution described on the left side of Figure~\ref{scc_knotsurgerybisection}. 

\end{proposition}

\begin{proof}
We cut the surface $\Sigma_{2g+n-1}^2$ along the curves $C_1, C_2$ to obtain the surface $\Sigma_{2g}^4$. 
Take points $u_3 \in C_1$ and $u_4\in C_2$. 
We denote the set $\{u_1,u_2,u_3,u_4\}$ by $U$ and the fixed points of $\iota$ in $\Sigma_{2g}^4$ by $v_1, v_2$. 
Since the simple closed curve $B_i$ is preserved by $\iota$, we can regard the Dehn twist $t_{B_i}$ as an element in $\pi_0(C(\Sigma_{2g}^4,U ;\iota))$ (for the definition of $C(\Sigma_{2g}^4,U;\iota)$, see Section~\ref{subsec_lemmainvolution}). 
The quotient space $\Sigma_{2g}^4/\iota$ is homeomorphic to $\Sigma_g^2$. 
The quotient map $/\iota: \Sigma_{2g}^4\rightarrow \Sigma_g^2$ induces the following homomorphism: 
\[
\iota_\ast : \pi_0(C(\Sigma_{2g}^4,U;\iota)) \rightarrow \Mod(\Sigma_{g}^2; \hat{U}, \{\hat{v}_1,\hat{v}_2\}), 
\]
where $\hat{U}$ and $\hat{v}_i$ are the images of $U$ and $v_i$, respectively, under $/\iota$ (see Figure~\ref{quotient_spaces}). 
\begin{figure}[htbp]
\begin{center}
\includegraphics[width=120mm]{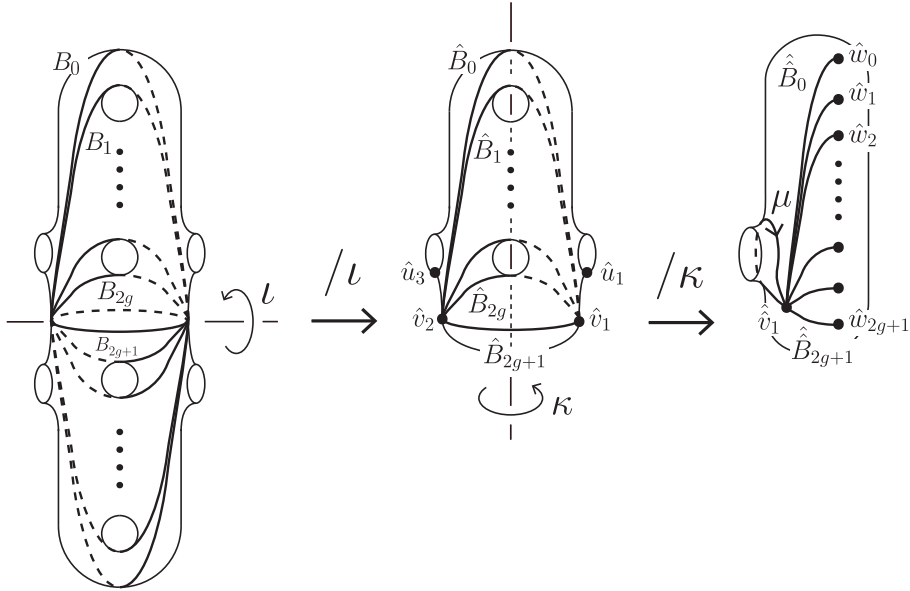}
\end{center}
\caption{The involutions $\iota$ and $\kappa$. }
\label{quotient_spaces}
\end{figure}
By Lemma~\ref{lem_involutionMCG} the kernel of $\iota_\ast$ is generated by the isotopy class of $\iota$. 
Let $\hat{B}_i$ be the image of the simple closed curve $B_i$ under $/\iota$. 
Since the image $\iota_\ast(t_{B_i})$ is the half twist $\tau_{\hat{B}_i}$, the following holds in $\Mod(\Sigma_g^2;\hat{U}, \{v_1,v_2\})$:
\begin{equation}\label{eq_quotient1}
\iota_\ast(t_{B_0}\cdots t_{B_{2g+1}}) = \tau_{\hat{B}_0}\cdots \tau_{\hat{B}_{2g+1}}. 
\end{equation} 
Let $\kappa$ be an involution of $\Sigma_g^2$ as described in the middle of Figure~\ref{quotient_spaces} and $w_0,\ldots,w_{2g+1}\in \Sigma_g^2$ fixed points of $\kappa$. 
We regard the half twist $\tau_{\hat{B}_i}$ as an element in $\pi_0(C(\Sigma_g^2,\hat{U}; \kappa), \id)$. 
The quotient space $\Sigma_g^2/\kappa$ is homeomorphic to $\Sigma_0^1$. 
Thus the quotient map $/\kappa: \Sigma_g^2\rightarrow \Sigma_0^1$ induces the following homomorphism: 
\[
\kappa_\ast : \pi_0(C(\Sigma_g^2,\hat{U}; \kappa)) \rightarrow \Mod_{\partial\Sigma_0^1}(\Sigma_0^1; \{\hat{w}_0,\ldots, \hat{w}_{2g+1}\}), 
\]
where $\hat{w}_i$ is the image $/\kappa(w_i)$. 
By Lemma~\ref{lem_involutionMCG} the kernel of $\kappa_\ast$ is generated by the isotopy class of $\kappa$. 
Let $\hat{\hat{B}}_i$ be the image of $\hat{B}_i$ under $/\kappa$ (see Figure~\ref{quotient_spaces}). 
We take an oriented loop $\beta_i\subset \Sigma_0^1$ based at $\hat{v}_1= /\kappa(v_1)$ by connecting $p_0$ with a small circle around $\hat{w}_i$ oriented counterclockwise using $\hat{\hat{B}}_i$. 
The following equation holds in $\Mod_{\partial \Sigma_0^1}(\Sigma_0^1; \{\hat{w}_0,\ldots, \hat{w}_{2g+1}\})$: 
\allowdisplaybreaks{
\begin{align}\label{eq_quotient2}
\begin{split}
\kappa_\ast(\tau_{\hat{B}_0}\cdots \tau_{\hat{B}_{2g+1}}) & = \Push(\beta_0)\cdots \Push(\beta_{2g+1}) \\ 
& = \Push(\mu), 
\end{split}
\end{align}
}
where $\mu$ is an oriented based loop described in Figure~\ref{quotient_spaces}. 
Combining the equations \eqref{eq_quotient1} and \eqref{eq_quotient2}, we obtain the following relation in $\Mod(\Sigma_{2g}^4; U)$: 
\[
t_{B_0} \cdots t_{B_{2g+1}} =\widetilde{\tau}^{-1} \widetilde{\theta}^{-1} t_{C_1} t_{C_2} t_{\delta_1} t_{\delta_2} \iota, 
\]
where $\theta\subset \Sigma_{2g}^4$ is a path between $u_3$ and $u_4$ described in Figure~\ref{path_forknotsurgery}. 
\begin{figure}[htbp]
\begin{center}
\includegraphics[width=55mm]{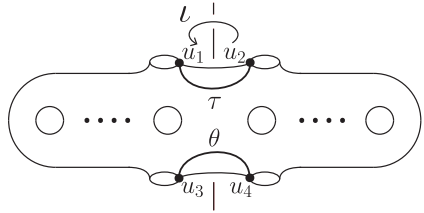}
\end{center}
\caption{Paths in $\Sigma_{2g}^4$. }
\label{path_forknotsurgery}
\end{figure}
Thus, the product $t_{A_{2n-2}}\cdots t_{A_1}t_{A_1} \cdots t_{A_{2n-2}} t_{B_0}\cdots t_{B_{2g+1}}$ is calculated as follows:
\allowdisplaybreaks{
\begin{align*}
& t_{A_{2n-2}}\cdots t_{A_1}t_{A_1} \cdots t_{A_{2n-2}} t_{B_0}\cdots t_{B_{2g+1}} \\
= & t_{A_{2n-2}}\cdots t_{A_1}t_{A_1} \cdots t_{A_{2n-2}} \tilde{\tau}^{-1}\tilde{\theta}^{-1}t_{C_1} t_{C_2} t_{\delta_1} t_{\delta_2} [\iota|_{\Sigma_{2g}^4}] \\
= & t_{A_{2n-2}}\cdots t_{A_1}t_{A_1} \cdots t_{A_{2n-2}} (t_{A_{2n-3}} \cdots t_{A_{1}})^{2n-2} \tilde{\tau}^{-1}\tilde{\theta}^{-1}t_{\delta_1} t_{\delta_2} [\iota|_{\Sigma_{2g}^4}] \\
= & t_{A_{2n-2}}\cdots t_{A_1}t_{A_1} \cdots t_{A_{2n-3}} (t_{A_{2n-2}} \cdots t_{A_{1}}) (t_{A_{2n-3}} \cdots t_{A_{1}})^{2n-3} \tilde{\tau}^{-1}\tilde{\theta}^{-1}t_{\delta_1} t_{\delta_2} [\iota|_{\Sigma_{2g}^4}] \\
= & t_{A_{2n-2}}\cdots t_{A_1}t_{A_1} \cdots t_{A_{2n-4}} (t_{A_{2n-2}} \cdots t_{A_{1}})^2 (t_{A_{2n-3}} \cdots t_{A_{1}})^{2n-4} \tilde{\tau}^{-1}\tilde{\theta}^{-1}t_{\delta_1} t_{\delta_2} [\iota|_{\Sigma_{2g}^4}] \\
= & \cdots \\
= & (t_{A_{2n-2}} \cdots t_{A_{1}})^{2n-1} \tilde{\tau}^{-1}\tilde{\theta}^{-1}t_{\delta_1} t_{\delta_2} [\iota|_{\Sigma_{2g}^4}]. 
\end{align*}
}
It is easy to verify (using the Alexander method, for example) that the product $(t_{A_{2n-2}} \cdots t_{A_{1}})^{2n-1}\tilde{\theta}^{-1} [\iota|_{\Sigma_{2g}^4}]$
is equal to $\iota$ in $\Mod(\Sigma_{2g+n-1}^2; \{u_1, u_2\})$. 
Thus, we obtain: 
\[
t_{A_{2n-2}}\cdots t_{A_1}t_{A_1} \cdots t_{A_{2n-2}} t_{B_0}\cdots t_{B_{2g+1}}=\iota \tilde{\tau}^{-1}t_{\delta_1} t_{\delta_2}. 
\]
This completes the proof of Proposition \ref{prop_relation_surface}. 
\end{proof}

We take simple closed curves $c_1,c_2,c_3$ in $\Sigma_0^{2m+1}$ and points $u_1,\ldots, u_{2m}$ on the boundary of $\Sigma_0^{2m+1}$ as described in Figure~\ref{path_disk}. 
\begin{figure}[htbp]
\begin{center}
\includegraphics[width=80mm]{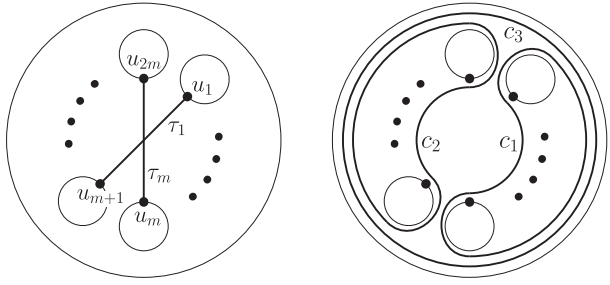}
\end{center}
\caption{Simple closed curves and arcs in the disk with $2m$ small disks removed. }
\label{path_disk}
\end{figure}
Let $\tau_i$ be a radial arc between $u_i$ and $u_{m+i}$ and $U^\prime$ the set $\{u_1,\ldots, u_{2m}\}$. 
We denote by $\lambda \in \Mod_{\partial \Sigma_0} (\Sigma_0^{2m+1}; U^\prime)$ a mapping class represented by a diffeomorphism which is positive $180$-degree rotation inside of $\nu c_3$ and preserves the outside of $\nu c_3$, where $\partial \Sigma_0$ is the outermost boundary component of $\Sigma_0^{2m+1}$ in Figure~\ref{path_disk}. 

\begin{proposition}\label{prop_relation_disk}

The following equality holds in $\Mod_{\partial \Sigma_0} (\Sigma_0^{2m+1}; U^\prime)$: 
\[
t_{c_1}t_{c_2}\lambda^{-1} = t_{\delta_1}\cdots t_{\delta_{2m}} \widetilde{\tau_1}^{-1} \cdots \widetilde{\tau_m}^{-1},  
\]
where $\widetilde{\tau_{i}}\in \Mod_{\partial \Sigma_0} (\Sigma_0^{2m+1}; U^\prime)$ is a lift of a half twist along $\tau_i$ as described in Figure~\ref{lift_monodromy1}. 

\end{proposition}

\begin{proof}

We denote the involution of $\Sigma_0^{2m+1}$ given by the $180$-rotation by $\tilde{\lambda}$. 
We regard $\widetilde{\tau_i}$  and $t_{\delta_i}t_{\delta_{m+i}}$ as elements in $\pi_0(C(\Sigma_0^{2m+1}, U^\prime; \tilde{\lambda}))$. 
The quotient map $/\tilde{\lambda}: \Sigma_0^{2m+1} \rightarrow \Sigma_0^{2m+1}/\tilde{\lambda}\cong \Sigma_0^{m+1}$ induces the following homomorphism: 
\[
\tilde{\lambda}_\ast : \pi_0(C(\Sigma_0^{2m+1},U^\prime; \tilde{\lambda})) \rightarrow \Mod_{\partial \Sigma_0}(\Sigma_0^{m+1}; \{u_0\}, \hat{U}^\prime), 
\]
where $u_0\in \Sigma_0^{m+1}$ is the image of the origin of the disk under $/\tilde{\lambda}$ and $\hat{U}^\prime = /\lambda(U^\prime)$. 
By Lemma~\ref{lem_involutionMCG} the map $\tilde{\lambda}_\ast$ is an isomorphism and the image $\tilde{\lambda}_\ast(\widetilde{\tau_i}^{-1}t_{\delta_i}t_{\delta_{m+i}})$ is a pushing map along some loop based at $u_0$. 
We can easily obtain the equality in Proposition \ref{prop_relation_disk} using these fact together with some equality in $\pi_1(\Sigma_0^{m+1} \setminus \hat{U}^\prime, u_0)$. 
The details are left to the readers. 
\end{proof}

We remove $m$ disks from the disk $\Sigma_0$ to obtain $\Sigma_0^{m+1} \subset \Sigma_0$. 
We obtain the surface $\Sigma_{2g+n-1}^{2m}$ by attaching two $\Sigma_0^{m+1}$'s to $\Sigma_{2g+n-1}^2$: 
\[
\Sigma_{2g+n-1}^{2m} = \Sigma_{2g+n-1}^{2} \cup_{\partial \Sigma_{2g+n-1}^{2} = \partial \Sigma_0 \amalg \partial \Sigma_0 } (\Sigma_0^{m+1} \amalg \Sigma_0^{m+1}). 
\]
Combining the equalities in Propositions \ref{prop_relation_surface} and \ref{prop_relation_disk}, we obtain the following equality in $\Mod(\Sigma_{2g+n-1}^{2m}; U^\prime)$: 
\[
\eta_{n,g} \eta_{n,g} \Phi_K(\eta_{n,g}) \Phi_K(\eta_{n,g}) = t_{\delta_1}^4\cdots t_{\delta_{2m}}^4 \widetilde{\tau_1}^{-4} \cdots \widetilde{\tau_m}^{-4}.  
\]
Eventually, for arbitrarily large $m$, we can find $m$ disjoint bisections in the Lefschetz fibration $f_{n,K}:E(n)_K\rightarrow S^2$ each of which has self-intersection $0$. 
Furthermore, each of the bisections has $4$ branched points. 
Thus, all the bisections are tori.

\begin{remark}
It is in fact possible to generalize these examples to cover knot surgered elliptic surfaces which are not symplectic, when the knots used in the construction are not fibered. In this case, following the arguments in \cite{Baykur}, we instead obtain a broken Lefschetz fibration on each knot surgered $4$-manifold, where Seiberg-Witten basic classes still appear as a collection of torus bisections. 
\end{remark}

\vspace{0.2in}
\noindent \textit{Acknowledgements.} The first author was partially supported by the NSF grant DMS-0906912 and the ERC Grant LDTBud. The second author was partially supported by JSPS Research Fellowships for Young Scientists~(24$\cdot$993) and JSPS KAKENHI (26800027). We also would like to thank Ersin \c{C}elik for running the code for our calculations to detect the smallest genera exotic pencils.


\vspace{0.2in}
\begin{thebibliography}{99}

\bibitem{Akbulut_Kirby_1980} S.~Akbulut and R.~Kirby, \emph{Branched Covers of Surfaces in $4$-manifolds}, Math. Ann., \textbf{252}(1980), 111--131

\bibitem{AkbulutOzbagci} S. Akbulut and B. Ozbagci, {\it Lefschetz fibrations on compact Stein surfaces}. Geom. Topol. 5 (2001), 319--334. 


\bibitem{AK} D. Auroux and L. Katzarkov, {\it A degree doubling formula for braid monodromies and Lefschetz pencils,} Pure Appl. Math. Q. 4 (2008), no. 2, part 1, 237--318.

\bibitem{Bauer} S. Bauer, {\it Almost complex $4$-manifolds with vanishing first Chern class,} J. Differential Geom. 79 (2008), no. 1, 25--32. 

\bibitem{Baykur} R. I. Baykur, {\it Topology of broken Lefschetz fibrations and near-symplectic four-manifolds,} Pacific J. Math. 240 (2009), no. 2, 201--230. 

\bibitem{BaykurDecomp} R. I. Baykur, {\it Minimality and fiber sum decompositions of Lefschetz fibrations,} preprint; http://arxiv.org/abs/1407.5351.

\bibitem{BaykurLuttingerLF} R. I. Baykur, {\it Inequivalent Lefschetz fibrations and surgery equivalence of symplectic \linebreak $4$-manifolds,} preprint; http://arxiv.org/abs/1408.4869.

\bibitem{BH2} R. I. Baykur and K. Hayano, {\it Hurwitz equivalence for Lefschetz fibrations and their multisections,} preprint in preparation. 

\bibitem{BG} P. Bellingeri and S. Gervais, \emph{``Surface framed braids'',}  Geom. Dedicata 159 (2012), 51--69. 



\bibitem{Donaldson} S. K. Donaldson, {\it Lefschetz pencils on symplectic manifolds}, J. Differential Geom. {\bf 53}(1999), no.2, 205--236

\bibitem{Donaldson_Smith_2003} S. K. Donaldson and I. Smith, {\it Lefschetz pencils and the canonical class for symplectic four-manifolds}, Topology \textbf{42}(2003), no. 4, 743--785

\bibitem{Dorfmeister} J. Dorfmeister, \emph{Kodaira dimension of fiber sums along spheres}, to appear in Geom. Dedicata, available at \verb|http://link.springer.com/article/10.1007/s10711-014-9974-2|


\bibitem{Endo_2000} H. Endo, {\it Meyer's signature cocycle and hyperelliptic fibrations}, Math. Ann., \textbf{316}(2000), 237--257

\bibitem{Endo_Gurtas_2010} H. Endo and Y. Z. Gurtas, {\it Lantern relations and rational blowdowns}, Proc. Amer. Math. Soc., \textbf{138}(2010), no. 3, 1131--1142

\bibitem{Farb_Margalit_2011} B. Farb and D. Margalit, {\it A Primer on Mapping Class Groups}, Princeton University Press, 2011

\bibitem{F} S.~Finashin, {\em Knotting of algebraic curves in $\CP$,} Topology 41 (2002), 47--55.



\bibitem{FSrationalblowdown}  R. Fintushel and R. Stern, {\it Rational blowdowns of smooth $4$-manifolds,} J. Differential Geom. 46 (1997), no. 2, 181--235.

\bibitem{FSsurfaces} R.~Fintushel, R.~J.~Stern,  {\em Surfaces in $4$-manifolds,} Math. Res. Lett. 4 (1997) 907--914.

\bibitem{Fintushel_Stern_1998} R. Fintushel and R. Stern, {\it Knots, links, and $4$-manifolds}, Invent. Math. {\bf 134}(1998), no. 2, 363--400

\bibitem{FSKnotSurgeryLF} R. Fintushel and R. Stern, {\it Families of simply connected $4$-manifolds with the same Seiberg-Witten invariants,} Topology 43 (2004), no. 6, 1449--1467. 

\bibitem{Fintushel_stern_2004} R. Fintushel and R. Stern, {\it Families of simply connected $4$-manifolds with the same Seiberg-Witten invariants}, Topology {\bf 43}(2004), 1449--1467

\bibitem{FS-survey} R. Fintushel and R. Stern, {\it Six lectures on four 4-manifolds.} Low dimensional topology, 265--315, IAS/Park City Math. Ser., 15, Amer. Math. Soc., Providence, RI, 2009. 

\bibitem{FriedlVidussi} S. Friedl and S. Vidussi, {\it On the topology of symplectic Calabi-Yau $4$-manifolds,} J. Topol. 6 (2013), no. 4, 945--954.


\bibitem{GayMark} D. Gay and T. E. Mark, {\it Convex plumbings and Lefschetz fibrations,}
J. Symplectic Geom. 11 (2013), no. 3, 363--375. 

\bibitem{Gompf} R. E. Gompf, {\it A new construction of symplectic manifolds,} Ann. of Math. (2) 142 (1995), no. 3, 527--595. 


\bibitem{GS} R. E. Gompf and A.I.Stipsicz, {\it 4-Manifolds and Kirby Calculus}, Graduate Studies in Mathematics \textbf{20}, American Mathematical Society, 1999.

\bibitem{HambletonKreck_1988} I.~Hambleton and M.~Kreck, \emph{Smooth structures on algebraic surfaces with cyclic fundamental group}, Invent. Math. 91(1988), no. 1, 53--59. 

\bibitem{HambletonKreck} I. Hambleton and M. Kreck, {\it Cancellation, elliptic surfaces and the topology of certain four-manifolds,} J. Reine Angew. Math. 444 (1993), 79--100. 

\bibitem{Kas_1980} A. Kas, {\it On the handlebody decomposition associated to a Lefschetz fibration}, Pacific J. Math. \textbf{89}(1980), 89--104.


\bibitem{K} H.-J.~Kim, {\em Modifying surfaces in 4-manifolds by twist spinning,} Geom. Topol. 10 (2006), 27--56 (electronic).

\bibitem{KR}  H.-J.~Kim,  D.~Ruberman, {\em Topological triviality of smoothly knotted surfaces in $4$-manifolds,} to appear, Trans. Amer. Math. Soc.

\bibitem{Korkmaz_2001} M.~Korkmaz, {\it Noncomplex smooth $4$-manifolds with Lefschetz fibrations}, Internat. Math. Res. Notices, \textbf{2001}, no. 3, 115--128

\bibitem{Korkmaz_Ozbagci_2008}M. Korkmaz and B. Ozbagci, {\it On sections of elliptic fibrations}, Michigan Math. J., \textbf{56}(2008), 77--87.

\bibitem{LiLiu} T.\,J. Li and A. Liu, {\it Symplectic structures on ruled surfaces and a generalized adjunction
formula,} Math. Res. Letters 2 (1995), 453--471.

\bibitem{Li_1999} T.\,J. Li, {\it Smoothly embedded spheres in symplectic $4$-manifolds}, Proc. Amer. Math. Soc., \textbf{127}(1999), no 2, 609--613

\bibitem{Li2} T.-J. Li, {\it Quaternionic bundles and Betti numbers of symplectic 4-manifolds with Kodaira dimension zero,} Int. Math. Res. Not. 2006, Art. ID 37385, 28 pp.

\bibitem{Li3} T.-J. Li, {\it The Kodaira dimension of symplectic $4$-manifolds,} Floer homology, gauge theory, and low-dimensional topology, 249--261, Clay Math. Proc., 5, Amer. Math. Soc., Providence, RI, 2006. 

\bibitem{Li1} T.-J. Li, {\it Symplectic $4$-manifolds with Kodaira dimension zero,} J. Differential Geom. 74 (2006), no. 2, 321--352.

\bibitem{LP} A. Loi and R. Piergallini, {\it Compact Stein surfaces with boundary as branched covers of $B^4$,} Invent. Math. 143: 325--348, 2001.

\bibitem{MOS} G. Massuyeau, A. Oancea and D.\,A. Salamon, {\it Lefschetz fibrations, intersection numbers, and representations of the framed braid group,}  Bull. Math. Soc. Sci. Math. Roumanie (N.S.) 56(104) (2013), no. 4, 435--486.


\bibitem{Matsumoto3} Y. Matsumoto, {\it Lefschetz fibrations of genus two - a topological approach} -, Proceedings of the 37th Taniguchi Symposium on Topology and Teichm\"{u}ller Spaces, (S. Kojima, et. al., eds.), World Scientific, 1996, 123--148.

\bibitem{Meyer} W.~Meyer, {\it Die {S}ignatur von Fl\"achenb\"undeln}, Math. Ann., \textbf{201}(1973), 239--264.

\bibitem{MS} D. McDuff and M. Symington, {\it Associativity properties of the symplectic sum,} Math. Res. Lett. 3 (1996), no. 5, 591--608.
 


\bibitem{ParkYun} J. Park and K.-H. Yun, {\it Nonisomorphic Lefschetz fibrations on knot surgery $4$-manifolds,} Math. Ann. 345 (2009), no. 3, 581--597.
 

\bibitem{Sato_2008} Y. Sato, {\it $2$-spheres of square $-1$ and the geography of genus-$2$ Lefschetz fibrations}, J. Math. Sci. Univ. Tokyo, \textbf{15}(2008), 461--491.

\bibitem{Sato_2013} Y. Sato, {\it Canonical classes and the geography of nonminimal Lefschetz fibrations over $S^2$}, Pacific J. Math., \textbf{262}(2013), no. 1, 191--226.

\bibitem{Scorpan} A. Scorpan, {\it The wild world of $4$-manifolds,} American Mathematical Society, Providence, RI, 2005. xx+609 pp. 


\bibitem{SiebertTian} B. Siebert and G. Tian, {\it On the holomorphicity of genus two Lefschetz fibrations,} Ann. of Math. (2) 161 (2005), no. 2, 959--1020. 

\bibitem{Smith} I. Smith, {\it Geometric monodromy and the hyperbolic disc,} Q. J. Math. 52 (2001), no. 2, 217--228. 

\bibitem{Smith2} I. Smith, {\it Lefschetz pencils and divisors in moduli space,} Geom. Topol. 5 (2001), 579--608.


\bibitem{Stipsicz} A. I. Stipsicz, {\it Indecomposability of certain Lefschetz fibrations}, Proc. Amer. Math. Soc. \textbf{129}(2001), 1499--1502

\bibitem{Sunukjian} N. Sunukjian, {\it Surfaces in $4$-manifolds: Concordance, Isotopy, and Surgery}, to appear in Int. Math. Res. Not.; http://arxiv.org/abs/1305.6542. 

\bibitem{T} C. H. Taubes, {\it The Seiberg-Witten invariants and symplectic forms,} Math. Res. Letters {\bf 1}: 809--
822, 1994.

\bibitem{T2} C. H. Taubes, {\it SW $\Rightarrow$ Gr: From the Seiberg-Witten equations to pseudo-holomorphic curves}, Journal of the AMS, 9 (1996), 845--918.

\bibitem{Usher} M. Usher, {\it Minimality and symplectic sums,} Int. Math. Res. Not. 2006, Art. ID49857, 1--17. 
\end{thebibliography}
\end{document}